\documentclass{amsart}

\usepackage{mathrsfs,comment,amsmath,amsfonts,amssymb}
\usepackage[runin]{abstract}
\abslabeldelim{.}
\usepackage[usenames,dvipsnames]{color}
\usepackage[normalem]{ulem}
\usepackage{url}
\usepackage[all,arc,2cell]{xy}
\UseAllTwocells
\usepackage{enumerate}
\usepackage{hyperref}  
\hypersetup{%
  bookmarksnumbered=true,%
  colorlinks=true,%
  linkcolor=blue,%
  citecolor=blue,%
  filecolor=blue,%
  menucolor=blue,%
  urlcolor=blue,%
  pdfnewwindow=true,%
  pdfstartview=FitBH}   

\def\makeautorefname#1#2{\expandafter\def\csname#1autorefname\endcsname{#2}}
\newcounter{mycounter}
\makeautorefname{footnote}{footnote}%
\makeautorefname{item}{item}%
\makeautorefname{figure}{Figure}%
\makeautorefname{table}{Table}%
\makeautorefname{part}{Part}%
\makeautorefname{appendix}{Appendix}%
\makeautorefname{chapter}{Chapter}%
\makeautorefname{section}{Section}%
\makeautorefname{subsection}{Section}%
\makeautorefname{subsubsection}{Section}%
\makeautorefname{theorem}{Theorem}%
\makeautorefname{corollary}{Corollary}%
\makeautorefname{lemma}{Lemma}%
\makeautorefname{proposition}{Proposition}%
\makeautorefname{property}{Property}
\makeautorefname{conjecture}{Conjecture}%
\makeautorefname{definition}{Definition}%
\makeautorefname{notation}{Notation}
\makeautorefname{remark}{Remark}%
\makeautorefname{question}{Question}%
\makeautorefname{example}{Example}%
\makeautorefname{examples}{Examples}
\makeautorefname{exercise}{Exercise}
\makeautorefname{axiom}{Axiom}%
\makeautorefname{claim}{Claim}%
\makeautorefname{assumption}{Assumption}%
\makeautorefname{assumptions}{Assumptions}%
\makeautorefname{construction}{Construction}%
\makeautorefname{problem}{Problem}%
\makeautorefname{warning}{Warning}%
\makeautorefname{observation}{Observation}%
\makeautorefname{convention}{Convention}%
\makeautorefname{question}{Question}%

\theoremstyle{plain}
\newtheorem{theorem}{Theorem}[section]
\newtheorem{corollary}{Corollary}[section]
\newtheorem{proposition}{Proposition}[section]

\newtheorem*{conjecture*}{Conjecture}
\newtheorem*{claim*}{Claim}
\newtheorem{lemma}{Lemma}[section]

\newtheorem*{assumption*}{Assumption}

\newtheorem*{assumptions*}{Assumptions}

\theoremstyle{definition}
\newtheorem{definition}{Definition}[section]

\newtheorem{example}{Example}[section]

\newtheorem{notation}{Notation}[section]
\newtheorem*{notation*}{Notation}

\newtheorem{question}{Question}[section]
\newtheorem*{question*}{Question}
\newtheorem*{doubts*}{Doubts}
\newtheorem{remark}{Remark}[section]

\newtheorem*{warning*}{Warning}

\newtheorem{convention}{Convention}[section]
\newtheorem*{exercise*}{Exercise}

\newtheorem*{fact*}{Fact}

\makeatletter
\let\c@theorem=\c@theorem
\let\c@observation=\c@theorem
\let\c@corollary=\c@theorem
\let\c@proposition=\c@theorem
\let\c@lemma=\c@theorem
\let\c@problem=\c@theorem
\let\c@construction=\c@theorem
\let\c@conjecture=\c@theorem
\let\c@definition=\c@theorem
\let\c@notation=\c@theorem
\let\c@example=\c@theorem
\let\c@examples=\c@theorem
\let\c@axiom=\c@theorem
\let\c@property=\c@theorem
\let\c@assumption=\c@theorem
\let\c@convention=\c@theorem
\let\c@assumptions=\c@theorem
\let\c@warning=\c@theorem
\let\c@remark=\c@theorem
\let\c@sch=\c@theorem
\let\c@question=\c@theorem
\let\c@exercise=\c@theorem
\numberwithin{equation}{section}
\let\c@equation=\c@theorem

\makeatother

\usepackage{tikz-cd}
\usepackage{makecell}
\usepackage{tikz, pgfplots}
\usepackage{amsmath,amsfonts,amsthm,amssymb,nameref,mathtools,mathbbol,mathrsfs,amscd}
\pgfplotsset{compat=1.18}
\usetikzlibrary{shapes,snakes}
\usepackage[margin=1in]{geometry}
\usepackage[utf8]{inputenc}

\usepackage{cite}
\usepackage{graphicx}
\DeclareFontFamily{OT1}{pzc}{}
\DeclareFontShape{OT1}{pzc}{m}{it}{<-> s * [1.10] pzcmi7t}{}
\DeclareMathAlphabet{\mathpzc}{OT1}{pzc}{m}{it}

\def\d#1{\mathrm{d}^{\pi}_{#1}}
\def\abs#1{\left |{#1}\right |}
\def\diam{\mathrm{diam}}
\def\gp#1#2#3{({#1}|{#2})_{#3}}
\def\ksi{\xi}
\def\MCG{\mathrm{Mod}}
\def\Teich{\mathpzc{Teich}}
\def\Kappa{\kappa}
\def\horo#1{\partial_{h}{#1}} 
\def\act{\curvearrowright} 

\def\hor#1{\left [ \partial_{h}{#1}\right ]} 
\def\G{\Gamma}

\title[MLS rigidity in groups with contracting elements]{Marked Length Spectrum Rigidity in Groups with Contracting Elements}

\author{Renxing Wan}
\address{School of Mathematical Sciences,  Key Laboratory of MEA (Ministry of Education) \& Shanghai Key Laboratory of PMMP,  East China Normal University, Shanghai 200241, China P. R.}
\email{rxwan@math.ecnu.edu.cn}

\author{Xiaoyu Xu}
\address{Beijing International Center for Mathematical Research\\
Peking University\\
 Beijing 100871, China P.R. }
\email{xuxiaoyu@stu.pku.edu.cn}

\author{Wenyuan Yang}
\address{Beijing International Center for Mathematical Research\\
Peking University\\
 Beijing 100871, China P.R.}
\email{wyang@math.pku.edu.cn}


\begin{document}
\begin{sloppypar}
	\maketitle

\begin{abstract}
    This paper presents a study of the well-known marked length spectrum rigidity problem in the coarse-geometric setting. 
    For any two (possibly non-proper) group actions $G\curvearrowright X_1$ and $G\curvearrowright X_2$ with contracting property, we prove that if the two actions have the same marked length spectrum, then the orbit map $Go_1\to Go_2$ must be a rough isometry. In the special case of cusp-uniform actions, the rough isometry can be extended to the entire space. This generalises the existing results in hyperbolic groups and relatively hyperbolic groups. In addition, we prove a finer marked length spectrum rigidity from confined subgroups and further, geometrically dense subgroups. Our proof is based on the Extension Lemma and uses purely elementary metric geometry. This study produces new results and recovers existing ones for many more interesting groups through a unified and elementary approach.
\end{abstract}

 \setcounter{tocdepth}{1}
 \tableofcontents
\section{Introduction}

\subsection{Background and motivations}

\text{ }

Let $(M,\mathfrak g)$ be a closed Riemannian manifold. If $\mathfrak g$ has negative sectional  {curvatures}, then each free homotopy class $c$ has a unique closed geodesic representative $\gamma_c$. Denote by $\mathcal C$ the set of all free homotopy classes, or equivalently the set of conjugacy classes in the fundamental group $\pi_1(M)$. The function 
$$\begin{aligned}
\ell_{\mathfrak g}:\quad & \mathcal C\to \mathbb R_{\ge 0}\\
&c\mapsto \mathrm{Len}(\gamma_c)
\end{aligned}$$ is called the \textit{marked length spectrum} (MLS) of $(M, \mathfrak g)$, where $\mathrm{Len}$ denotes the length of a curve. The definition extends naturally to manifolds of nonpositive curvature, in which case the closed geodesic might not be unique, but its length is uniquely determined by the free homotopy class.

The well-known MLS rigidity conjecture \cite{BK85} states that if two negatively curved Riemannian metrics on a closed manifold have the same MLS, then they are isometric. This is known to be
true for surfaces by the result of Otal \cite{Ota90} and Croke \cite{Cro90} independently. When one of the Riemannian manifolds is rank one locally symmetric, Hamenst\"{a}dt \cite{Ham99} proved the conjecture using the minimal entropy rigidity theorem of Besson--Courtois--Gallot \cite{BCG95}. More recently, Guillarmou and Lefeuvre \cite{GL19} showed that the conjecture holds if the two metrics are close enough in a suitable $C^k$ topology. However, the conjecture remains open in general.

In \cite{Fur02}, Furman started to consider the MLS rigidity conjecture in a coarse-geometric setting. Namely, consider a left-invariant pseudo-metric $d$ on a group $G$, which is usually assumed to be quasi-isometric to a word metric. For every $g\in G$, the \textit{stable translation length} of $g$ is defined as $$\ell_d(g)=\lim_{n\to +\infty}\frac{d(1,g^n)}{n}.$$ We call the function $\ell_d: G\to [0,+\infty)$ given by $g\mapsto \ell_d(g)$ the \textit{MLS} for this metric $d$. In particular, for a compact negatively curved Riemannian manifold $(M,\mathfrak g)$ with a universal cover $(\widetilde M, \widetilde{\mathfrak{g}})$,  he considered the pull-back metric on $\pi_1(M)$ from the action $\pi_1({M})\curvearrowright \widetilde{M}$ defined as: $$d_{\mathfrak{g},x}(\gamma_1,\gamma_2):=d_{\widetilde{\mathfrak{g}}}(\gamma_1x,\gamma_2x),\,\forall \gamma_1,\gamma_2\in \pi_1(M),$$ where $x$ is a fixed basepoint in $\widetilde{M}$. The geometric length of $\gamma_c$ could be recovered as $\ell_{d_{\mathfrak{g},x}}(g)$, independent of $x$, where $g\in \pi_1({M})$ is represented by $c$. %

In the coarse-geometric setting, Furman has described the MLS rigidity problem in the sense of large scale geometry. 
\begin{question}\label{Que: pull-back metric}
    Let $\mathcal D_{G}$ denote a preferred collection of left-invariant, proper, pseudo-metrics on a group $G$, and two pseudo-metrics $d,d_{\ast}\in \mathcal D_{G}$ have the same marked length spectrum: $\ell_{d}(g)=\ell_{d_{\ast}}(g)$ for all $g\in G$. Is it true that $d$ and $d_{\ast}$ are roughly isometric, namely $\abs{d-d_{\ast}}\leq C$ for some $C\geq 0 $?
\end{question}

An important sub-class of $\mathcal D_{G}$ is given by pulling back the geodesic metric via the isometric actions of $G$ on a family of geodesic metric spaces. In this case, for an isometric action $G\curvearrowright (X,d)$, the stable translation length of an element $g\in G$ is defined as $$\ell_d(g)=\lim_{n\to+\infty}\frac{d(o,g^no)}{n},$$ and is independent with the choice of the basepoint $o\in X$. The function $\ell_d:G\to [0,+\infty)$ is called the \textit{marked length spectrum} for this action.

\begin{definition}\cite{BS00}
    Let $(X, d)$ be a metric space and let $\delta \ge 0$. We say that the metric $d$ is \textit{$\delta$-roughly geodesic} if for every $x, y \in X$, there is a $(1, \delta)$-quasi-geodesic from $x$ to $y$. A metric is said \textit{roughly geodesic} if there is $\delta \ge 0$ such that it is $\delta$-roughly geodesic.
\end{definition}
\begin{remark}\label{IntroRmk: RoughGeodesic}
    \begin{enumerate}
        \item If $G$ acts coboundedly on a geodesic metric space, then the pull back metric is a roughly geodesic, left-invariant pseudo-metric on $G$ which is quasi-isometric to a (possibly non-proper) word metric. 
        \item\label{NotRouGeo} As pointed out by Nguyen--Wang in \cite{NW22} that any word metric $d$ on $G$ has the same MLS as $d_*=d+\sqrt{d}$ which is (2,1)-quasi-isometric to $d$. However, $d$ and $d_*$ are not roughly isometric when $d$ is unbounded. One reason is that $d_*$ is not a roughly geodesic metric. 
    \end{enumerate}
\end{remark}
Motivated by the above remark, we consider the following question in this paper.

\begin{question}\label{Que: MLSR}
Let $G$ be a group acting isometrically on two geodesic metric spaces $(X_1,d_1)$ and $(X_2,d_2)$ with basepoints $o_1\in X_1, o_2\in X_2$. Suppose that the two actions have the same MLS. Is it true that $Go_1$ and $Go_2$ must be roughly isometric, namely there exist an orbit map $\rho: Go_1\to Go_2$ and a constant $C \ge 0$ such that $\rho$ is a $(1, C)$-quasi-isometry? 
Moreover, does the rough isometry extend to the whole spaces (provided that the actions are not cobounded)?
\end{question}

%

Furman \cite{Fur02} proved that \autoref{Que: pull-back metric} holds for Gromov hyperbolic groups, where $\mathcal D_{G}$ denotes the collection of left-invariant, Gromov hyperbolic pseudo-metrics on $G$ which are quasi-isometric to a word metric. Later on, Fujiwara \cite{Fuj13} gave an affirmative answer to toral relatively hyperbolic groups  {where $\mathcal D_{G}$ denotes a collection of roughly geodesic, left-invariant  pseudo-metrics on $G$ which are quasi-isometric to  word metrics.} Recently, this  was  generalised by Nguyen--Wang \cite{NW22} to all relatively hyperbolic groups {with respect to any peripheral subgroups}.

The first main goal of this paper is to consider the coarse version of the MLS rigidity for group actions with contracting property described below.

\subsection{MLS rigidity for actions with contracting property}

The first result of this paper gives a large generalisation of the above results, including \cite[Theorem 2]{Fur02} and \cite[Lemma 3.6]{CR22} for hyperbolic groups and \cite[Prop. 3.12]{NW22} for relatively hyperbolic groups.

\begin{theorem}[\autoref{THM: Main Rigidity}]\label{IntroThm: Main rigidity}
Suppose that a group $G$ acts isometrically on two geodesic metric spaces $(X_1, d_1)$ and $(X_2, d_2)$ with contracting property (cf. \autoref{DEF: Contracting Property}), respectively. Then $G$ has MLS rigidity: if $\ell_{d_1}(g)=\ell_{d_2}(g)$ for every $g\in G$, then the orbit map $\rho: Go_1\to Go_2$ is well-defined up to a bounded error and is a rough isometry, for any fixed basepoints $o_1\in X_1,o_2\in X_2$.
\end{theorem}

{Actually, we prove \autoref{IntroThm: Main rigidity} in a more precise way.
\begin{theorem}[\autoref{COR: Rigidity from SC}]
    Suppose that a group $G$ acts isometrically on two geodesic metric spaces $(X_1, d_1)$ and $(X_2, d_2)$ with contracting property, respectively. Assume for some $\lambda_1,\lambda_2\geq 0$, the following inequality holds: \begin{align*}
        \lambda_1  \ell_{d_2}(g)\leq  \ell_{d_1}(g)\leq \lambda_2 \ell_{d_2}(g),\,\forall g\in G.
    \end{align*}  
    
    Then, for any fixed points $o_1\in X_1$ and $o_2\in X_2$, there exists $C\geq 0$ such that $$\lambda_1 d_2(go_2,ho_2)-C\leq d_1(go_1,ho_1)\leq \lambda_2 d_2(go_2,ho_2)+C, \,\forall g,h\in G.$$
\end{theorem}
}
Roughly speaking, a group action $G\curvearrowright X$ has \textit{contracting property} if there exist two weakly independent contracting elements in $G$. The contracting property captures the key feature of quasi-geodesics in Gromov hyperbolic spaces, rank-1 geodesics in CAT(0) spaces,  thick geodesics in Teichm\"{u}ller spaces, and so on. In recent years, this notion and its variants have been proven fruitful in the setup of general metric spaces. 

Let $A$ be a subset of $X$, and $\pi_A: X \rightarrow A$ be the (coarsely) closest projection (set-valued) map. We say that $A$ is \textit{$C$-contracting} for $C \geq 0$ if $\diam(\pi_A(\gamma))\le C$ for any geodesic (segment) $\gamma$ with $ \gamma\cap N_C(A)=\varnothing$. In fact, the contracting notion is equivalent to the usual one: $\diam(\pi_A(B)) \leq C'$ for any metric ball $B$ disjoint with $A$. A proof given in \cite[Corollary 3.4]{BF09} for CAT(0) spaces is valid in the general case.  An element $g\in G$ is called \textit{contracting}, if the orbit of $\langle g\rangle$ is a contracting quasi-geodesic.  

The prototype of a contracting element is a loxodromic isometry on a Gromov hyperbolic space, but more interesting examples are furnished by the following (cf. \autoref{EX: Contracting Subset} and \autoref{EX: Contracting Property} for details):
\begin{itemize}
    \item hyperbolic elements in relatively hyperbolic groups, cf. \cite{Ger15,GP13};
    \item rank-1 elements in CAT(0) groups, cf. \cite{Bal12,BF09};
    \item hyperbolic elements in groups with non-trivial Floyd boundary, cf. \cite{Yang14};
    \item certain infinite order elements in graphical small cancellation groups, cf. \cite{ACGH16};
    \item pseudo-Anosov elements in mapping class groups, cf. \cite{Min96}.
\end{itemize}


For more results about groups with contracting elements, see \cite{Yang14,Yang19,GY22,Yang22}. We remind the readers that people always study a group with contracting elements under the additional assumption of a proper action. However, some of our results in this paper can directly apply to non-proper actions. For example, \autoref{IntroThm: Main rigidity} applies to any non-elementary acylindrical (or even WPD) actions with contracting elements, cf.\cite{Osi16,DGO17,Bow08}.

In particular, when $G$ is a relatively hyperbolic group, we can extend the rough isometry from the $G$-orbit to the entire space through a cusp-extension technique, even without a cocompact assumption. This also generalises the questions considered in \cite{Bur20} and \cite{HH20}. Recall that a group action on a Gromov hyperbolic space $G\curvearrowright X$ is called \textit{non-elementary} if $\sharp \Lambda G>2$.

\begin{theorem}[\autoref{THM: Cusp Uniform}]\label{IntroThm: Cusp uniform actions}
Suppose that 
$(G,\left\{H_i\right\})$ admits two non-elementary cusp-uniform actions on two proper $\delta$-hyperbolic geodesic spaces $(X_1,d_1)$, $(X_2,d_2)$. If the two actions have the same MLS, then there exists a $G$-coarsely equivariant rough isometry between $X_1$ and $X_2$.
\end{theorem}

Our method also allows to deal with the MLS rigidity for certain interesting non-geodesic metrics as in \autoref{Que: pull-back metric},  which are not pulled back  from the action on a geodesic metric space. {Compared with \autoref{IntroRmk: RoughGeodesic} (\ref{NotRouGeo}), we point out that this}  class of metrics is coarsely additive along transversal points defined in a reference action on a geodesic metric space with contracting property (see \autoref{CoarseAddtiveDef}). Green metrics associated to random walks introduced in \cite{BB07} are such  examples recently studied in \cite{GGPY,DWY23}. Let us state the  following consequence  of our main result \autoref{MLSnonGeodesic} in this regard.
\begin{corollary}[\autoref{MLSGreenMetrics}]\label{IntroCor: Green Metric}
Suppose that $G$ is a non-elementary relatively hyperbolic group equipped with two probability measures $\mu_1,\mu_2$ of finite symmetric support generating $G$. Then $G$ has MLS rigidity for the pair $(d_{\mu_1}, d_{\mu_2})$ of the corresponding Green metrics.   
\end{corollary}

We also exhibit an example of relatively hyperbolic group with two (non-geodesic) proper left invariant metrics without MLS rigidity for which, however, a maximal parabolic group has the same MLS and attains the maximal growth. See \autoref{MLSFailsMaximalSubgroup}. 

\subsection{MLS rigidity from subgroups}

As a second goal of this paper, we consider a further question following Cantrell and Reyes \cite{CR23}. 

\begin{question}\label{Que: MLSR from subset}
    Let $G$ be a group acting isometrically on two metric spaces $(X_1,d_1)$ and $(X_2,d_2)$. For what kind of subsets $H\subseteq G$, the same MLS on $H$ will imply the same MLS on $G$?
\end{question}

Recently, Cantrell--Reyes \cite{CR23} considered this question for a non-elementary hyperbolic group $G$, and gave an affirmative answer to the subsets with the same growth rate as the ambient group. Beyond that, Cantrell--Reyes--Tanaka developed a series of work on the space of metric structures on $G$, see \cite{CT21,CR22,Rey23} for more details.


For a group $G$ with contracting property, the following result shows that a normal subgroup with unbounded orbits would meet the requirements. This generalises the result in \cite{GH20} which considers homologically trivial elements in the fundamental group of a negatively curved Riemannian manifold. {If $G$ is hyperbolic, an independent  proof using geodesic current has appeared in  the work of Bonahon \cite[Page 161]{Bon91}.}

\begin{theorem}[\autoref{THM: Normal Subgroup Rigidity}]\label{IntroThm: MLSR from normal subgroups}
Suppose that a group $G$ acts, respectively, on $(X_1, d_1)$ and $(X_2, d_2)$ with contracting property. Let $H\lhd G$ be a normal subgroup which has an unbounded orbit in either $X_1$ or $X_2$. Then, $G$ has MLS rigidity from $H$: if $\ell_{d_1}(g)=\ell_{d_2}(g)$ for any $g\in H$, then there exists a rough isometry between the $G$-orbits, and equivalently $\ell_{d_1}(g)=\ell_{d_2}(g)$ for all $g\in G$.
\end{theorem}

We note that an infinite normal subgroup may not have maximal growth in a hyperbolic group. It is known that it obtains the maximal growth if and only if the quotient group is amenable (see \cite{CDS18} and references therein). Therefore, the above result could not be derived from Cantrell--Reyes' work \cite{CR23}.

\begin{corollary}[\autoref{COR: Sequence of Normal Subgroup}]\label{IntroCor: MLSR from subnormal subgroups}
Suppose that a group $G$ acts on $(X_1, d_1)$ and $(X_2, d_2)$ with contracting property. Let $H_n\lhd H_{n-1}\lhd\cdots\lhd H_1\lhd H_0=G$ be a subnormal sequence of subgroups with unbounded orbits in either $X_1$ or $X_2$. Then, $G$ has MLS rigidity from $H_n$.
\end{corollary}

When the actions are assumed to be proper, we can improve \autoref{IntroThm: MLSR from normal subgroups} to a wider class of  confined subgroups, introduced in \cite{HZ97} in the study of ideals in group algebras. Confined subgroups can arise as uniformly recurrent subgroups \cite{GW} on the Chabauty space of subgroups, and are recently studied for various countable groups in \cite{LBMB}. We refer to these paper for further references. Geometrically, a confined subgroup $H$ in a geometric action of $G$ on a metric space $X$ is equivalent to the bounded injective radius of $X/H$.   Recent works \cite{GL23,CGYZ} emphasise the geometric aspects, which we follow in this paper.  

A subgroup $H$ of $G$ is called \textit{confined} if there exists a finite subset $P\subseteq G\setminus \{1\}$ such that $gHg^{-1}\cap P\neq \varnothing$ for every $g\in G$. We refer   $P$ as the \textit{confining} subset. 
If a group $G$   acts properly on a geodesic metric space with contracting property, there exists a unique maximal finite normal subgroup $E(G)\lhd G$ called the \textit{elliptic radical} of $G$, which fixes pointwise the limit points of all contracting elements. In this case, an infinite normal subgroup of $G$ is  confined  with a confining subset disjoint with $E(G)$.

\begin{theorem}[\autoref{Thm: MLS from confined subgps}]\label{IntroThm: MLSR from confined subgroups}
   Suppose that a group $G$ acts   properly on $(X_1, d_1)$ and $(X_2, d_2)$ with contracting property, respectively. Let $H\le G$ be a confined subgroup with a confining subset disjoint with $E(G)$. Then, $G$ has MLS rigidity from $H$.
\end{theorem}
We remark that \autoref{IntroThm: MLSR from confined subgroups} is  new even in the class of hyperbolic groups. 

Following Osin \cite{Osi22}, a  subgroup $H$ of $G$ is called \textit{geometrically dense} if $[\Lambda_{X_i} H]=[\Lambda_{X_i} G]$ for $i=1,2$. Geometrically dense (non-normal) subgroups exist in abundance as random and stationary invariant subgroups on the Chabauty space of subgroups in \cite{Osi22} and \cite{GL23} respectively. This will be extended for groups with contracting property in a future work.

In particular, the answer to \autoref{Que: MLSR from subset} can be any geometrically dense subgroup satisfying a condition called \hyperref[Condition B]{Extension Condition}. %

\begin{theorem}[\autoref{THM: MLSR from Geometrically Dense Subgroup}]\label{IntroThm: MLSR from GD subgroups}
Suppose that a group $G$ acts properly  on  $(X_1, d_1)$ and $(X_2, d_2)$ with contracting property respectively. Let $H<G$ be a geometrically dense subgroup with contracting property and satisfying \hyperref[Condition B]{Extension Condition}. Then $G$ has MLS rigidity from $H$.%
\end{theorem}

This theorem also generalises the recent result by Hao in \cite[Theorem B]{Hao22}, where the author derived a similar result in the setting of closed negatively curved
Riemannian manifolds. We emphasise that \autoref{IntroThm: MLSR from confined subgroups} could not be derived immediately from \autoref{IntroThm: MLSR from GD subgroups}, even though confined subgroups are geometrically dense under appropriate conditions. Actually, \autoref{IntroThm: MLSR from confined subgroups} is proved by direct means without verifying  \hyperref[Condition B]{Extension Condition}. 

Roughly speaking, a geometrically dense subgroup $H\le G$ satisfies \hyperref[Condition B]{Extension Condition} if the orbit map $\rho_H: Ho_1\to Ho_2$ can be extended continuously to a $G$-equivariant map $\bar \rho: \Lambda_{X_1}(G)\to \Lambda_{X_2}(G)$. There are several situations (cf. \autoref{Lem: Conditions for extension property}) in which \hyperref[Condition B]{Extension Condition} is fulfilled.  For an example without \hyperref[Condition B]{Extension Condition}, see \autoref{Example without extension property}. 

In \cite{HH20}, Healy-Hruska recently proved that \hyperref[Condition B]{Extension Condition} holds under cusp-uniform actions on hyperbolic spaces with constant horospherical distortions.  Rank-1 symmetric spaces and cusped spaces by combinatorial horoballs have constant horospherical distortions (see \cite{HH20} for more details). Hence, together with \autoref{IntroThm: Cusp uniform actions},  MLS rigidity from geometrically dense subgroups holds in relatively hyperbolic groups.

\begin{corollary}[\autoref{COR: Cusp Uniform Geometrically Dense}]\label{IntroCor: GD of cusp uniform actions}
Suppose that  two cusp-uniform actions of $G$ on $(X_1,d_1)$, $(X_2,d_2)$ as in \autoref{IntroThm: Cusp uniform actions} have constant horospherical distortions. If $K<G$ is a geometrically dense subgroup and the two actions have the same MLS on $K$, then there exists a $G$-coarsely equivariant rough isometry $X_1\to X_2$.
\end{corollary}

For rank-1 symmetric spaces, a stronger conclusion actually holds. In \cite{Kim01}, Inkang Kim  proved that Zariski dense (possibly indiscrete) subgroups are determined by the same MLS up to conjugacy.

At the completion of writing this paper, we found the very recent paper of Yanlong Hao \cite{Hao23} where \autoref{IntroThm: Cusp uniform actions} was independently proved  as  \cite[Theorem A]{Hao23}.  Our \hyperref[Condition B]{Extension Condition} was also introduced independently for Gromov boundary of hyperbolic spaces, which is called \textit{comparable actions} there. In particular, \autoref{IntroThm: MLSR from GD subgroups} for the case that $X_i$ are proper and roughly complete hyperbolic spaces and the $G$-actions are dense at infinity follows from \cite[Theorem B]{Hao23}, which is a combination of our \autoref{IntroThm: MLSR from GD subgroups} and a $G$-coarsely equivariant version of the  extension theorem for bilipschitz boundary maps \cite[Theorem 7.1.2]{BS07}. However, our \autoref{IntroThm: MLSR from GD subgroups} could apply to  any metric spaces with horofunction boundary. 

\subsection{MLS rigidity from subsets} 
Following the path of Cantrell--Reyes \cite{CR23}, we  attempt to give some preliminary results  on MLS rigidity from `large' subsets in  \autoref{Que: MLSR from subset}. We  quantify the largeness using the \textit{growth rate} of a subset $E\subseteq G$ with respect to the proper action $G\curvearrowright (X_i,d_i)$ as follows: for given $i=1,2$,
\begin{equation}\label{Definition of growth rate of subset}
    \delta_i(E):=\limsup_{n\to\infty}\frac{\log \sharp\{g\in E: d_i(o_i,go_i)\le n\}}{n}
\end{equation}
and the \textit{growth rate of conjugacy classes} in $E$ using stable translation length
\begin{equation}\label{Definition of conjugacy growth rate of subset}
\delta_i^c(E):=\limsup_{n\to\infty}\frac{\log \sharp\{[g]: g\in E, \ell_i(g)\le n\}}{n}.
\end{equation}
If $G\act X_i$ is co-compact, then $\delta_i(E)<\infty$, but  $\delta_i^c(E)$ may still be infinite. If $X_1$ and $X_2$ (e.g. CAT(0) and hyperbolic spaces) are coarsely convex (cf. \autoref{Def: coarsely convex metric}), the stable translation length differs by a bounded amount from the algebraic length by \autoref{Lem: ConvexMetric}, so we see that $\delta_i^c(E)<\infty$. For this reason, we assume  $G\curvearrowright (X_i,d_i)$  are  proper and cocompact actions on coarsely convex spaces. 

As demonstrated in Cantrell--Reyes \cite{CR23}, it is crucial  to understand the Manhattan curves introduced by Burger \cite{Bur93}, which are the boundary of the regions  where  the following two series take finite values:
$$\mathcal P(a,b)=\sum_{g\in G}e^{-ad_2(o_2,go_2)-bd_1(o_1,go_1)}, \quad \mathcal Q(a,b)=\sum_{[g]\in \mathbf{conj}(G)}e^{-a\ell_{d_2}(g)-b\ell_{d_1}(g)},$$
where $\mathbf{conj}(G)$ denotes the set of all conjugacy classes in $G$.
Namely, for any given $a\in \mathbb R$, denote $\theta(a)=\inf\left\{b\in \mathbb{R}\mid \mathcal{P}(a,b)<+\infty\right\}$ and $\Theta(a)=\inf\left\{b\in \mathbb{R}\mid \mathcal{Q}(a,b)<+\infty\right\}$. Again, if $X_1$ and $X_2$ (e.g. CAT(0) and hyperbolic spaces) are coarsely convex, then $\Theta(a)\in (-\infty, +\infty)$. 

Given $i=1,2$, we say that a subset $E\subseteq G$ is  \textit{exponentially generic} in $d_i$ if $\delta_i(G\setminus E)<\delta_i(G)$.

\begin{theorem}[\autoref{Thm: MLSR From Conjugacy-large Subset}]\label{IntroThm: MLSR from large subsets}
    Suppose that a group $G$ acts properly and cocompactly on two  coarsely convex geodesic spaces $(X_1, d_1)$ and $(X_2, d_2)$ with contracting property, respectively. Assume  the curve $\{(a,\Theta(a))\mid a\in \mathbb R\}$  is {either globally strictly convex or a straight line} and that $\delta_1(G)=\delta_2(G)$. 
    Then{,} 
    \begin{itemize}
        \item[(1)] for any subset $E\subseteq G$ with $\delta_1^c(E)=\delta_1^c(G)$, $G$ has MLS rigidity from $E${;}
        \item[(2)] for any exponentially generic subset $E\subseteq G$ in $d_1$ or $d_2$, $G$ has MLS rigidity from $E$. 
    \end{itemize}
\end{theorem}

In practice, it is a hard problem to prove the {dichotomy, that is, global strict convexity or linearity} of the Manhattan curves. {Since the dichotomy can be derived from an analyticity of the curve, there are some examples which satisfy the conditions of \autoref{IntroThm: MLSR from large subsets}.} If $G$ is a hyperbolic group, Cantrell--Tanaka {\cite[Theorem 5.4]{CT24}} proved that the analyticity   for any pair of word metrics and of strongly hyperbolic metrics (including Green metrics). Recently, Cantrell--Reyes \cite[Theorem 1.2]{CR23b} showed the analyticity  for a pair of geometric actions with good properties on CAT(0) cube complexes endowed with weighted combinatorial metrics. This   includes cubulable non-elementary hyperbolic groups, some infinite families of right-angled Artin and Coxeter groups. Therefore, \autoref{IntroThm: MLSR from large subsets} applies to those classes of groups, for  which MLS rigidity holds for large subsets. 

\subsection{Key technique: Extension Lemma}

A large proportion of existing proofs for MLS rigidity demonstrate a strong relation with ergodic theory. The major techniques adapted by Furman in \cite{Fur02} include Patterson--Sullivan measure and geodesic flow, which influenced a large number of later works including \cite{Hao22,Rey23,CT21}. Another inspiration comes from earlier work considering MLS rigidity in symmetric spaces, for example,  \cite{DK02,Kim01,Bur93}. In \cite{Bur93}, Burger introduced the Manhattan curve to demonstrate the relation between two different metrics on a group. This technique is widely applied in later works including \cite{CT21,CR23,CR22}. In addition, in proof of MLS rigidity for hyperbolic groups, Cantrell--Reyes \cite{CR22} relied on Cannon’s automatic structure. On the other hand, those works considering relatively hyperbolic groups, including \cite{Fuj13,NW22}, are mainly based on metric geometry.

Our proof uses a unified method in different group actions and is purely based on elementary metric geometry, which captures the general contracting property in various group actions. The most important tool in proving all the results is the Extension Lemma \cite[Lemma 2.14]{Yang19}, which is a key lemma in the study of groups with contracting elements. For convenience, we state here a simplified result and refer the reader to \autoref{SEC: Extension Lemma} for details.

\begin{lemma}\label{IntroLem: Extension}
    Suppose that $G$ admits an isometric action on a geodesic metric space $(X,d)$ with contracting property. Let $o\in X$ be a basepoint. Then there exist $\epsilon_0>0$ and a set $F$ of three elements in $G$ such that for any two elements $g, h \in G$, there exists $f \in  F$ such that $g\cdot  f \cdot h$ is almost a geodesic:
    $$\abs{d(o,gfh\cdot o)-d(o,g\cdot o)-d(o,h\cdot o)}\leq \epsilon_0.$$
\end{lemma}

\begin{remark}
    This result is best illustrated for a free group $G=F_2$ with standard generating set $G=\left\langle a,b\right\rangle $. In this case, we can choose $F=\left\{a,b,a^{-1}\right\}$ and $\epsilon_0=1$.
\end{remark}  


One crucial step in understanding \autoref{Que: MLSR} and \autoref{Que: MLSR from subset} is to relate two group actions with contracting property through a set of elements called \textit{simultaneously contracting elements}: $\mathcal{SC}(G)=\{ g\in G\mid g\text{ is contracting on both }X_1\text{ and }X_2\}$. 

A notion named \textit{admissible paths} (cf. \autoref{DEF: Admissible Path}) is highly relevant in the application of \autoref{IntroLem: Extension}, which captures a special class of consecutive geodesic segments in a metric space. For instance, the path labelled by $(g,f,h)$ in \autoref{IntroLem: Extension} is an example of admissible path. One important property of admissible paths is that `long' periodic admissible paths are contracting \cite[Proposition 2.9]{GY22}. This enables us to construct many more contracting elements from given ones in use of \autoref{IntroLem: Extension}, and in particular, we show that $\mathcal{SC}(G)\neq \varnothing$ (\autoref{PROP: SC non-empty}). 
Roughly speaking, we can choose an element $h\in G$ contracting in $X_2$ and an element $f\in F\subset  G$ from a set of three contracting elements in $X_1$. An appropriate element $g\in G$ and a sufficiently large integer $n\gg 0$ is chosen so that the path labelled by $(\cdots, h^ng, f,h^ng,f,\cdots)$ is a ``long'' admissible path in $X_1$ and that the path labelled by $(\cdots, h^n, gf,h^n, gf, \cdots )$ is a ``long'' admissible path in $X_2$. This implies that the element $h^ngf\in \mathcal{SC}(G)$.

\autoref{IntroLem: Extension} also relates the stable translation length of an element with the translation length at a basepoint up to a finite perturbation (\autoref{COR: Perturbation Length}). Combined with the notion of simultaneously contracting elements, we demonstrate here a concise result.

\begin{lemma}[\autoref{COR: Simultaneous Extension 15}]
Let a group $G$ admit two actions on two geodesic spaces $(X_1,d_1)$ and $(X_2,d_2)$ with contracting property, respectively. There exist $\epsilon >0$ and a set $S\subseteq \mathcal{SC}(G)$ of at most $15$ elements such that for any $g\in G$, we can choose $f\in S$ so that the following holds.
    \begin{enumerate}
        \item $\ell_{d_i}(gf)\ge d_i(o_i,gfo_i)-2\epsilon$, for $i=1,2${;}
        \item $gf,fg\in \mathcal{SC}(G)$.
    \end{enumerate}
\end{lemma}

\subsection*{{Structure of the paper}} This paper is organised as follows. In \autoref{SEC: Preliminaries}, we expand on a series of works by Yang in \cite{Yang14,Yang19,GY22} considering the metric geometry in the scope of group actions with contracting property, and mainly cover the preliminary material about contracting systems, admissible paths and group actions with contracting property. The most important tool, {the} so-called Extension Lemma, is given in \autoref{SEC: Extension Lemma}. Besides that, we construct a collection of simultaneously contracting elements for two group actions in \autoref{SEC: Simultaneously contracting}. Then, we prove \autoref{IntroThm: Main rigidity} and \autoref{IntroCor: Green Metric} in \autoref{SEC: Rigidity}. In \autoref{SEC: Application}, we focus on the applications in relatively hyperbolic groups and prove \autoref{IntroThm: Cusp uniform actions}. In \autoref{Sec: MLSR from confined subgroups}, we prove \autoref{IntroThm: MLSR from normal subgroups} and the proof guides us to the proof of \autoref{IntroThm: MLSR from confined subgroups}. To relate the stable length of a contracting element with the Gromov product of two boundary points, we recall the definition of horofunction boundary at first in \autoref{Sec: MLSR from GD subgroups} and then prove \autoref{IntroThm: MLSR from GD subgroups} in this section, which implies \autoref{IntroCor: GD of cusp uniform actions}. Finally, in \autoref{Sec: MLSR from subsets}, we follow the idea of Cantrell--Reyes \cite{CR23} to prove \autoref{IntroThm: MLSR from large subsets}.

{
\subsection*{Acknowledgments} 
We thank the anonymous referees for many useful comments and in particular, for the reference \cite{Bon91} on \autoref{IntroThm: MLSR from normal subgroups}. R. W. is supported by NSFC No.12471065 \& 12326601 and in part by Science and Technology Commission of Shanghai Municipality (No. 22DZ2229014). W. Y. is partially supported by National Key R \& D Program of China (SQ2020YFA070059) and  National Natural Science Foundation of China (No. 12131009 and No.12326601).
}

\section{Preliminaries}\label{SEC: Preliminaries}
\subsection{Notations and definitions}
\begin{notation}
\begin{enumerate}[(1)]
\item Let $(X,d)$ be a metric space.
\begin{enumerate}[(i).]
\item For $A,B\subseteq X$, $d(A,B)=\inf\left\{d(x,y)\mid x\in A,y\in B\right\}$.
\item For $A\subseteq X$ and $C\geq 0$, $N_C(A)=\left\{x\in X\mid d(x,A)\leq C\right\}$.
\item For $A,B\subseteq X$, the Hausdorff distance $d_H(A,B)=\inf\left\{r\geq 0\mid N_r(A)\supset B,N_r(B)\supset A\right\}${.}
\item For $A\subseteq X$, the diameter $\diam(A)=\sup\left\{d(x,y)\mid x,y\in A\right\}$.
\end{enumerate}
\item Let $(X,d)$ be a metric space.
\begin{enumerate}[(i).]
\item A path $\alpha: [0,L]\to X$ is called a (resp. $c$-rough) geodesic segment, if $d(\alpha(s),\alpha(t))=\abs{s-t}$ (resp. $\abs{d(\alpha(s),\alpha(t))-\abs{s-t}}\leq c$) $\forall s,t\in [0,L]$.

\item The length of a ($c$-rough) geodesic segment $\alpha: [0,L]\to X$ is denoted by $\mathrm{Len}(\alpha)=L$.
\item For a path $\gamma:[a,b]\to X$, $\gamma_-=\gamma(a)$, $\gamma_+ =\gamma(b)$; and for any $a\leq s<t\leq b$ with $\gamma(s)=x$ and $\gamma(t)=y$, $[x,y]_\gamma$ denotes the subpath of $\gamma$ between $x$ and $y$: $\gamma|_{[s,t]}:[s,t]\to X$.
\item The metric space $(X,d)$ is called ($c$-roughly) geodesic if for any two points $x,y\in X$, there exists a ($c$-rough) geodesic segment $\alpha$ such that $\alpha_-=x$ and $\alpha_+=y$.
\item Let $X$ be a ($c$-roughly) geodesic metric space. For any two points $x,y\in X$, $[x,y]$ denotes an arbitrary ($c$-rough) geodesic segment $\alpha$ with $\alpha_-=x$ and $\alpha_+=y$.
\end{enumerate}
\item Let $(X,d)$ and $(Y,d_{\ast})$ be two metric spaces, and $f:X\to Y$ be a map.
\begin{enumerate}[(i).]
\item $f$ is called a $(\lambda,c)$-quasi-isometric embedding for some $\lambda\geq 1$ and $c\geq 0$ if for any $x_1,x_2\in X$,
\begin{align*}
\frac{1}{\lambda}d(x_1,x_2)-c\leq d_{\ast}(f(x_1),f(x_2))\leq \lambda d(x_1,x_2)+c.
\end{align*}
It is called a $(\lambda,c)$-quasi-isometry if there exists a quasi-inverse $g:Y\to X$ such that $\sup_{x\in X}d(x,g(f(x)))<\infty$ and $\sup_{y\in Y}d_{\ast}(y,f(g(y)))<\infty$.
If $X\subset (-\infty,+\infty)$ is an interval, $f$ is called a $(\lambda,c)$-quasi-geodesic.
\item $f$ is called a rough isometry if $f$ is a $(1,c)$-quasi-isometry for some $c\geq 0$.
\item Suppose a group $G$ acts by isometries on both $X$ and $Y$. The map $f:X\to Y$ is called $G$-coarsely equivariant, if $\sup\left\{d_{\ast}(f(g\cdot x),g\cdot f(x))\mid x\in X,g\in G\right\}<+\infty$.
\end{enumerate}
\item For any two functions $\phi,\varphi$, we use the notation $\phi\sim \varphi$ (resp. $\phi\prec \varphi$) to mean that there exists a constant $C\geq 0$ such that $\abs{\phi-\varphi} \le C$ (resp. $\phi\le C\varphi$). If $C=C(c_1,\cdots,c_k)$, then we denote it as $\phi\sim_{c_1,\cdots,c_k} \varphi$ (resp. $\phi\prec_{c_1,\cdots,c_k} \varphi$). If $\phi\prec \varphi$ and $\varphi\prec \phi$, then we write $\phi\asymp \varphi$. %
\end{enumerate}
\end{notation}

\paragraph{\textbf{Entry and exit points.}} 
Given a property (P), a point $z$ on a geodesic segment $\alpha$ is called the entry point satisfying
(P) if $d(z,\alpha_-)$ is minimal among the points $z$ on $\alpha$ with the property (P). The exit point satisfying (P) is defined similarly so that $d(w,\alpha_+)$ is minimal.

For brevity, we will only deal with the following simple case in \autoref{SEC: Preliminaries} and \autoref{SEC: Extension Lemma}, and similar results can be generalised to non-proper and roughly geodesic cases.
\begin{convention}\label{CONV: Nearest Projection}
In \autoref{SEC: Preliminaries} and \autoref{SEC: Extension Lemma}, all geodesic metric spaces are assumed to be proper.

In this case, we have a concise definition for the \textit{closest point projection}.

For any subsets $A,B\subseteq X$, and any two points $x,x'\in X$.%
\begin{enumerate}
    \item $\pi_{A}(x)=\left\{y\in \overline{A}\mid d(x,y)=d(x,A)\right\}$.
    \item $\pi_{A}(B)=\cup_{x\in B}\pi_A(x)$ for $A,B\subseteq X$.
    \item $\d{A}(x,x')=\diam\left (\pi_A(\left\{x,x'\right\})\right)$.
\end{enumerate}\end{convention}
\paragraph{\textbf{Marked Length Spectrum.}}{
For an isometric group action on a metric space, we focus on a series of values named \textit{marked length spectrum}, which coincides with the length of closed geodesics when we consider $\pi_1(M)\curvearrowright \widetilde{M}$ for a negatively curved Riemannian manifold $M$.

Recall the following definition for marked length spectrum in Introduction.

\begin{definition}[Marked Length Spectrum]
Suppose a group $G$ acts isometrically on a metric space $(X,d)$ with a basepoint $o\in X$. For every $g\in G$, the \textit{stable translation length} of $g$ is defined as $$\ell_d(g)=\lim_{n\to +\infty}\frac{d(o,g^no)}{n}.$$ And we call the function $\ell_d: G\to [0,+\infty)$ given by $g\mapsto \ell_d(g)$ the \textit{marked length spectrum} (or \textit{MLS} for short) for this action.
\end{definition}
\begin{remark}
\label{RMK: MLS Basics}
\begin{enumerate}[(1)]
\item The limit exists due to sub-additivity: $d(o,g^{n+m}o)\leq d(o,g^no)+d(o,g^mo)$.
\item\label{MLSBasics2} By triangle inequality, $\ell_d(g)\leq d(o,go)$.
\item The marked length spectrum is independent with the choice of the basepoint, since $\abs{d(o,g^no)-d(o',g^no')}\leq 2d(o,o')$, for all $g\in G,\, o,o'\in X$.
\item The marked length spectrum is homogeneous: $\ell_d(g^n)=\abs{n}\ell_d(g)$ for all $g\in G$ and $n\in \mathbb{Z}$.
\item The marked length spectrum is invariant under conjugation: $\ell_{d}(hgh^{-1})=\ell_d(g)$ for all $g,h\in G$, since $\abs{d(o,hg^nh^{-1}o)-d(o,g^no)}\leq 2d(o,ho)$.
\end{enumerate}
\end{remark}

There are two commonly used terminologies closely related with marked length spectrum.
\begin{definition}\label{Algebraic and minimal length}
    Suppose a group $G$ acts isometrically on a metric space $(X,d)$.
    \begin{enumerate}
        \item The \textit{algebraic translation length} of an element $g\in G$ with respect to a basepoint $o\in X$ is defined as $\ell^o_{d}(g)=\inf_{x\in G\cdot o}d(x,gx)$. The function $\ell^o_d:G\to [0,+\infty)$ is called the \textit{marked algebraic length spectrum} (or \textit{MALS} for short) for this action.
        \item The \textit{minimal translation length} of an element $g\in G$ is defined as $\mu_{d}(g)=\inf_{x\in X}d(x,gx)$. The function $\mu_d:G\to [0,+\infty)$ is called the \textit{marked minimal length spectrum} (or \textit{MMLS} for short) for this action.
    \end{enumerate}
\end{definition}

As a direct corollary of \autoref{RMK: MLS Basics} (\ref{MLSBasics2}), for any basepoint $o\in X$ and any element $g\in G$, we have $$\ell_d(g)\le \mu_d(g)\leq \ell^o_{d}(g)\leq d(o,go).$$

Therefore, for any $n\in \mathbb N$, $n\ell_d(g)=\ell_d(g^n)\le \mu_d(g^n)\le \ell_d^o(g^n)\le d(o,g^no)$. Hence, 
\begin{align}\label{EQU: MALS implies MLS}
\ell_{d}(g)=\lim_{n\to +\infty}\frac{\mu_{d}(g^n)}{n}=\lim_{n\to +\infty}\frac{\ell^{o}_{d}(g^n)}{n}.
\end{align}
}

\subsection{Contracting system}
\begin{definition}[Contracting Subset]\label{DEF: Contracting Subset}
Let $(X,d)$ be a geodesic metric space. A subset $Y\subseteq X$ is called $C$-\textit{contracting} for $C\geq 0$ if for any geodesic (segement) $\alpha$ in $X$ with $d(\alpha,Y)\geq C$, we have $\diam(\mathrm{\pi}_Y(\alpha))\leq C$. 
$Y$ is called a \textit{contracting subset} if there exists $C\geq 0$ such that $Y$ is $C$-contracting, and $C$ is called a contraction constant of $Y$.
\end{definition}
\begin{example}\label{EX: Contracting Subset}
The following are well-known examples of contracting subsets.
\begin{enumerate}
    \item Bounded sets in a metric space.
    \item Quasi-geodesics and quasi-convex subsets in Gromov hyperbolic spaces. \cite{Ghy90}
    \item Fully quasi-convex subgroups, and maximal parabolic subgroups in particular, in relatively hyperbolic groups. \cite[Proposition 8.2.4]{Ger15}
    \item The subgroup generated by a hyperbolic element in groups with non-trivial Floyd boundary. \cite[Section 7]{Yang14}
    \item Contracting segments and axes of rank-$1$ elements in $\mathrm{CAT}(0)$-spaces in the sense of Bestvina and Fujiwara. \cite[Corollary 3.4]{BF09}
    \item The axis of any pseudo-Anosov element in the Teichm\"uller space equiped with Teichm\"uller distance by Minsky. \cite{Min96}
\end{enumerate}
\end{example}

It has been proven in \cite[Corollary 3.4]{BF09} that the definition of a contracting subset is equivalent to the following one considered by Minsky \cite{Min96}:

A subset $Y\subseteq X$ is contracting if and only if there exists $C'\geq 0$ such that any open metric ball $B$ with $B\cap Y=\varnothing$ satisfies $\diam(\mathrm{\pi}_Y(B))\leq C'$

Despite this
equivalence, we always base on \autoref{DEF: Contracting Subset} for a contracting subset.

We introduce some basic properties of contracting subsets.

\begin{lemma}\label{LEM: 1 Lipschitz}
Let  $Y\subseteq X$ be any $C$-contracting subset. Then for any $x,y\in X$, $\d{Y}(x,y)\leq d(x,y)+4C$.
\end{lemma}
\begin{proof}
Consider a geodesic segment $\alpha=[x,y]$. If $d(\alpha,Y)\geq C$, then $\d{Y}(x,y)\leq C$ by definition. 
Otherwise, let $z$ and $w$ be the entry and exit point of $\alpha\cap N_C(Y)$. 
For any $ a\in \pi_Y(x),b\in \pi_Y(y), p\in \pi_Y(z),q\in \pi_Y(w)$, it suffices to show that $d(a,b)\leq d(x,y)+4C$.

When $x\neq z$, the definition of a $C$-contracting subset implies $d(a,p)\le \diam (\pi_Y([x,z]))\le C$. By triangle inequality, $d(a,z)\le d(a,p)+d(z,p)\le 2C$. On the other hand, when $x=z$, obviously $d(a,z)\le C$.  In either case, $d(a,z)\le 2C$ and similarly, $d(b,w)\le 2C$. Therefore, $d(a,b)\le d(a,z)+d(b,w)+d(z,w)\le 4C+ d(x,y)$.
\end{proof}

\begin{lemma}\label{LEM: Geodesic Along Projection}
    Let $Y\subseteq X$ be a $C$-contracting subset. For any point $o\in Y$, $x\in X$, and $a\in \pi_Y(x)$, we have $d(o,x)\geq d(o,a)+d(a,x)-4C$.
\end{lemma}
\begin{proof}
    Let $\alpha=[o,x]$ be any geodesic connecting $o$ and $x$, and let $z\in \alpha$ be the exit point of $\alpha\cap N_C(Y)$. If $z=x$, then the result is obvious. Hence we only consider the case when $z\neq x$. Then $[z,x]_\alpha\cap N_C(Y)=\varnothing$. Choose a point $w\in \pi_Y(z)$, then $d(w,a)\leq C$ and $d(w,z)\le C$.
    Hence $d(o,x)=d(o,z)+d(z,x)\geq (d(o,w)-C)+(d(w,x)-C)\geq (d(o,a)-2C)+(d(a,x)-2C)$.
\end{proof}

\begin{lemma}\label{LEM: Thin Quadrilateral}
    Let $Y\subseteq X$ be a $C$-contracting subset. For any two points $x,y\in X$, and any $a\in \pi_Y(x), b\in \pi_Y(y)$, if $d(a,b)>C$, then $d(x,y)\ge d(x,a)+d(a,b)+d(b,y)-8C$.
\end{lemma}
\begin{proof}
    Let $[x,y]$ be any geodesic connecting $x$ and $y$. As $d(a,b)>C$, it follows from the definition of contracting set that $[x,y]\cap N_C(Y)\neq \varnothing$. Denote $x'(y')$ as the entry (exit) point of $[x,y]$ in $N_C(Y)$. The case when $x=x'$ or $y=y'$ is guaranteed by \autoref{LEM: Geodesic Along Projection}. Therefore, we only consider the case when $x\neq x'$ and $y\neq y'$. Then $d(a,\pi_Y(x'))\le \mathrm{diam}(\pi_Y([x,x']))\le C$ and $d(x',\pi_Y(x'))\le C$. This gives that $d(a,x')\le 2C$ and similarly one has that $d(b,y')\le 2C$. Hence,
    \begin{align*}
        d(x,y)& =d(x,x')+d(x',y')+d(y',y)\\ 
        &\ge d(x,a)-d(a,x')+d(a,b)-d(a,x')-d(b,y')+d(b,y)-d(b,y')\\
        &\ge d(x,a)+d(a,b)+d(b,y)-8C
    \end{align*}
completing the proof.    
\end{proof}

\begin{lemma}\label{LEM: Finite Hausdorff Distance}
Let $Y,Z\subseteq X$ be two subsets with $d_H(Y,Z)<\infty$. If  $Y$ is contracting, then $Z$ is also contracting.
\end{lemma}
\begin{proof}
Suppose $d_H(Y,Z)<D<\infty$ and $Y$ is $C$-contracting, we will show that $Z$ is $(6D+7C)$-contracting.

Suppose $\alpha=[x,y]$ is a geodesic segment such that $d(\alpha,Z)\geq 6D+7C$. 
For any $ a\in \pi_Z(x), p\in \pi_Y(x)$, there exist $p'\in Z,\, a'\in Y$ such that $d(a,a')\leq D$, $d(p,p')\leq D$. We will first show that $d(a',p)\leq 2D+3C$, and consequently $d(a,p)\leq 3D+3C$.

Denote $l=d(x,Y)=d(x,p)$, then $d(x,a')\leq d(x,a)+D=d(x,Z)+D\leq d(x,p')+D\leq l+2D$. Let $w\in [x,a']$ be the entry point in $N_C(Y)$. If $x=w$, then $l\leq C$ and $d(a',p)\leq d(x,a')+d(x,p)\leq 2l+2D\leq 2D+2C$. 
Thus it suffices to consider the case when $x\neq w$. In this case, $l\geq C$. Choose a point $z\in \pi_{Y}(w)$, then $d(w,z)=C$, $d(x,w)\geq l-C$ and $d(w,a')= d(x,a')-d(w,a')\leq (l+2D)-(l-C)=2D+C$. As $[x,w]\cap N_C(Y)=\varnothing$, we have $\d{Y}(x,w)\leq C$. Therefore, $d(a',p)\leq d(a',w)+d(w,z)+d(z,p)\leq (2D+C)+C+C=2D+3C$.

Similarly, for any $ b\in \pi_Z(y)$ and $q\in \pi_Y(y)$, we have $d(b,q)\leq 3D+3C$.

Notice that $d(\alpha,Y)\geq d(\alpha,Z)-d_H(Y,Z)\geq 6D+7C-D\geq C$, hence $\d{Y}(x,y)\leq C$. 
Thus $d(a,b)\leq d(a,p)+d(p,q)+d(q,b)\leq (3D+3C)+C+(3D+3C)=6D+7C$, and $\d{Z}(x,y)\leq 6D+7C$.
\end{proof}
\begin{lemma}\label{LEM: Quasi Convex}
Suppose $Y\subseteq X$ is a $C$-contracting subset. Then for any $x,y\in N_C(Y)$, $\alpha=[x,y]\subseteq N_{3C}(Y)$.
\end{lemma}
\begin{proof}
Without loss of generality, we may suppose $C>0$. Suppose by contrary that there exists a point $p\in \alpha$ such that $d(p,Y)\geq 3C$. Let $z\in [x,p]_{\alpha}$ be the exit point of $N_C(Y)$, and let $w\in [p,y]_{\alpha}$ be the entry point of $N_C(Y)$. Then $d(z,Y)=d(w,Y)=C$. choose arbitrary points $a\in \pi_Y(z)$ and $b\in \pi_Y(w)$. 
Notice that $d([z,w]_{\alpha},Y)\geq C$, hence $d(a,b)\leq C$. Thus $d(z,w)\leq d(z,a)+d(a,b)+d(b,w)\leq 3C$. 
On the other hand $d(z,w)=d(z,p)+d(p,w)\geq \abs{d(z,Y)-d(p,Y)}+\abs{d(p,Y)-d(w,Y)}\geq 4C$, which is a contradiction.
\end{proof}

\begin{definition}[Bounded Intersection]
Two subsets $Y,Z\subseteq X$ have $\mathcal{R}$-\textit{bounded intersection} for a function $\mathcal{R}:[0,+\infty)\to[0,+\infty)$ if $\diam(N_r(Y)\cap N_r(Z))\leq \mathcal{R}(r)$, for all $ r\geq 0$.
\end{definition}
\begin{definition}[Bounded Projection]
Two subsets $Y,Z\subseteq X$ have $B$-\textit{bounded projection} for some $B>0$ if either $\diam(\mathrm{\pi}_{Z}(Y))\leq B$, or $\diam(\mathrm{\pi}_{Y}(Z))\leq B$.
\end{definition}
\begin{lemma}\label{PROP: Bounded Intersection Equals Bounded Projection}
Assume that $Y,Y'$ are two $C$-contracting subsets in a geodesic metric space $X$. Then $Y$ and $Y'$ have $\mathcal{R}$-bounded intersection for some $\mathcal{R}:[0,+\infty)\to[0,+\infty)$ if and only if they have $B$-bounded projection for some $B>0$. Here $\mathcal R$ and $B$ depend on the other in a quantitative way: \begin{align*}
&\diam(\mathrm{\pi}_{Y'}(Y))\leq B\Rightarrow \diam(N_r(Y)\cap N_r(Y'))\leq B+8C+4r;\\
&\diam(N_r(Y)\cap N_r(Y'))\leq \mathcal{R}(r)\Rightarrow\mathrm{max}\left\{\diam(\mathrm{\pi}_{Y'}(Y)),\diam(\mathrm{\pi}_{Y}(Y'))\right\}\leq6C+\mathcal{R}(3C).
\end{align*} 
\end{lemma}

\begin{proof}
(1). Bounded Projection $\Rightarrow$ Bounded Intersection

Suppose $\diam(\mathrm{\pi}_{Y'}(Y))\leq B$. For any $ x_0,y_0\in N_r(Y)\cap N_r(Y')$, we can find $x,y\in Y$ such that $d(x_0,x)\leq r$, $d(y_0,y)\leq r$. 
Choose $x'\in \pi_{Y'}(x_0)$, $y'\in \pi_{Y'}(y_0)$, then $d(x_0,x')\leq r$, $d(y_0,y')\leq r$. 
Choose $x''\in \pi_{Y'}(x)$, $y''\in \pi_{Y'}(y)$, then $d(x'',y'')\leq \diam(\mathrm{\pi}_{Y'}(Y))\leq B$. According to \autoref{LEM: 1 Lipschitz}, $d(x',x'')\leq d(x_0,x)+4C\leq r+4C$, and similarly $d(y',y'')\leq  r+4C$.
Hence $d(x_0,y_0)\leq d(x_0,x')+d(x',x'')+d(x'',y'')+d(y'',y')+d(y',y_0)\leq B+8C+4r$.

(2). Bounded Intersection $\Rightarrow$ Bounded Projection

Suppose $\diam(N_r(Y)\cap N_r(Y'))\leq \mathcal{R}(r)$. For any $ x,y \in Y$, consider the geodesic segment $\alpha=[x,y]$. According to \autoref{LEM: Quasi Convex}, $\alpha\subseteq N_{3C}(Y)$. We will show that $\d{Y'}(x,y)\le \mathcal{R}(3C)+6C$.

If $d(\alpha,Y')\geq C$, then $\d{Y'}(x,y)\leq \diam(\pi_{Y'}(\alpha))\leq C$. Therefore, it suffices to consider the case when $\alpha\cap N_C(Y')\neq \varnothing$. 
Let $z,w\in \alpha$ be the entry and exit point of $\alpha\cap N_C(Y')$. Then $z,w\in N_{3C}(Y)\cap N_{3C}(Y')$, hence $d(z,w)\leq \mathcal{R}(3C)$. According to \autoref{LEM: 1 Lipschitz}, $\d{Y'}(z,w)\leq \mathcal{R}(3C)+4C$.

If $x\neq z$ and $y\neq w$, then $[x,z]\cap N_C(Y')=\varnothing$ and $\d{Y'}(x,z)\le C$ and similarly $\d{Y'}(w,y)\leq C$. Thus, $\d{Y'}(x,y)\leq \d{Y'}(x,z)+\d{Y'}(z,w)+\d{Y'}(x,y)\leq \mathcal{R}(3C)+6C$. Similar results also hold when $x=z$ or $y=w$.

Hence $\diam(\pi_{Y'}(Y))\leq \mathcal{R}(3C)+6C$, and the same result holds for $\diam(\pi_{Y}(Y'))$.
\end{proof}
\begin{definition}[Contracting system]
Let $(X,d)$ be a geodesic metric space, and $\mathbb{X}=\left\{X_i\mid i\in I\right\}$ be a collection of subsets in $X$.
\begin{enumerate}[(1)]
\item $\mathbb{X}$ is called a \textit{contracting system} if each $X_i$ is a contracting subset.
\item $\mathbb{X}$ has \textit{contraction constant} $C$ if every $X_i$ has contraction constant $C$.
\item $\mathbb{X}$ has $\mathcal{R}$-\textit{bounded intersection} if every pair $X_i$ and $X_j$ have $\mathcal{R}$-bounded intersection ($i\neq j\in I$).
\end{enumerate}
\end{definition}
\subsection{Admissible path}
\begin{definition}[Admissible path]\label{DEF: Admissible Path}
Let $(X,d)$ be a geodesic metric space and $\mathbb{X}$ be a contracting system in $X$. Given $D,\tau\geq 0$ and a function $\mathcal{R}:[0,+\infty)\to[0,+\infty)$ called the \textit{bounded intersection gauge}, a path $\gamma$ is called a $(D,\tau)$-\textit{admissible path}  with respect to  $\mathbb{X}$, if the path $\gamma$ consists of a (finite, infinite or bi-infinite) sequence of consecutive geodesic segments $\gamma=\cdots q_ip_iq_{i+1}p_{i+1}\cdots$, satisfying the following ``Long Local'' and ``Bounded Projection'' properties:
\begin{itemize}
\item[\textbf{(LL1)}]\label{LL1} For each $p_i$ there exists $X_i\in\mathbb{X}$ such that the two endpoints of $p_i$ lies in $X_i$, and $\mathrm{Len}(p_i)>D$ unless $p_i$ is the first or last geodesic segment in $\gamma$.
\item[\textbf{(BP)}]\label{BP} For each $X_i$ we have $\d{X_i}((p_i)_+,(p_{i+1})_-)\leq \tau$, and $\d{X_i}((p_{i-1})_+,(p_{i})_-)\leq \tau$. Here $(p_{i+1})_-=\gamma_+$ if $p_{i+1}$ does not exist, and $(p_{i-1})_+=\gamma_-$ if $p_{i-1}$ does not exist.
\item[\textbf{(LL2)}]\label{LL2} Either $X_i\neq X_{i+1}$ and $X_i$ and $X_{i+1}$ have $\mathcal{R}$-bounded intersection, or $d((p_i)_+,(p_{i+1})_-)>D$.
\end{itemize}
\end{definition}
\begin{definition}[Special Admissible Path]\label{DEF: Special Admissible Path}
Under the same assumption as \autoref{DEF: Admissible Path}, given $D,\tau\geq 0$, $\gamma$ is called a \textit{special $(D,\tau)$-admissible path}  with respect to  $\mathbb{X}$, if it satisfies \hyperref[LL1]{\textbf{(LL1)}}, \hyperref[BP]{\textbf{(BP)}} and \hyperref[LL2']{\textbf{(LL2')}}:
\begin{itemize}
\item[\textbf{(LL2')}]\label{LL2'} $d((p_i)_+,(p_{i+1})_-)>D$.
\end{itemize}
\end{definition}
\begin{notation}
For a $(D,\tau)$-admissible path $\gamma=\cdots q_ip_iq_{i+1}p_{i+1}\cdots$  with respect to  $\mathbb{X}$ defined in \autoref{DEF: Admissible Path} (or \autoref{DEF: Special Admissible Path}), 
let $X_i\in \mathbb{X}$ be the contracting subset associated to $p_i$ according to \hyperref[LL1]{\textbf{(LL1)}}. 
\begin{enumerate}[(1)]
    \item The endpoints of each geodesic segment: $(p_i)_+,(p_i)_-,(q_i)_-,(q_i)_+$ are called \textit{vertices} of $\gamma$.
    \item The collection of all $X_i$'s is denoted by $\mathbb{X}(\gamma)$, which forms a contracting system in $X$.
    \item $\gamma$ has $\mathrm{Len}_q$-constant $L$, if for each $i$, $\mathrm{Len}(q_i)\leq L$; and $\gamma$ has $\mathrm{Len}_p$-constant $K$, if for each $i$, $\mathrm{Len}(p_i)\leq K$.
\end{enumerate}
\end{notation}
In the following definitions, a sequence of points $x_i$
in a path $\alpha$ is called linearly ordered if $x_{i+1} \in  [x_i, \alpha_+]_\alpha $ for each $i$.
\begin{definition}[Fellow Travel]\label{DEF: Fellow Travel}
Let $\gamma=p_0q_1p_1\cdots q_np_n$ be a $(D,\tau)$-admissible path, and $\alpha$ be a path such that $\alpha_-=\gamma_-$, $\alpha_+=\gamma_+$. Given $\epsilon>0$, the path $\alpha$ $\epsilon$-\textit{fellow travels} $\gamma$ if there exists a sequence of linearly ordered points $z_i,w_i$ $(0\leq i\leq n)$ on $\alpha$ such that $d(z_i,(p_i)_-)\leq \epsilon$, $d(w_i,(p_i)_+)\leq \epsilon$.
\end{definition}
\begin{proposition}[{\cite[Proposition 3.3 and Corollary 3.9]{Yang14}} ]\label{PROP: Fellow Travel}
Suppose $(X,d)$ is a geodesic metric space and $\mathbb{X}$ is a contracting system of $X$ with some contraction constant $C$. For any $\tau>0$ and $\mathcal{R}:[0,+\infty)\to[0,+\infty)$, there are constants $D=D(\tau,C,\mathcal{R})>0,\epsilon=\epsilon(\tau,C,\mathcal{R})>0, B=B(\tau,C,\mathcal{R})>0$, such that the following holds.

Let $\gamma$ be any $(D,\tau)$-admissible path  with respect to  $\mathbb{X}$ and the bounded intersection gauge $\mathcal{R}:[0,+\infty)\to[0,+\infty)$, and $\gamma$ consists of a finite sequence of geodesic segments $\gamma=\cdots q_ip_iq_{i+1}p_{i+1}\cdots$ with each $p_i$ having endpoints in $X_i\in\mathbb{X}$.
\begin{enumerate}[(1)]
\item\label{2.15.1} $\diam(\pi_{X_i}(\left [\gamma_-,(p_i)_-\right]_\gamma))\leq B$, $\diam(\pi_{X_i}(\left [(p_i)_+,\gamma_+\right]_\gamma))\leq B$.
\item\label{2.15.2} Let $\alpha$ be a geodesic segment between $\gamma_-$ and $\gamma_+$. Then $\alpha$ $\epsilon$-fellow travels $\gamma$.
\end{enumerate}

In addition, if $\gamma$ is a special $(D,\tau)$-admissible path  with respect to  $\mathbb{X}$, then $D=D(\tau,C),\epsilon=\epsilon(\tau,C),B=B(\tau,C)$ are only dependent on $\tau$ and $C$. 
\end{proposition}

\begin{corollary}\label{COR: Bounded Intersection}
Let $\mathbb{X}$ be a contracting system of $X$ with some contraction constant $C$. For any $\tau>0$ and $\mathcal{\nu}:[0,+\infty)\to [0,+\infty)$, let $D=D(\tau,C,\nu )>0$ be given by \autoref{PROP: Fellow Travel}. Then for any $(D,\tau)$-admissible path $\gamma$  with respect to  $\mathbb{X}$ and the bounded intersection gauge $\nu$, the contracting system $\mathbb{X}(\gamma)$ has $\mathcal{R}$-bounded intersection, with $\mathcal{R}=\mathcal{R}_{\tau,C,\nu}$.

In addition, if $\gamma$ is a special $(D,\tau)$-admissible path  with respect to  $\mathbb{X}$ then $D=D(\tau,C)$ and $\mathcal{R}=\mathcal{R}_{\tau,C}$ are only dependent on $\tau$ and $C$.
\end{corollary}
\begin{proof}
Suppose $\gamma=\cdots q_ip_iq_{i+1}p_{i+1}\cdots$, and let $B=B(\tau,C)$ be given by \autoref{PROP: Fellow Travel}. We will show that $\mathbb{X}(\gamma)=\left\{X_i\right\}$ has $2B$-bounded projection.

For any $i<j$ (the case $i>j$ is similar), and for any $x\in X_j$, consider a special $(D,\tau)$-admissible path $\alpha=p_iq_{i+1}p_{i+1}\cdots p_{j-1}q_{j}[(p_j)_-,x]$. According to \autoref{PROP: Fellow Travel} (\ref{2.15.1}), $\d{X_i}([(p_i)_+,\gamma_+])\leq B$. Especially, $\d{X_i}((p_j)_-,x)\leq B$, $\forall x\in X_j$.

Hence $\diam(\mathrm{\pi}_{X_i}(X_j))\leq 2B$, and the result follows according to \autoref{PROP: Bounded Intersection Equals Bounded Projection}.
\end{proof}

\begin{proposition}[{\cite[Proposition 2.9]{Yang19}} ]\label{PROP: Admissible Path Contracting}
Let $\mathbb{X}$ be a contracting system of $X$ with contraction constant $C$ and $\mathcal{R}$-bounded intersection. For any $\tau>0$ there is a constant $D=D(\tau,C,\mathcal{R})>0$, such that any $(D,\tau)$-admissible path $\gamma$  with respect to  $\mathbb{X}$ with $\mathrm{Len}_q$-constant $L$ and $\mathrm{Len}_p$-constant $K$ is $c$-contracting, for some $c=c(\tau,C,\mathcal{R},L,K)>0$.
\end{proposition}

\begin{convention}[Path Label Convention]
Suppose $G$ acts by isometry on a geodesic metric space $(X,d)$ with a basepoint $o\in X$. A path labelled by a sequence of elements $(\cdots,g_{-1},g_0,g_1,g_2,g_3,\cdots)$ in $G$ denotes a path $\gamma$ consisting of a sequence of geodesic segments $\gamma=\cdots\cup[g_0^{-1}g_{-1}^{-1}o,g_0^{-1}o]\cup[g_0^{-1}o,o]\cup[o,g_1o]\cup [g_1o,g_1g_2o]\cup\cdots\cup [g_1\cdots g_no,g_1\cdots g_{n+1}o]\cup\cdots$. The endpoints of these geodesic segments are called vertices.
\end{convention}

\subsection{Group actions with contracting property}
\text{ }

In this subsection, we assume that a group $G$ acts by isometry on a geodesic metric space $(X,d)$ with a basepoint $o$.

\begin{definition}[Contracting Element]\label{DEF: Contracting Element}
An element $h\in G$ is called a \textit{contracting element} if $\left\langle h\right\rangle\cdot o$ is a contracting subset in $X$ and the map $\mathbb{Z}\to X$, $n\mapsto h^no$ is a quasi-isometric embedding.
\end{definition}

\begin{lemma}\label{LEM: Fellow Travel Property}
For any contracting element $h\in G$, there exists $\epsilon>0$ such that the following property holds: for any integers $ l<m<n$ and any geodesic segment $\alpha=[h^lo,h^no]$, $d(h^mo,\alpha)\leq \epsilon$.
\end{lemma}
\begin{proof}
    By \autoref{DEF: Contracting Element}, $\left\langle h\right\rangle o$ is $C$-contracting for some $C\ge 0$. It follows from \autoref{LEM: Quasi Convex} that $\alpha=[h^lo,h^no]\subseteq N_{3C}(\left\langle h\right\rangle o)$. 
    For each point $x\in \alpha$, let $I_x=\left\{k\in \mathbb{Z}\mid d(x,h^ko)\leq 3C\right\}$. Then $I_x\neq \varnothing$. In addition, $l\in I_{h^lo}$ and $n\in I_{h^no}$. 

    By intermediate value principle, there exist two points $x,y\in \alpha$ with $d(x,y)\le 1$, such that there exist two integers $r\le m\le s$ so that $r\in I_x$ and $s\in I_y$. Hence, $d(h^ro,h^so)\le d(h^ro,x)+d(x,y)+d(y,h^so)\leq 6C+1$.

    \autoref{DEF: Contracting Element} also provides that the map $\mathbb{Z}\to X$, $n\mapsto h^no$ is a $(\lambda,c)$-quasi-isometric embedding. That is, $\frac{1}{\lambda}\abs{j-i}-c\leq d(h^io,h^jo)\leq \lambda \abs{j-i}+c$ for any $i,j\in \mathbb Z$.

    Therefore, 
    \begin{align*}
        d(h^mo,\alpha)& \le d(h^mo,x)\le d(h^mo,h^ro)+d(h^ro,x)\le \lambda(m-r)+c+3C\\
    &\le \lambda(s-r)+c+3C\le \lambda^2(d(h^ro,h^so))+(\lambda^2+1)c+3C\le \lambda^2(6C+c+1)+c+3C.
    \end{align*}

    By setting $\epsilon:=\lambda^2(6C+c+1)+c+3C$, we complete the proof.
\end{proof}

\begin{definition}[weakly independent]
Suppose $g,h\in G$ are two contracting elements. $g$ and $h$ are called \textit{weakly independent} if $\left\langle g\right\rangle\cdot o$ and $\left\langle h\right\rangle\cdot o$ have $\mathcal{R}$-bounded intersection for some $\mathcal{R}:[0,+\infty)\to[0,+\infty)$.
\end{definition}
\begin{remark}
The definition of ``weak-independence'' is a weaker condition compared with the term ``independence'' in \cite{Yang19,GY22}, but if the action is proper (always in the case there), then these two notions of independence actually coincide for a pair of non-conjugate elements. We caution reader that in most literature, weakly-independence is just called independence. We made the above difference, in order to be distinguished  with our earlier use in \cite{Yang19,GY22}.  
\end{remark}

As a straightforward application of \autoref{LEM: Finite Hausdorff Distance}, one can generate a new contracting element $h'$ from a contracting element $h$ by one of the following operations:
\begin{enumerate}[(1)]
\item $h'^n=h$, for some $n\in\mathbb{Z}$;
\item $h'=h^m$, for some $m\in\mathbb{Z}$;
\item $h'=khk^{-1}$, for some $k\in G$.
\end{enumerate}

In addition, if $h$ and $g$ are two weakly independent contracting elements, then: 
\begin{enumerate}[(1)]
\item If $h'^n=h$ for some $n\in\mathbb{Z}$, or $h'=h^m$ for some $m\in\mathbb{Z}$, then $h'$ and $g$ are weakly independent.
\item $khk^{-1}$ and $kgk^{-1}$ are weakly independent, for any $ k\in G$.
\end{enumerate}


%
\begin{lemma}\label{LEM: Non Elementary}
Suppose $g,h\in G$ are two weakly independent contracting elements. Then there exists $N>0$ such that the following holds.
\begin{enumerate}
    \item\label{_NE1} For any $n,m\ge N$, $g^nh^m$ is a contracting element.
    \item\label{_NE2} For any $n,m\ge N$, $\left\{(h^mg^{n})^kh(h^mg^{n})^{-k}\mid {k\in\mathbb{Z}}\right\}$ is a collection of pairwise weakly independent contracting elements.
\end{enumerate}
\end{lemma}
\begin{proof}
According to \autoref{PROP: Bounded Intersection Equals Bounded Projection}, $\diam(\mathrm{\pi}_{\left\langle h\right\rangle o}(\left\langle g\right\rangle o))\leq \tau $ for some $\tau >0$.

Let $C$ be a contraction constant of $\left\langle h\right\rangle o$. Let $D_1=D(\tau, C)$ and $\epsilon=\epsilon(\tau, C)$ be decided by \autoref{PROP: Fellow Travel}, $\mathcal{R}=\mathcal{R}_{\tau,C}$ be decided by \autoref{COR: Bounded Intersection}, and $D_2=D(\tau,C,\mathcal{R})$ be decided by \autoref{PROP: Admissible Path Contracting}. Let $D=\max\left\{D_1+2\epsilon,D_2\right\}$.

Because $\left\langle g\right\rangle o$ and $\left\langle h\right\rangle o$ are quasi-geodesics, we can find $N\gg 0$ such that $d(o,g^no)>D$ and $d(o,h^mo)>D$ for any $n,m\ge N$.

For any $n,m\ge N$, it is easy to verify that the path $\gamma$ labelled by $(\cdots,g^n,h^m,g^n,h^m,g^n,h^m,\cdots)$ is a special $(D,\tau)$-admissible path  with respect to  $\mathbb{X}(\gamma)=\left\{(g^nh^m)^kg^n\left\langle h\right\rangle o\mid k\in \mathbb{Z}\right\}$.

As $\gamma$ has $\mathrm{Len}_q$-constant $d(o,g^no)$ and $\mathrm{Len}_p$-constant $d(o,h^mo)$, according to \autoref{PROP: Admissible Path Contracting}, $\gamma$ is a contracting subset. By \autoref{LEM: Finite Hausdorff Distance}, as $d_H(\left\langle g^nh^m\right\rangle o,\gamma)<+\infty$, the orbit $\left\langle g^nh^m\right\rangle o$ is also a contracting subset. 

Let $\beta=[o,g^nh^mo]$ be a geodesic segment. By \autoref{PROP: Fellow Travel}, there exists a point $x\in \beta$ such that $d(g^no,x)\le \epsilon$. By triangle inequality, $d(o,g^nh^mo)\ge d(o,g^no)+d(o,h^mo)-2\epsilon >2D_1+2\epsilon$.

In addition, for each $k>0$, any geodesic segment $\alpha=[o,(g^nh^m)^ko]$ $\epsilon$-fellow travels the subpath of $\gamma$ between $o$ and $(g^nh^m)^ko$. In particular, there exist a point $y\in \alpha$ such that $d((g^nh^m)^{k-1}o,y)\le \epsilon$. By triangle inequality, $d(o,(g^nh^m)^ko)\ge d(o,(g^nh^m)^{k-1}o)+d(o,g^nh^mo)-2\epsilon$. By induction, $k(d(o,g^nh^mo)-2\epsilon)\le d(o,(g^nh^m)^ko)\le kd(o,g^nh^mo)$, where $d(o,g^nh^mo)-2\epsilon> 2D_1\ge 0$. Hence, the map $k\in \mathbb{Z}\mapsto (g^nh^m)^ko\in X$ is a quasi-isometric embedding, thus proving that $g^nh^m$ is a contracting element.

On the other hand, according to \autoref{COR: Bounded Intersection}, the contracting system $\mathbb{X}(\gamma)$ has $\mathcal{R}$-bounded intersection.


Denote $X_k=(g^nh^m)^kg^n\left\langle h\right\rangle o$, so that $\mathbb{X}(\gamma)=\left\{X_k\mid k\in \mathbb{Z}\right\}$.

For any $ j\neq k\in\mathbb{Z}$, $g^{-n}X_k=(h^mg^{n})^k\left\langle h\right\rangle o$ and $g^{-n}X_j=(h^mg^{n})^j\left\langle h\right\rangle o$ have $\mathcal{R}$-bounded intersection. 
Notice that $d_H(\left\langle(h^mg^{n})^k h (h^mg^{n})^{-k}\right\rangle o, g^{-n}X_k)\leq d(o,(h^mg^{n})^{-k}o)<\infty$. And similarly, $d_H(\left\langle(h^mg^{n})^j h (h^mg^{n})^{-j}\right\rangle o, g^{-n}X_j)<\infty$.

Hence $\left\langle(h^mg^{n})^j h (h^mg^{n})^{-j}\right\rangle o$ and $\left\langle(h^mg^{n})^k h (h^mg^{n})^{-k}\right\rangle o$ have bounded intersection, implying that $\left\{(h^mg^{n})^kh(h^mg^{n})^{-k}\mid {k\in\mathbb{Z}}\right\}$ is a collection of pairwise weakly independent contracting elements.
\end{proof}
\begin{definition}[Contracting Property]\label{DEF: Contracting Property}
A group action $G\curvearrowright X$ has \textit{contracting property} if there exist two  weakly independent contracting elements in $G$.
\end{definition}

By \autoref{LEM: Non Elementary} (\ref{_NE2}), if $G\curvearrowright X$ has the contracting property, then there exist infinitely many pairwise weakly independent elements in $G$.

\begin{example}\label{EX: Contracting Property}
The following provides abundant examples of group actions with contracting property.
\begin{enumerate}
    \item A group action on a (quasi-)geodesic Gromov hyperbolic space which is non-elementary.
    \item A group $G$ with non-trivial Floyd boundary (for example, a relatively hyperbolic group) acts on its Cayley graph with respect to a generating set $S$.  \cite{Ger15,Yang14}
    \item A non-elementary group action on a $\mathrm{CAT}(0)$-space with rank-$1$ elements. \cite{BF09}
    \item A $Gr'(1/6)$-labelled graphical small cancellation group $G$ with finite components labelled by a finite set $S$ acts on the Cayley graph $(Y, d)$ with respect to the generating set $S$. \cite{ACGH16}
    \item 
    Any non-elementary (i.e. non-virtually cyclic) group acts properly on a geodesic metric space with one contracting elements. \cite[Lemma 4.6]{Yang14}
    \item The mapping class group of a hyperbolic surface $\MCG(\Sigma)$ acts on the Teichm\"uller space $\Teich(\Sigma)$ equipped with Teichm\"uller metric, or on the curve complex $\mathcal{C}(\Sigma)$. \cite{Bow08}
    \item 
    More generally, any non-elementary group acts with WPD loxodromic elements in sense of Bestvina-Fujiwara on hyperbolic spaces. \cite[Theorem 6.8]{DGO17}
\end{enumerate}
\end{example}


\section{Extension lemma}\label{SEC: Extension Lemma}
In this section, suppose a group $G$ acts by isometry on a geodesic metric space $(X,d)$ with a basepoint $o\in X$, and the action $G\curvearrowright X$ has contracting property. For brevity of proof, we assume that $X$ is proper so that we use the same definition for the closest point projection as in \autoref{CONV: Nearest Projection}.

\begin{lemma}\label{LEM: Finite Projection}
Suppose $Y,Z\subseteq X$ are two $C$-contracting subsets with $\mathcal{R}$-bounded intersection and $Y\cap Z\neq \varnothing$. Then there exist $\tau=\tau(C,\mathcal{R})\geq 0$, such that for any $ o\in Y\cap Z$ and $x\in X$, if $\alpha=[o,x]$ is a geodesic segment, then
\begin{align*}
\min\left\{\diam (\pi_Y(\alpha)),\diam (\pi_Z(\alpha))\right\}\leq \tau.
\end{align*}

In particular, $\min\left\{\d{Y}(o,x),\d{Z}(o,x)\right\}\leq \tau$.
\end{lemma}
\begin{proof}
We will prove that $\tau=\mathcal{R}(3C)+5C$ satisfies the conditions.

Let $y,z\in \alpha$ be the exit point of $\alpha$ in $N_C(Y)$ and $N_C(Z)$, respectively. Without loss of generality, suppose $o,y,z,x$ are linearly ordered on $\alpha$. 
According to \autoref{LEM: Quasi Convex}, $[o,y]_{\alpha}\subseteq N_{3C}(Y)$, and $[o,z]_{\alpha}\subseteq N_{3C}(Z)$. Hence $[o,y]_{\alpha}\subseteq N_{3C}(Y)\cap N_{3C}(Z)$. Since $Y$ and $Z$ have $\mathcal{R}$-bounded intersection, we have $d(o,y)\leq \mathcal{R}(3C)$. 
Therefore, by \autoref{LEM: 1 Lipschitz}, $\diam\left (\pi_{Y}([o,y]_{\alpha})\right )\leq \mathcal{R}(3C)+4C$.

We consider two cases. If $y=x$, then
$\diam(\pi_Y(\alpha))=\diam\left (\pi_{Y}([o,y]_{\alpha})\right )\leq \mathcal{R}(3C)+4C$.
Otherwise, if $y\neq x$, then $d([y,x]_\alpha,Y)\geq C$, hence $\diam\left (\pi_{Y}([y,x]_{\alpha})\right )\leq C$. 
Therefore, $\diam(\pi_Y(\alpha))\leq \diam\left (\pi_{Y}([o,y]_{\alpha})\right )+\diam\left (\pi_{Y}([y,x]_{\alpha})\right )\leq \mathcal{R}(3C)+5C$.
\end{proof}
We are now ready to prove a simplified version of Extension Lemma in \cite{Yang19}.
\begin{lemma}[Extension Lemma]\label{LEM: Modified Extension Lemma}
For any three pairwise weakly independent contracting elements $h_1,h_2,h_3\in G$, 
there exists $\tau>0$ such that the following holds:\\
For any $D>0$, choose $f_i\in \left\langle h_i\right\rangle$ such that $d(o,f_io)>D$, and let $F=\left\{f_1,f_2,f_3\right\}$.
\begin{enumerate}[(1)]
\item\label{ext1} For any pair of elements $g,g'\in G$, we can find $i\in\left\{1,2,3\right\}$, such that $\d{\left\langle h_i\right\rangle o}\left(o,g^{-1}o\right)\leq \tau$, and $\d{\left\langle h_i\right\rangle o}\left(o,g'o\right)\leq \tau$.
\item\label{ext2} For any (finite or infinite) sequence of elements $\left\{g^{(k)}\in G\right\}$ satisfying $d(o,g^{(k)} o)>D$, $f^{(k)}\in F$ is chosen for each pair $(g^{(k)},g^{(k+1)})$ according to (\ref{ext1}). Then the path $\gamma$ labelled by $(\cdots, g^{(k)}, f^{(k)}, g^{(k+1)}, f^{(k+1)}, g^{(k+2)}, \cdots)$ is a special $(D,\tau)$-admissible path  with respect to  $\mathbb{X}=\{g\cdot \left\langle h_i\right\rangle\cdot o\mid g\in G, 1\leq i\leq 3\}$.
\end{enumerate}
\end{lemma}
\begin{proof}
By assumption, the three contracting subsets $\left\langle h_i\right\rangle o$ have contraction constant $C$ and pairwise $\mathcal{R}$-bounded intersection for some $C\geq 0$ and $\mathcal{R}:[0,+\infty)\to[0,+\infty)$. The basepoint $o\in \bigcap_{1\leq i\leq 3}\left\langle h_i\right\rangle o$. Let $\tau=\tau(C,\mathcal{R})$ according to \autoref{LEM: Finite Projection}.

\textbf{(1)} According to \autoref{LEM: Finite Projection}, there are at least two indices $\left\{i_1,i_2\right\}\subset\left\{1,2,3\right\}$ such that $\d{\left\langle h_{i_{1,2}}\right\rangle o}\left(o,g^{-1}o\right)\leq \tau$, and at least two indices $\left\{j_1,j_2\right\}\subset\left\{1,2,3\right\}$ such that $\d{\left\langle h_{j_{1,2}}\right\rangle o}\left(o,g'o\right)\leq \tau$. 
Then $i\in\left\{i_1,i_2\right\}\cap \left\{j_1,j_2\right\}$ satisfies the condition in (\ref{ext1}).

\textbf{(2)} Suppose $f^{(k)}=f_{j_k}$, $j_k\in\left\{1,2,3\right\}$. 

The path $\gamma$ labelled by $(\cdots,g^{(k)},f^{(k)},g^{(k+1)},f^{(k+1)},g^{(k+2)},\cdots)$ consists of a series of geodesic segments: \begin{align*}
&p_k=\left [ \prod_{i=1}^{k}(g^{(i)}f^{(i)})\cdot(f^{(k)})^{-1}o,\prod_{i=1}^{k}(g^{(i)}f^{(i)})\cdot o\right ],\\
&q_k=\left [\prod_{i=1}^{k-1}(g^{(i)}f^{(i)})\cdot o,\prod_{i=1}^{k-1}(g^{(i)}f^{(i)})\cdot g^{(k)}o\right ],\\
&X_k=\prod_{i=1}^{k}(g^{(i)}f^{(i)})\cdot \left\langle h_{j_k}\right\rangle o.
\end{align*}
The two endpoints of $p_k$ lies in $X_k$.

\begin{itemize}
\item[\textbf{(LL1)}] For each $k$, \begin{align*} \mathrm{Len}(p_k)=d((f^{(k)})^{-1}o,o)=d(o,f^{(k)}o)>D.\end{align*}
\item[\textbf{(BP)}] For each $k$, \begin{align*}
\d{X_k}((p_k)_+,(p_{k+1})_-)&=\d{\left\langle h_{j_k}\right\rangle o}(o,g^{(k+1)}o)\leq \tau  ,
\end{align*}
and 
\begin{align*}
\d{X_k}((p_{k-1})_+,(p_{k})_-)&=\d{\left\langle h_{j_k}\right\rangle o}(o,(g^{(k)})^{-1}o)\leq \tau.
\end{align*}
\item[\textbf{(LL2')}] For each $k$, \begin{align*} d((p_k)_+,(p_{k+1})_-)=d(o,g^{(k+1)}o)>D.\end{align*}
\end{itemize}
\end{proof}
As a special case of \autoref{LEM: Modified Extension Lemma}, we usually consider the sequence $\left\{g^{(k)}\in G\mid k\in\mathbb{Z}\right\}$ to be cyclic.
\begin{corollary}\label{LEM: Construction of Perturbation}
For any three pairwise weakly independent contracting elements $h_1,h_2,h_3\in G$, let $\mathbb{X},\tau,D,F$ be as stated in \autoref{LEM: Modified Extension Lemma}. 
Then for any $g\in G$ such that $d(o,go)>D$, there exist $f\in F$ such that the path labelled by $(\cdots,g,f,g,f,g,f,\cdots)$ is a special $(D,\tau)$-admissible path with respect to $\mathbb{X}$.
\end{corollary}
\begin{corollary}\label{COR: Perturbation Length}
For any three pairwise weakly independent contracting elements $h_1,h_2,h_3\in G$, there exist $\epsilon, \hat{D}>0$ with the following property.

Fix any $F=\left\{f_1,f_2,f_3\right\}$ with $f_i\in\left\langle h_i\right\rangle,d(o,f_io)>\hat{D}$. Then, for any $g\in G$ satisfying $d(o,go)>\hat{D}$, we can choose $f\in F$ so that the following holds:
\begin{enumerate}[(1)]
\item\label{pl1} $\ell_d(gf)\geq d(o,gfo)-2\epsilon$.
\item\label{3.3.2} The map $\mathbb{Z}\to X,\,n\mapsto (gf)^no$ is a quasi-isometric embedding.
\item $gf$ and $fg$ are both contracting elements.
\end{enumerate}
\end{corollary}
\begin{remark}
Actually, with a bit more effort, the assumption ``$d(o,go)>\hat{D}$"  can be removed under taking more elements in $F$ as in \autoref{COR: Simultaneous Extension 15}. We state the above concise version for simplicity. On the other hand, the existence of weakly independent $h_1,h_2,h_3\in G$ is necessary. For example, consider the infinite dihedral group $\langle s, t|tst=s^{-1}, t^2=1\rangle$ acting on the real line. The contracting elements are exactly the infinite order elements of form $g=s^n$, but  $gt$ has order 2.   %

\end{remark}
\begin{proof}[Proof of \autoref{COR: Perturbation Length}]
Let $\tau$ be decided by \autoref{LEM: Modified Extension Lemma}, and $C$ be a contraction constant for the contracting system $\mathbb{X}=\left\{g\cdot \left\langle h_i\right\rangle\cdot o\mid g\in G, 1\leq i\leq 3\right\}$.

Let $D_1=D(\tau,C),\epsilon=\epsilon(\tau,C)$ be given by \autoref{PROP: Fellow Travel}, $\mathcal{R}=\mathcal{R}_{\tau,C}$ be given by \autoref{COR: Bounded Intersection}, and $D_2=D(\tau,C,\mathcal{R})$ be given by \autoref{PROP: Admissible Path Contracting}. Let $\hat{D}=\mathrm{max}(D_1+3\epsilon, D_2)$. Then $f\in F$ is chosen according to \autoref{LEM: Construction of Perturbation}.

\textbf{(1)} The path $\gamma$ labelled by $(g,f,g,f,\cdots)$ is a special $(D_1,\tau)$-admissible path with respect to $\mathbb{X}$.

For any $n\ge 1$, consider the subpath $\gamma_n$ labelled by $(g,f,g,f\cdots,g,f)$ with $(\gamma_n)_+=(gf)^no$, then according to \autoref{PROP: Fellow Travel}, $\alpha=[o,(gf)^no]$ $\epsilon$-fellow travels $\gamma_n$.

In particular, there exists a point $x\in \alpha$ such that $d((gf)^{n-1}o,x)\leq \epsilon$. By triangle inequality, $d(o,(gf)^no)=d(o,x)+d(x,(gf)^no)\ge d(o,(gf)^{n-1}o)+d(o,gfo)-2\epsilon$. By induction, $d(o,(gf)^no)\geq nd(o,gfo)-2(n-1)\epsilon$. 

Hence $\ell_d(gf)=\lim_{n\to\infty}\frac{d(o,(gf)^no)}{n}\geq\lim_{n\to\infty}\frac{nd(o,gfo)-2(n-1)\epsilon}{n}=d(o,gfo)-2\epsilon$.

\textbf{(2)} Since $gf$ acts by isometry, it suffices to show that there exist $\lambda_1, \lambda_2>0$ such that $\lambda_1 n\leq d(o,(gf)^no)\leq \lambda_2 n$, $\forall n\in \mathbb{N}^{+}$. 

In the proof of \textbf{(1)}, we have $d(o,(gf)^no)\geq (d(o,gfo)-2\epsilon)n$ and the geodesic segment $\beta=[o,gfo]$ $\epsilon$-fellow travels $[o,go]\cup[go,gfo]$. Hence there exists $y\in\beta$ such that $d(go,y)\leq \epsilon$, and $d(o,gfo)\geq d(o,y)\geq d(o,go)-d(go,y)> D_1+3\epsilon-\epsilon=D_1+2\epsilon$. 
So $\lambda_1=d(o,gfo)-2\epsilon>0$.

On the other hand $d(o,(gf)^no)\leq (d(o,go)+d(o,fo))n$, hence $\lambda_2=d(o,go)+d(o,fo)>0$.

\textbf{(3)} $\gamma$ is a special $(D_1,\tau)$-admissible path  with respect to  $\mathbb{X}$, then according to \autoref{COR: Bounded Intersection}, $\mathbb{X}(\gamma)$ has $\mathcal{R}$-bounded intersection. Also, $\gamma$ is a special $(D_2,\tau)$-admissible path  with respect to  $\mathbb{X}(\gamma)$, with $\mathrm{Len}_q$-constant $L=d(o,go)$ and $\mathrm{Len}_p$-constant $K=d(o,fo)$. According to \autoref{PROP: Admissible Path Contracting}, the path $\gamma$ is $c$-contracting for some $c$. Since $d_H(\left\{(gf)^no\mid n\in \mathbb{Z}\right\},\gamma)\leq d(o,go)+d(o,fo)<\infty$, according to \autoref{LEM: Finite Hausdorff Distance}, the orbit $\left\{(gf)^no\mid n\in \mathbb{Z}\right\}$ is also a contracting subset. Together with \textbf{(2)}, $gf$ is a contracting element. Notice that $fg$ is a conjugate of $gf$, hence $fg$ is also a contracting element.
\end{proof}

\begin{lemma}\label{LEM: Weakly Independent One In Three}
Suppose $h_1,h_2,h_3\in G$ are three weakly independent contracting elements. For any contracting element $g\in G$, there exists $h\in\left\{h_1,h_2,h_3\right\}$ such that $g$ and $h$ are weakly independent.
\end{lemma}
\begin{proof}
Suppose the three contracting subsets $\left\langle h_1\right\rangle o,\left\langle h_2\right\rangle o,\left\langle h_3\right\rangle o$ have contraction constant $C$ and $\mathcal{R}$-bounded intersection. Let $\tau=\tau(C,\mathcal{R})$ be given by \autoref{LEM: Finite Projection}.

For each $n\in \mathbb{Z}_{>0}$, let $\alpha=[o,g^no]$. 
Then at least two of $\diam (\pi_{\left\langle h_1\right\rangle o}(\alpha))$, $\diam (\pi_{\left\langle h_2\right\rangle o}(\alpha))$, $\diam (\pi_{\left\langle h_3\right\rangle o}(\alpha))$ are not greater than $\tau$.

Hence, there exist two elements $\left\{h_{i_1},h_{i_2}\right\}\subseteq \left\{h_1,h_2,h_3\right\}$ such that $\diam \left(\pi_{\left\langle h_{i_1}\right\rangle o}([o,g^no])\right)\leq \tau$ for infinitely many $n>0$, and $\diam \left(\pi_{\left\langle h_{i_2}\right\rangle o}([o,g^mo])\right)\leq \tau$ for infinitely many $m>0$.

For any $k\in \mathbb{Z}_{>0}$, we can find $n>k$ such that $\diam \left(\pi_{\left\langle h_{i_1}\right\rangle o}([o,g^no])\right)\leq \tau$. Let $\epsilon$ be given by \autoref{LEM: Fellow Travel Property} for the contracting element $g$. Then we can find $x\in [o,g^no]$ such that $d(g^ko,x)\leq \epsilon$. According to \autoref{LEM: 1 Lipschitz}, $\d{\left\langle h_{i_1}\right\rangle o}(g^ko,x)\leq \epsilon+4C$. Hence $\d{\left\langle h_{i_1}\right\rangle o}(g^ko,o)\leq \epsilon+4C+\tau$, for any $ k>0$.

Similarly, $\d{\left\langle h_{i_2}\right\rangle o}(g^ko,o)\leq \epsilon+4C+\tau$, for any $ k>0$.

A same proof adapted to $n\in \mathbb{Z}_{<0}$ and $\alpha=[o,g^no]$ implies that there exist two elements $\left\{h_{j_1},h_{j_2}\right\}\subseteq \left\{h_1,h_2,h_3\right\}$ such that $\d{\left\langle h_{j_1}\right\rangle o}(g^ko,o)\leq \epsilon+4C+\tau$ and $\d{\left\langle h_{j_2}\right\rangle o}(g^ko,o)\leq \epsilon+4C+\tau$, for any $ k<0$.

Choose $h\in\left\{h_{i_1},h_{i_2}\right\}\cap\left\{h_{j_1},h_{j_2}\right\}$, then $\diam\left( \pi_{\left\langle h\right\rangle o}\left\langle g\right\rangle o\right)\leq 2(\epsilon+4C+\tau)$.

By \autoref{PROP: Bounded Intersection Equals Bounded Projection}, $\left\langle g\right\rangle o$ and $\left\langle h\right\rangle o$ have bounded intersection. Hence $g$ and $h$ are weakly independent.
\end{proof}
\begin{corollary}\label{COR: New Contracting}
    If $G\curvearrowright X$ has contracting property, then for any contracting element $g\in G$, we can find another contracting element $g_{\ast}\in G$ such that $g$ and $g_{\ast}$ are weakly independent.
\end{corollary}

As a remark concluding this section, all of our results in \autoref{SEC: Preliminaries} and \autoref{SEC: Extension Lemma} can be generalised to $c$-roughly geodesic spaces which are not assumed to be proper. In such case, we define the coarsely closest point projection as follows: for any subset $A\subseteq X$, 
\begin{align*}
\pi_{A}(x)=\left\{y\in A\mid d(x,y)\leq d(x,A)+1\right\},
\end{align*}
and \autoref{DEF: Contracting Subset} is re-defined as follows:
\begin{definition}
Let $(X,d)$ be a $c$-roughly geodesic metric space. A subset $Y\subseteq X$ is called $C$-contracting for $C\geq 0$, if for any $c$-rough geodesic (segment) $\alpha$ in $X$ with $d(\alpha,Y)\geq C$, we have $\diam(\mathrm{\pi}_X(\alpha))\leq C$.
\end{definition}

\section{Simultaneously contracting elements}\label{SEC: Simultaneously contracting}

In this section, we introduce a key notion in relating two group actions with contracting property, called simultaneously contracting elements. Throughout this section, we consider a group $G$ acting isometrically on two geodesic metric spaces $(X_1,d_1)$ and $(X_2,d_2)$ with contracting property.%
\begin{definition}\label{Def: SCSystem}
    Denote $\mathcal{SC}(G)=\{ g\in G\mid g\text{ is contracting on both }X_1\text{ and }X_2\}$ the set of \textit{simultaneously contracting} elements on $X_1$ and $X_2$. We frequently choose  an \textit{weakly independent set} of simultaneously contracting elements, $F\subseteq \mathcal{SC}(G)$,  which means that  any pair of distinct elements in $F$ are weakly independent on both $X_1$ and $X_2$.
\end{definition}

\begin{proposition}\label{PROP: SC non-empty}
    $\mathcal{SC}(G)\neq \varnothing$.
\end{proposition}

Let us first introduce a useful lemma in the proof of \autoref{PROP: SC non-empty}.

\begin{definition}\label{DEF: Finite Index}
    Let $\Gamma$ be a group. A subset (not necessarily a subgroup) $E\subseteq \Gamma$ is called a \textit{finite-index subset} if there exists a finite collection of elements $g_1,\cdots, g_m\in \Gamma$ such that $\bigcup_{i=1}^{m}Eg_i=\Gamma$.
\end{definition}

\begin{lemma}\label{LEM: Elementary subset not finite index}
    Suppose a group $\Gamma$ acts by isometry on a geodesic metric space $(X,d)$ with contracting property. Let $f\in \Gamma$ be a contracting element and $o\in X$ be a basepoint. Let $A(f)=\left\{h\in \Gamma\mid h\left\langle f\right\rangle o \text{ and }\left\langle f\right\rangle o \text{  have no bounded intersection}\right\}$. Then $A(f)$ is not a finite-index subset of $\Gamma$.
\end{lemma}
\begin{proof}
     It suffices to show that for any finite collection of elements $g_1,\cdots, g_m\in \Gamma$, $\bigcup_{i=1}^{m}A(f)g_i\neq\Gamma$. Since obviously $A(f)=A(f)^{-1}$, this is equivalent to show that $\bigcup_{i=1}^{m}g_i^{-1}A(f)\neq\Gamma$.

     Let $f_{\ast}$ be a contracting element which is weakly independent with $f$ given by \autoref{COR: New Contracting}. According to \autoref{LEM: Non Elementary}, there exists $n\gg 0$ such that $\left\{(f^nf_{\ast}^n)^kf(f^nf_{\ast}^n)^{-k}\right\}_{k\in \mathbb{Z}}$ is a collection of pairwise weakly independent elements. We denote by $h=f^nf_{\ast}^n$.

     {By way of contradiction, let us assume} $\bigcup_{i=1}^{m}g_i^{-1}A(f)=\Gamma$. Then there exist $k_1>k_2>k_3>0$ and an element $g\in \left\{g_1,\cdots, g_m\right\}$ such that $\left\{h^{k_1},h^{k_2},h^{k_3}\right\}\subseteq g^{-1}A(f)$. This implies that $gh^{k_i}\left\langle f\right\rangle o$ and $\left\langle f\right\rangle o$ do not have bounded intersection for each $i=1,2,3$. 
     Up to an isometry, this is equivalent to {say that} $h^{k_i}\left\langle f\right\rangle o$ and $g^{-1}\left\langle f\right\rangle o$ do not have bounded intersection. Since $d_{H}(g^{-1}\left\langle f\right\rangle o,\left\langle g^{-1}fg\right\rangle o)<\infty$,  {we obtain that} $\left\langle h^{k_i}fh^{-k_i}\right\rangle o$ and $\left\langle g^{-1}fg\right\rangle o$ do not have bounded intersection, 
     which is to say $h^{k_i}fh^{-k_i}$ and $g^{-1}fg$ are not weakly independent for each $i=1,2,3$.

     However, $h^{k_i}fh^{-k_i}$ {for $i=1,2,3$} are pairwise weakly independent. By \autoref{LEM: Weakly Independent One In Three}, $g^{-1}fg$ must be weakly independent with one of them, which leads to a contradiction.
\end{proof}

\begin{remark}
    For a proper action, $A(f)$ coincides with $E(f)$, the elementary subgroup associated to $f$ (cf. \autoref{DEF: Elementary subgroup}). However, when the action is non-proper, $A(f)$ can strictly contain $E(f)$ as a subset. For example, $X$ is Gromov hyperbolic and  the fixed point $f^+\in \partial X$ of $f$ may  also be fixed by a parabolic isometry.
\end{remark}

We are now ready to prove \autoref{PROP: SC non-empty}.

\begin{proof}[Proof of \autoref{PROP: SC non-empty}]
    Fix basepoints $o_1\in X_1$ and $o_2 \in X_2$. According to \autoref{COR: Perturbation Length}, there exist a set of three elements $F=\left\{f_1,f_2,f_3\right\}\subseteq G$ and $\hat{D}>0$ such that for any $g\in G$ with $d_1(o_1,go_1)>\hat{D}$, we can find $f\in F$ such that $gf$ is contracting in $X_1$. We fix an element $h\in G$ which is contracting in $X_2$, and denote $A_2(h)=\left\{g\in G\mid g\left\langle h\right\rangle o_2 \text{ and }\left\langle h\right\rangle o_2 \text{ do NOT have bounded intersection}\right\}$.

    Now we choose an element $g\in G$ according to the following two different cases:

    \textbf{Case 1:} $\left\langle h\right\rangle o_1$ is an unbounded subset in $X_1$. Since $A_2(h)$ is not a finite-index subset by \autoref{LEM: Elementary subset not finite index}, we choose $g\in G\setminus\left (\bigcup_{i=1}^{3}A_2(h)f_i^{-1}\right )$.

    \textbf{Case 2:} $\left\langle h\right\rangle o_1$ is a bounded subset in $X_1$. Let $M=\sup\left\{d_1(o_1,h^no_1)\mid n\in\mathbb{Z}\right\}$. It is easy to see that $E=\left\{ g\in G\mid d_1(o_1,go_1)>M+\hat{D}\right\}$ is a finite-index subset of $G$; for instance, choose an element $g_{\ast}\in G$ with $d_1(o_1,g_{\ast}o_1)>2(M+\hat{D})$ and obviously $G=E\cup Eg_{\ast}$. On the other hand, since $A_2(h)$ is not a finite-index subset by \autoref{LEM: Elementary subset not finite index}, $\bigcup_{i=1}^{3}A_2(h)f_i^{-1}$ is not a finite-index subset either. Therefore, $E\not\subseteq \bigcup_{i=1}^{3}A_2(h)f_i^{-1}$. We choose $g\in E\setminus \left (\bigcup_{i=1}^{3}A_2(h)f_i^{-1}\right )$.

    In either case, there exists an infinite increasing sequence of positive integers $\left\{n_k\right\}_{k=1}^{\infty}$ such that $d_1(o_1,h^{n_k}go_1)>\hat{D}$ for all $k$. As is mentioned in the beginning, for each $k$ there exists an element $f\in F$ such that $h^{n_k}gf$ is contracting in $X_1$. We will also show that $h^{n_k}gf$ is contracting in $X_2$ for $k\gg 0$.

    By construction, $gf_i\notin A_2(h)$. Therefore, $gf_i\left\langle h\right\rangle o_2$ and $\left\langle h\right\rangle o_2$ have $\nu_i$-bounded intersection ($i=1,2,3$). Let $\nu=\max\limits_{1\le i\le 3} \nu_i:[0,+\infty)\to [0,+\infty)$, $C\ge 0$ be a contraction constant of $\left\langle h\right\rangle o_2$ and $\tau=\diam\left (\pi_{\left\langle h\right\rangle o_2}(\left\{o_2,gf_io_2,(gf_i)^{-1}o_2\mid 1\le i\le 3\right\}\right )$. Let $D_0=D(\tau,C,\nu)$ and $\epsilon=\epsilon(\tau,C,\nu)$ be decided by \autoref{PROP: Fellow Travel}, $D_1=D(\tau,C,\nu)$ and $\mathcal{R}=\mathcal{R}_{\tau,C,\nu}$ be decided by \autoref{COR: Bounded Intersection}, and let $D_2=D(\tau,C,\mathcal{R}_{\tau,C,\nu})$ be decided by \autoref{PROP: Admissible Path Contracting}. Denote $D=\max\left\{D_0+3\epsilon,D_1,D_2\right\}$.

    Since $\left\langle h\right\rangle o_2$ is a quasi-geodesic, we have $d_2(o_2,h^{n_k}o_2)>D$ for $k\gg 0$. For any such $k$, consider the path $\gamma$ in $X_2$ labelled by $(\cdots, h^{n_k},gf,h^{n_k},gf,\cdots)$. Then $\gamma$ is a $(D,\tau)$-admissible path  with respect to  $\mathbb{X}=\left\{u\left\langle h\right\rangle o_2\mid u\in G\right\}$ and the bounded intersection gauge $\nu$. By \autoref{COR: Bounded Intersection}, $\mathbb{X}(\gamma)$ has $\mathcal{R}$-bounded intersection. Since $\gamma$ is a periodic path, it must have some $\mathrm{Len}_p$ constant and $\mathrm{Len}_q$ constant. Therefore, according to \autoref{PROP: Admissible Path Contracting}, $\gamma$ is a contracting subset in $X_2$. Notice that $d_H(\left\langle h^{n_k}gf\right\rangle o_2,\gamma)<+\infty$.
    Hence according to \autoref{LEM: Finite Hausdorff Distance}, $\left\langle h^{n_k}gf\right\rangle o_2$ is a contracting subset in $X_2$.
    
    In addition, $\gamma$ has $\epsilon$-fellow travel property by \autoref{PROP: Fellow Travel}. A similar proof as in \autoref{COR: Perturbation Length} (\ref{3.3.2}) shows that the map $n\in \mathbb{Z}\mapsto (h^{n_k}gf)^no_2\in X_2$ is a quasi-isometric embedding. 
    Therefore, $h^{n_k}gf$ is also a contracting element in $X_2$.

    Thus, $h^{n_k}gf\in \mathcal{SC}(G)$ for $k\gg 0$, and $\mathcal{SC}(G)\neq \varnothing$.
\end{proof}

\begin{lemma}\label{LEM: Simultaneous Contracting}
There exists a  weakly independent set of infinitely many simultaneously contracting elements in $\mathcal{SC}(G)$.
\end{lemma}

\begin{proof}
Fix basepoints $o_1\in X_1$, $o_2\in X_2$. 
We first construct two elements in $\mathcal{SC}(G)$ which are weakly independent in both $X_1$ and $X_2$.

Suppose $f\in \mathcal{SC}(G)$.
According to \autoref{COR: New Contracting}, we can find an element $g_{\ast}\in G$ which is contracting and weakly independent with $f$ in $G\curvearrowright X_1$.

According to \autoref{LEM: Non Elementary} (\ref{_NE2}), there exists $n\gg 0$ such that $\theta_k= (f^n g_{\ast}^n)^kf(f^ng_{\ast}^n)^{-k}$ $({k \in \mathbb{Z}})$ are pairwise weakly independent in $G\curvearrowright X_1$.

Similarly, we can find an element $h_{\ast}\in G$ which is contracting and weakly independent with $f$ in $G\curvearrowright X_2$. 
For $n\gg  0$, $\eta_k=(f^n h_{\ast}^n)^kf(f^nh_{\ast}^n)^{-k}$ $({k \in \mathbb{Z}})$ are pairwise weakly independent in $G\curvearrowright X_2$.

Notice that $\theta_k$, $\eta_k$ are all conjugate to $f$, so they are simultaneously contracting in $X_1$ and $X_2$.

Since $\theta_1,\cdots,\theta_5$ are pairwise weakly independent contracting elements in $X_1$, according to \autoref{LEM: Weakly Independent One In Three}, for any $1\leq k\leq 5$ and $1\leq i_1< i_2< i_3\leq 5$, there exist one element in $\left\{\theta_{i_1},\theta_{i_2},\theta_{i_3}\right\}$ which is weakly independent with $\eta_k$ in $X_1$.

Similarly, for any $1\leq k\leq 5$ and $1\leq i_1< i_2< i_3\leq 5$, there  {exists} one element in $\left\{\eta_{i_1},\eta_{i_2},\eta_{i_3}\right\}$ which is weakly independent with $\theta_k$ in $X_2$.

Hence there are at least $5\times 3=15$ pairs $(\theta_k,\eta_l)$ weakly independent in $X_1$ and also at least $5\times 3=15$ pairs $(\theta_k,\eta_l)$ weakly independent in $X_2$. However, there are altogether only $5\times 5=25$ pairs $(\theta_k,\eta_l)$, $1\leq k,l\leq 5$. Hence, by Pigeonhole Principle, there must be a pair $(\theta, \eta)$ weakly independent in both $X_1$ and $X_2$.

Now we can construct the weakly independent set of infinitely many simultaneously contracting elements.

By \autoref{LEM: Non Elementary} (\ref{_NE2}), there exists $n\gg 0$ such that $((\theta^n \eta^n)^k\theta(\theta^n\eta^n)^{-k})_{k \in \mathbb{Z}}$ are pairwise weakly independent and simultaneously contracting on both $X_1$ and $X_2$, and hence $\left\{(\theta^n \eta^n)^k\theta(\theta^n\eta^n)^{-k}\mid k \in \mathbb{Z}\right\}$ is the desired weakly independent set of contracting elements in $\mathcal{SC}(G)$. 
\end{proof}

Through  the Pigeonhole Principle, we establish a simultaneous version of \autoref{COR: Perturbation Length}.

\begin{lemma}\label{Lem: Simultaneous Extension Lemma}
    Fix basepoints $o_1\in X_1$ and $o_2\in X_2$.
    Given a weakly independent set of five elements $F_0=\left\{h_1,h_2,h_3,h_4,h_5\right\}\subseteq \mathcal {SC}(G)$, there exists $D>0,\epsilon>0$ with the following property.

    Form the set $F=\{f_1,f_2,f_3,f_4,f_5\}$ constructed from $F_0$ with $f_j\in\left\langle h_j\right\rangle$ and $d_i(o_i,f_jo_i)>D$ ($i=1,2,\,1\le j\le 5$). Then for any $g\in G$ satisfying $d_i(o_i,go_i)>D$ ($i=1,2$), we can choose $f\in F$ so that the following holds.
    
    \begin{enumerate}[(1)]
        \item $\ell_{d_i}(gf)\ge d_i(o_i,gfo_i)-2\epsilon$, for $i=1,2$.
        \item $gf,fg\in \mathcal{SC}(G)$.
    \end{enumerate}
\end{lemma}
\begin{proof}
    For each $1\leq i \leq 2$, and each triple $1\leq j<k<l\leq 5$, consider the action $G\curvearrowright X_i$ and let $\epsilon^{(i)}_{j,k,l}>0,\hat{D}^{(i)}_{j,k,l}>0$ be decided by \autoref{COR: Perturbation Length} for the three contracting elements $h_j,h_k,h_l$. Let $\epsilon=\max\epsilon^{(i)}_{j,k,l}$ and $D=\max\hat{D}^{(i)}_{j,k,l}$.

    Let $F=\{f_1,f_2,f_3,f_4,f_5\}$ be the weakly independent set of simultaneously contracting elements constructed with $f_j\in\left\langle h_j\right\rangle$ and $d_i(o_i,f_jo_i)>D$ ($i=1,2,\,1\le j\le 5$), and let $g\in G$ satisfy $d_i(o_i,go_i)>D$ ($i=1,2$). Then according to \autoref{COR: Perturbation Length}, for each $1\leq j<k<l\leq 5$, we can find $f\in \left\{f_j,f_k,f_l\right\}$ such that $\ell_{d_1}(gf)\ge d_1(o_1,gfo_1)-2\epsilon$ and $gf,fg$ are contracting elements in $G\curvearrowright X_1$. In other words, there are at least three elements in $F$ such that the above condition holds. Similarly, there are also at least three elements in $F$ such that $\ell_{d_2}(gf)\ge d_2(o_2,gfo_2)-2\epsilon$ and $gf,fg$ are contracting elements in $G\curvearrowright X_2$.

    By Pigeonhole Principle, there exist at least one element $f\in F$ such that $\ell_{d_i}(gf)\ge d_i(o_i,gfo_i)-2\epsilon$ for $i=1,2$, and $gf,fg$ are contracting elements in both $G\curvearrowright X_1$ and $G\curvearrowright X_2$.
\end{proof}

\begin{corollary}\label{COR: Simultaneous Extension 15}
    There exist $\epsilon >0$ and a set $S\subseteq \mathcal{SC}(G)$ of at most $15$ elements such that for any $g\in G$ we can choose $f\in S$ so that the following holds.
    \begin{enumerate}
        \item $\ell_{d_i}(gf)\ge d_i(o_i,gfo_i)-2\epsilon$, for $i=1,2$.
        \item $gf,fg\in \mathcal{SC}(G)$.
    \end{enumerate}
\end{corollary}
\begin{proof}
    Fix basepoints $o_1\in X_1$ and $o_2\in X_2$.
    Let $F_0=\left\{h_1,h_2,h_3,h_4,h_5\right\}\subseteq \mathcal {SC}(G)$ be a weakly independent set of five elements, and let $D>0,\epsilon>0$ be decided by \autoref{Lem: Simultaneous Extension Lemma}.

    According to \autoref{LEM: Non Elementary} (\ref{_NE1}), there exists $N\gg 0$ such that $h_j^nh_k^m\in \mathcal{SC}(G)$ for any $n,m\ge N$ and $1\le j,k\le 5$. For each $1\le j\le 5$, pick $n_j\ge N$ such that $\min\left\{d_1(o_1,h_j^{n_j}o_1),d_2(o_2,h_j^{n_j}o_2)\right\}>2D$ and let $f_j=h_j^{n_j}$.

    Denote $K_i=d_i(o_i,f_1o_i)$ for $i=1,2$. Pick $m\ge 1$ such that $d_i(o_i,f_1^mo_i)>2D+K_i$ for $i=1,2$.

    Let $F=\left\{f_1,f_2,f_3,f_4,f_5\right\}$ and then $S=F\cup (f_1\cdot F) \cup (f_1^m \cdot F)$. By construction, $S\subseteq \mathcal{SC}(G)$.

    Now, we consider four possibly overlapping cases that encompass all $g\in G$ and find the desired element $f\in S$.

    \textbf{Case 1:} $d_1(o_1,go_1)>D$ and $d_2(o_2,go_2)>D$.

    According to \autoref{Lem: Simultaneous Extension Lemma}, there exists $f\in \left\{f_1,f_2,f_3,f_4,f_5\right\}\subseteq S$ such that $\ell_{d_i}(gf)\ge d_i(o_i,gfo_i)-2\epsilon$ ($i=1,2$) and $gf,fg\in \mathcal{SC}(G)$.

    \textbf{Case 2:} $d_1(o_1,go_1)>D+K_1$ and $d_2(o_2,go_2)\le D$.

    By triangle inequality, $d_1(o_1,gf_1o_1)\ge d_1(o_1,go_1)-d_1(o_1,f_1o_1)>D$, and $d_2(o_2,gf_1o_2)\ge d_2(o_2,f_1o_2)-d_2(o_2,go_2)>2D-D=D$. Therefore, by \autoref{Lem: Simultaneous Extension Lemma}, there exists $\tilde{f}\in F=\left\{f_1,f_2,f_3,f_4,f_5\right\}$ such that $\ell_{d_i}(gf_1\tilde{f})\ge d_i(o_i,gf_1\tilde{f}o_i)-2\epsilon$ ($i=1,2$) and $gf_1\tilde{f}\in \mathcal{SC}(G)$. Let $f=f_1\tilde{f}\in f_1 \cdot F\subseteq S$. Then $\ell_{d_i}(gf)\ge d_i(o_i,gfo_i)-2\epsilon$ ($i=1,2$) and $gf\in \mathcal{SC}(G)$. As $fg$ is a conjugate of $gf$, we have $fg\in \mathcal{SC}(G)$.

    \textbf{Case 3:} $d_1(o_1,go_1)\le D$ and $d_2(o_2,go_2)> D+K_2$.

    This case follows a similar proof with \textbf{Case 2}.

    \textbf{Case 4:} $d_1(o_1,go_1)\le D+K_1$ and $d_2(o_2,go_2)\le D+K_2$.

    For each $i=1,2$, by triangle inequality, $d_i(o_i,gf_1^mo_i)\ge d_i(o_i,f_1^mo_i)-d_i(o_i,go_i)>(2D+K_i)-(D+K_i)=D$. Therefore, by \autoref{Lem: Simultaneous Extension Lemma}, there exists $\tilde{f}\in F$ such that $\ell_{d_i}(gf_1^m\tilde{f})\ge d_i(o_i,gf_1^m\tilde{f}o_i)-2\epsilon$ ($i=1,2$) and $gf_1^m\tilde{f}\in \mathcal{SC}(G)$. Let $f=f_1^m\tilde{f}\in f_1^{m}\cdot F\subseteq S$. Then $\ell_{d_i}(gf)\ge d_i(o_i,gfo_i)-2\epsilon$ ($i=1,2$) and $gf\in \mathcal{SC}(G)$. As $fg$ is a conjugate of $gf$, we have $fg\in \mathcal{SC}(G)$.
\end{proof}

\section{Marked length spectrum rigidity}\label{SEC: Rigidity}
Throughout this section,  suppose that a group $G$ acts isometrically on two geodesic metric spaces $(X_1,d_1)$ and $(X_2,d_2)$ with contracting property (c.f. \autoref{DEF: Contracting Property}).

We say two actions $G\curvearrowright X_1$ and $G\curvearrowright X_2$ have the \textit{same marked length spectrum} on a subset $E\subseteq G$, if \begin{align*}
            \ell_{d_1}(g)=\ell_{d_2}(g),\,\forall g\in E.
        \end{align*}

Questions \ref{Que: MLSR} and \ref{Que: MLSR from subset}  on marked length spectrum rigidity  in Introduction can now be recasted as follows.

For which subsets $E\subseteq G$, the same marked length spectrum on $E$ will imply that $$\sup_{g,h\in G}\abs{d_1(go_1,ho_1)-d_2(go_2,ho_2)}<+\infty,$$ or equivalently the orbit map $\rho:Go_1\to Go_2$ given by $go_1\mapsto go_2$ is well-defined up to a bounded error and is a rough isometry? 

For any such subset $E$, we say that $G$ has \textit{marked length spectrum rigidity from $E$}, or simply \textit{marked length spectrum rigidity} when $E=G$. 

Addressing this question is the  goal of the subsequent sections.  We start with the case $E=G$ in this section.

\subsection{MLS rigidity for actions with contracting property}
\begin{theorem}\label{THM: Main Rigidity}
$G$ has marked length spectrum rigidity: if $\ell_{d_1}(g)=\ell_{d_2}(g)$ for all $g\in G$, then  for any two $o_1\in X_1, o_2\in X_2$,  $\sup_{g,h\in G}\abs{d_1(go_1,ho_1)-d_2(go_2,ho_2)}<+\infty$.

In particular, if the action of $G$ on $X_1$ and $X_2$ are both cobounded, then  there exists a $G$-coarsely equivariant rough isometry from $X_1$ to $X_2$. 

\end{theorem}

In fact, \autoref{THM: Main Rigidity} can be derived directly from the following more precise inequality:
\begin{theorem}\label{COR: Rigidity from SC}
    Assume for some $\lambda_1,\lambda_2\geq 0$, the marked length spectrum on $\mathcal{SC}(G)$ satisfies the following inequality: \begin{align}\label{SCInequality}
        \lambda_1  \ell_{d_2}(g)\leq  \ell_{d_1}(g)\leq \lambda_2 \ell_{d_2}(g),\,\forall g\in \mathcal{SC}(G).
    \end{align}  
    
    Then for any fixed points $o_1\in X_1$ and $o_2\in X_2$, there exists $C\geq 0$ such that $$\lambda_1 d_2(go_2,ho_2)-C\leq d_1(go_1,ho_1)\leq \lambda_2 d_2(go_2,ho_2)+C, \,\forall g,h\in G.$$
\end{theorem}

\begin{proof}

We first deal with the right-hand side of this inequality: $d_1(go_1,ho_1)\le \lambda_2d_2(go_2,ho_2)+C$. 
Since $G$ acts isometrically, it suffices to show that there exists $C\ge 0$ such that $d_1(o_1,go_1)\le \lambda_2 d_2(o_2,go_2)+C$ for all $g\in G$.

According to \autoref{COR: Simultaneous Extension 15}, there exist $\epsilon>0$ and a finite set $S\subseteq \mathcal{SC}(G)$ such that for any $g\in G$ we can find $f\in S$ so that $gf\in \mathcal{SC}(G)$ and 
\begin{equation}\label{5.3.1EQ1}
d_i(o_i,gfo_i)\leq \ell_{d_i}(gf)+2\epsilon.    
\end{equation}

Let $C_0=\max\left\{d_i(o_i,fo_i)\mid f\in S,\,i=1,2\right\}$. Then for any $g\in G$, we have:
\begin{align*}
d_1(o_1,go_1)&\leq d_1(o_1,gfo_1)+d_1(go_1,gfo_1) & (\text{$\Delta$-Ineq.})\\
&\leq \ell_{d_1}(gf)+2\epsilon+C_0 & (\ref{5.3.1EQ1})\\&\leq \lambda_2\ell_{d_2}(gf)+2\epsilon+C_0 &  (\ref{SCInequality})\\
&\leq \lambda_2d_2(o_2,gfo_2)+2\epsilon+C_0& (\text{\autoref{RMK: MLS Basics}\,(\ref{MLSBasics2})})\\
&\leq \lambda_2(d_2(o_2,go_2)+d_2(go_2,gfo_2))+2\epsilon+C_0 &(\text{$\Delta$-Ineq.})\\&\leq \lambda_2d_2(o_2,go_2)+(\lambda_2+1)C_0+2\epsilon.&
\end{align*}

Let $c=(\lambda_2+1)C_0+2\epsilon$. Then $d_1(o_1,go_1)\le \lambda_2d_2(o_2,go_2)+c$ for all $g\in G$.

For the left-hand side of the inequality, if $\lambda_1=0$ the conclusion is trivial. We may assume $\lambda_1>0$, so $\ell_{d_2}(g)\le \lambda^{-1}\ell_{d_1}(g)$ for all $g\in G$. A similar calculation shows that there exists $c'\geq 0$ such that $d_2(o_2,go_2)\leq \lambda_1^{-1} d_1(o_1,go_1)+c'$, for all $ g\in G$.

Letting $C=\max\left\{c,\lambda_1c'\right\}$  finishes the proof.
\end{proof}

As a direct application of \autoref{COR: Rigidity from SC}, we can deduce a MLS rigidity from a finite-index subset.

\begin{corollary}\label{Cor: MLSR from a finite-index subset}
    Suppose $E\subseteq G$ is a finite-index subset (cf. \autoref{DEF: Finite Index}). Then $G$ has marked length spectrum rigidity from $E$.
\end{corollary}
\begin{proof}
    Suppose $g_1,\cdots, g_m\in G$ such that $\bigcup_{i=1}^{m}Eg_i=G$, and $\ell_{d_1}(g)=\ell_{d_2}(g)$ for all $g\in E$. 
    For any element $f\in G$ contracting in $X_1$, let $A_1(f)=\{ g\in G\mid g \left\langle f\right\rangle o_1\text{ and }\left\langle f\right\rangle o_1\text{ have no bounded intersection}\}$.

    According to \autoref{LEM: Elementary subset not finite index}, there exists an element $h\in G\setminus \bigcup_{i=1}^{m}  g_i^{-1}A_1(f)$. Then $g_ih\notin A_1(f)$ for each $1\le i \le m$, and in addition, $\bigcup_{i=1}^{m}Eg_ih=(\bigcup_{i=1}^{m}Eg_i)h=G$.
    
    By finiteness, there exists an element $g\in \left\{(g_1h)^{-1},\cdots, (g_mh)^{-1}\right\}$ and an infinite increasing sequence of positive integers $\left\{n_k\right\}_{k=1}^{+\infty}$ such that $f^{n_k}g\in E$. Moreover, $g\notin A_1(f)$.

    Therefore, $\left\langle f\right\rangle o_1$ and $g\left\langle f\right\rangle o_1$ have $\mathcal{R}$-bounded intersection for some $\mathcal{R}:[0,+\infty)\to[0,+\infty)$. Let $C\ge 0$ be a contraction constant of $\left\langle f\right\rangle o_1$, and let $\tau= \diam \left (\pi_{\left\langle f\right\rangle o_1}(\left\{ o_1, go_1,g^{-1}o_1\right\})\right )$. Let $D=D(\tau,C,\mathcal{R})$ and $\epsilon=\epsilon(\tau,C,\mathcal{R})$ be determined by \autoref{PROP: Fellow Travel}.

    Since $\left\langle f\right\rangle o_1$ is a quasi-geodesic, for $k\ge k_0$ we must have $d_1(o_1,f^{n_k}o_1)>D$. For any such $k$, consider the path $\gamma$ labelled by $(\cdots, f^{n_k}, g, f^{n_k}, g,\cdots )$. Then $\gamma$ is a $(D,\tau)$-admissible path  with respect to  $\mathbb{X}=\left\{u\left\langle f\right\rangle o|\, u\in G\right\}$ and the bounded intersection gauge $\mathcal{R}$. By \autoref{PROP: Fellow Travel}, $\gamma$ has $\epsilon$-fellow travel property.

    Consequently,
    \begin{equation}\label{FIEQU1}
    \begin{aligned}
        &d_1(o_1,f^{n_k}o_1)\leq d_1(o_1,f^{n_k}go_1)+\epsilon;\\
        &\ell_{d_1}(f^{n_k}g)\ge d_1(o_1,f^{n_k}go_1)-2\epsilon.
    \end{aligned}
    \end{equation}

    Hence,
    \begin{align*}
        d_1(o_1,f^{n_k}o_1)\leq &\ell_{d_1}(f^{n_k}g)+3\epsilon& (\ref{FIEQU1})&\\
        =&\ell_{d_2}(f^{n_k}g)+3\epsilon& &\\
        \leq & d_2(o_2,f^{n_k}go_2)+3\epsilon&(\text{\autoref{RMK: MLS Basics}\,(\ref{MLSBasics2})})&\\
        \leq & d_2(o_2,f^{n_k}o_2)+d_2(o_2,go_2)+4\epsilon &(\text{$\Delta$-Ineq.})&.
    \end{align*}

    Therefore,
    \begin{align*}
        \ell_{d_1}(f)=\lim_{k\to +\infty}\frac{d_1(o_1,f^{n_k}o_1)}{n_k}\leq \limsup_{k\to +\infty}\frac{d_2(o_2,f^{n_k}o_2)+d_2(o_2,go_2)+4\epsilon}{n_k}=\ell_{d_2}(f).
    \end{align*}

    Thus $\ell_{d_2}(f)\ge \ell_{d_1}(f)$. 
    Similarly, for any element $f'\in G$ contracting in $X_2$, we have $\ell_{d_1}(f')\ge \ell_{d_2}(f')$. 
    Hence, $\ell_{d_1}(f)=\ell_{d_2}(f)$ for all $f\in \mathcal{SC}(G)$.
    Therefore, according to %
    \autoref{COR: Rigidity from SC}, the orbit map $\rho$ must be a rough isometry. 
\end{proof}

\subsection{MLS rigidity for non-geodesic metrics}
In this subsection, we present some results on MLS rigidity for some non-geodesic metrics as in \autoref{Que: pull-back metric}. For this purpose, let us fix a reference action $G\act (X,d)$ on a geodesic metric space with contracting property. 
Consider three weakly independent contracting elements: $h_1,h_2,h_3\in G$.
Given a basepoint $o\in X$, denote  $\mathcal F_o=\left\{g\left\langle h_i\right\rangle o\mid g\in G,1\le i\le 3\right\}$ the contracting system of $G$-translated axes, and let $C\ge 0$ be a contraction constant of $\mathcal F_o$.

We introduce a class of transversal points relative to $\mathcal F_o$. This is motivated  by the corresponding notion in  relatively hyperbolic groups, which we briefly recall as follows. 

Let $(G,\{H_i\}_{i=1}^n)$ be a relatively hyperbolic group as in \autoref{RHGDefn}. Consider the $C$-contracting system $\{gH_i: g\in G, 1\le i\le n\}$ for some $C>0$. A notion of transition points is very helpful in understanding the geometry of the Cayley graph of $G$. Roughly speaking, a point $x$ on a (word) geodesic $\gamma$ is called $(C,L)$-deep in some $gH_i$ if the $L$-neighbourhood $B(x,L)\cap \gamma$ is contained in $N_C(gH_i)$. A point $x\in \gamma$ which is not $(C,L)$-deep in any  $gH_i$ is called a $(C,L)$-transition point. 
We refer to \cite{Hru10} for more details and relevant discussions.

Let us return to the above setup of a general metric space, equipped with a $C$-contracting system $\mathcal F_o$. We define an analogous notion to transition points, which is also in line with the vertices of an admissible path (see \autoref{TransOnAdmisPath} below).
\begin{definition}
Let $\gamma$ be a geodesic segment and $L>0$. We say a point $x\in \gamma$ is \textit{($\mathcal F_o,L$)-transversal} if there exists $Z\in \mathcal F_o$ such that $diam(\gamma\cap  N_C(Z))>L$ and $x$ is the entry or the exit points of $\gamma\cap  N_C(Z)$.  Denote by $\mathrm{Trans}_{L}(\gamma, \mathcal F_o)$ the set of ($\mathcal F_o,L$)-transversal points.
\end{definition} 
In a relatively hyperbolic group $(G,\{H_i\}_{i=1}^n)$, any hyperbolic element $h$ (i.e. not conjugated into any $H_i$ and of infinite order) is contracting with respect to the action on the Cayley graph.
\begin{lemma}\label{LEM: Transition}
    Let $(G,\{H_i\}_{i=1}^n)$ be a relatively hyperbolic group, and $F\subset G$ be a weakly independent set of three hyperbolic elements. Then there exists $L> 0$ such that any $(\mathcal F_1, L)$-transversal point is a $(3C,L)$-transition point defined in the sense of maximal parabolic subgroups. 
\end{lemma}
\begin{proof}
Denote $F=\left\{h_1,h_2,h_3\right\}$.
For each $1\le j\le 3$, by \cite{Osin06}, the maximal elementary subgroup $E(h_j)$ could be adjoined into the peripheral structure as a maximal parabolic group.  That is,  $(G,\{H_i\}_{i=1}^n\cup \{E(h_j)\})$ is still  relatively hyperbolic. 

\cite[Corollary 5.7]{GP13} showed that $\left\{gH_i\mid g\in G,1\le i\le n\right\}\cup \left\{gE(h_j)\mid g\in G\right\}$ is a contracting system with $\mathcal{R}_j$-bounded intersection. Let $\mathcal{R}=\max \mathcal{R}_j$, and let $L>\mathcal{R}(3C)$.  

If $x\in \mathrm{Trans}_L(\mathcal F_o,\gamma)$, then there exist $Z=g_0\left\langle h_j\right\rangle \in \mathcal F_1$ and a point $y\in \gamma$ such that $d(x,y)\ge L$ and $x,y\in N_C(Z)$. \autoref{LEM: Quasi Convex} implies that $[x,y]_{\gamma} \subset N_{3C}(Z)$. Suppose by contrary that $x\in \gamma$ is $(3C, L)$-deep in some $gH_i$. Then $B(x,L)\cap \gamma \subset N_{3C}(gH_i)$. Let $z\in [x,y]_\gamma$ such that $d(x,z)=L$, then $[x,z]_\gamma \in N_{3C}(gH_i)\cap N_{3C}(Z)$. However, $\diam (N_{3C}(gH_i)\cap N_{3C}(Z)) \le \mathcal R_j(3C)<L $, which leads to a contradiction.
\end{proof}

\begin{definition}\label{CoarseAddtiveDef}
Let $d_G$ be a left-invariant metric on $G$. For some $L>0$, we say that $d_{G}$ is \textit{coarsely additive} along $(\mathcal F_o,L)$-transversal points if for any $g,f, h\in G$ so that $fo$ is $r$-close to a $(\mathcal F_o,L)$-transversal point of $[go,ho]$ for some $r>0$, we have 
\begin{equation}\label{CoarseAdditiveEQ}
\begin{aligned}    
|d_G(g,h)- d_G(g,f)-d_G(f,h)|\le R(r) 
\end{aligned}
\end{equation}    
where $R$ depends only on $r$.
\end{definition}

\begin{remark}
If $G$ acts properly on $X$, we only need to verify the above   \autoref{CoarseAdditiveEQ} for $r=C$, from which the general case would follow. Indeed, by definition, a $(\mathcal F_o,L)$-transversal point is $C$-close to a point $f'o\in Go$. If $fo$ is $r$-close to an $(\mathcal F_o,L)$-transversal point of $[go,ho]$,  then  $d(fo,f'o)\le r+C$ for some $f'o\in Go$. The set $\{g\in G: d(o,go)\le r+C\}$ is finite by the proper action $G\act X$, so there exists a constant $r_1$ depending only on $r+C$ so that $d_G(f,f')\le r_1$. Hence,  if \autoref{CoarseAdditiveEQ} holds for $g,f',h$ with $R(C)$, we obtain $|d_G(g,h)- d_G(g,f)-d_G(f,h)|\le R(C)+2r_1:=R(r)$. 
\end{remark}%

Recall that $\mathcal F_o$ is a $C$-contracting system. The following explains the relation of the transversal points with admissible paths. 
\begin{lemma}\label{TransOnAdmisPath}
For given $\tau>0$ and $L>0$ and a $C$-contracting system $\mathcal F_o$, there exist $D=D(\tau,L,C)>0,r=r(\tau,L,C)> 0$ with the following property.
Let $\gamma=\cdots q_ip_iq_{i+1}p_{i+1}\cdots $ be a finite special $(D,\tau)$-admissible path with respect to $\mathcal F_o$, and $\alpha$ be any geodesic with the same endpoints as $\gamma$. Then, for each $i$ such that $\mathrm{Len}(p_i)>D$, we have $\left\{(p_i)_-,(p_i)_+\right\}\subseteq N_r(Trans_{L}(\alpha,\mathcal F_o))$.    
\end{lemma}
\begin{proof}
We follow the same convention as in \autoref{CONV: Nearest Projection}. 
Let $D_1=D(\tau,C)$, $ \epsilon=\epsilon(\tau,C)$ and $B=B(\tau, C)$ be given by \autoref{PROP: Fellow Travel}. Pick $D=\max\left\{D_1,4\epsilon+4C+L\right\}$, and $r>\max\left\{B+2C,2\epsilon+2C\right\}$. We show that any finite special $(D,\tau)$-admissible path $\gamma$  with respect to  $\mathcal F_o$ satisfies the property.

Suppose the segment $p_i$ is associated to $X_i\in \mathcal F_o$, and $\alpha=[\gamma_-,\gamma_+]$.
We will first show that $\alpha \cap N_C(X_i)\neq \varnothing$ and $\diam (\alpha \cap N_C(X_i))>L$.

According to \autoref{PROP: Fellow Travel}, there exist two linearly ordered points $x,y\in \alpha$ such that $d((p_i)_-,x)\le \epsilon $ and $d((p_i)_+,y)\le \epsilon$. As $\mathrm{Len}(p_i)>D$, by triangle inequality $d(x,y)>D-2\epsilon$. In addition, $d(x,\pi_{X_i}(x))\le \epsilon$ and $d(y,\pi_{X_i}(y))\le \epsilon$. Therefore, $\d{X_i}(x,y)\ge d(x,y)-2\epsilon>D-4\epsilon\ge 4C+L> C$.
By definition of a $C$-contracting subset, $[x,y]_{\alpha}\cap N_C(X_i)\neq\varnothing$. In addition, let $x'(y')$ be the entry (exit) point of $[x,y]_{\alpha}$ in $N_C(X_i)$. When $x'\neq x$, then $d(x',x)\le d(x,X_i)+d(x',X_i)+\d{X_i}(x,x')\le \epsilon + 2C$. Similarly, $d(y',y)\le \epsilon + 2C$. By triangle inequality, $d(x',y')\ge d(x,y)-2(\epsilon + 2C)>D-2\epsilon -2(\epsilon + 2C)\ge L$, which is to say $\diam (\alpha \cap N_C(X_i))>L$.

Next we will show that $(p_i)_-$ (resp. $(p_i)_+$) is $r$-close to the entry point (resp. exit point) of $\alpha$ in $N_C(X_i)$. We will only prove the statement for $(p_i)_-$ and the other one is similar.

Let $z$ be the entry point of of $\alpha$ in $N_C(X_i)$. If $z\notin [\alpha_-,x]_{\alpha}$ then $z=x'$, and $d((p_i)_-,z)\le d((p_i)_-,x)+d(x,x')\le \epsilon+(\epsilon+2C)<r$. Hence in the following, we only consider $z\in [\alpha_-,x]_{\alpha}$. By \autoref{PROP: Fellow Travel}, $\d{X_i}(\gamma_-,(p_i)_-)\le B$. If $z\neq \gamma_-$, the definition of a $C$-contracting subset implies that $\d{X_i}(\gamma_-,z)\le C$. Hence, by triangle inequality, $\d{X_i}(z,(p_i)_-)\le B+C$. Therefore, $d((p_i)_-,z)\le d((p_i)_-,X_i)+d(z,X_i)+\d{X_i}(z,(p_i)_-)\le C+B+C<r$.
\end{proof}

\begin{theorem}\label{MLSnonGeodesic}
Let $d_1, d_2$ be two left-invariant metrics on $G$. Assume that $d_1, d_2$ are coarsely additive along $(\mathcal F_o,L)$-transversal points for the given action $G\act X$ for some $L>0$. Then $d_1$ and $d_2$ have marked length spectrum rigidity.       
\end{theorem}
\begin{proof}

Suppose the two metrics have the same marked length spectrum: 
\begin{equation}\label{NonGeodEqu0}
    \ell_{d_1}(g)=\ell_{d_2}(g),\;\forall g\in G.
\end{equation}
By left-invariance, it suffices to show that $\sup_{g\in G}\abs{d_1(1,g)-d_2(1,g)}<+\infty$.

Let $\tau>0$ be decided by \autoref{LEM: Modified Extension Lemma} for the three contracting elements $h_1,h_2,h_3$ acting on $X$. Let $D =D(\tau,L,C),r=r(\tau,L,C)$ be given by  \autoref{TransOnAdmisPath}. %
We consider two cases that encompass all $g\in G$.

\textbf{Case 1:} $d(o,go)>D$.

Choose $f_i\in \left\langle h_i\right\rangle $ so that $d(o,f_io)>D$ for each $1\le i\le 3$. By \autoref{LEM: Construction of Perturbation}, there exist $f\in \left\{f_1,f_2,f_3\right\}$ such that $(\cdots, g, f,g,f,\cdots )$ labels a special $(D,\tau)$-admissible path  with respect to  $\mathcal{F}_o$. According to \autoref{TransOnAdmisPath}, for each $n\ge 1$, $(gf)^{n-1}o\in N_r(\mathrm{Trans}_L([o,(gf)^no],\mathcal F_o))$. 
The  definition of coarse additivity gives a constant $R=R(r)$ verifying \autoref{CoarseAdditiveEQ}, so for each $i=1,2$, we have
$$\abs{d_i(1,(gf)^{n})-d_i(1,(gf)^{n-1})-d_i(1,gf)}\le R.$$
By induction, this shows
$$
\abs{d_i(1, (gf)^n)-nd_i(1,gf)}\le nR
,$$ which in turn implies
$$
\ell_{d_i}(gf)=\lim_{n\to\infty}\frac{d_i(1, (gf)^n)}{n}\ge d_i(1,gf)-R.
$$

Denote $L_0:=\max\{d_1(1,f_j), d_2(1,f_j): 1\le j \le 3\}<\infty$, then 
\begin{align*}
d_1(1,g)&\leq d_1(1,gf)+d_1(g,gf)  &\\
&\leq \ell_{d_1}(gf)+R+L_0\\
&=\ell_{d_2}(gf)+R+L_0 &  (\ref{NonGeodEqu0})\\
&\leq d_2(1,gf)+R+L_0& (\text{\autoref{RMK: MLS Basics}\,(\ref{MLSBasics2})})\\
&\leq d_2(1,g)+L_0+R+L_0 &(\Delta\text{-Ineq.})
\end{align*}  
and similarly, $d_2(1,g)\leq d_1(1,g)+R+2L_0$.

\textbf{Case 2:} $d(o,go)\le D$.

Fix an element $h\in G$ such that $d(o,ho)>2D$. By triangle inequality, $d(o,gho)>D$. Then according to \textbf{Case 1}, $\abs{d_1(1,gh)-d_2(1,gh)}\le R+2L_0$. By triangle inequality, $\abs{d_1(1,g)-d_2(1,g)}\le R+2L_0+d_1(1,h)+d_2(1,h)$.

Denoting $M=R+2L_0+d_1(1,h)+d_2(1,h)<+\infty $, we conclude $\sup_{g\in G}\abs{d_1(1,g)-d_2(1,g)}\leq M$, completing the proof.
\end{proof}
{We caution the readers that the assumption in Theorem \ref{MLSnonGeodesic} is crucial, otherwise Remark \ref{IntroRmk: RoughGeodesic} (\ref{NotRouGeo}) gives a counterexample.}

We now present some applications to the Green metric associated to a random walk on a finitely generated group $G$.

Let $\mu$ be a probability measure on $G$ with finite symmetric support, which generates $G$ as a group. Let $\mu^{\star n}$ be the $n$-th convolution on $G$ (i.e. the push forward of $(G^n, \mu^n)$ to $G$ under multiplication map). Define the spectral radius $R:=\limsup_{n\to \infty} (\mu^{\star n})^{1/n}$, and  the $r$-Green function for $1\le r\le R^{-1}$ $$\forall x,y\in G: \;G_r(x,y)=\sum_{n\ge 0}\mu^{\star n}(x^{-1}y)r^n$$
Assume that $G_1(x,y)<\infty$ for some (thus any) $x,y$ (for the support of $\mu$ generates $G$). In other words,  the $\mu$-random walk driven by $\mu$ is transient on $G$; equivalently,  $G$ is not virtually $\mathbb Z^2$. If $G$ is non-amenable, then $R<1$ and $G_{R^{-1}}(x,y)<\infty$.
 
Define a family of the Green metric indexed by $r\in [1,R^{-1}]$ on a non-amenable  group $G$ as follows
$$
\forall x,y\in G,\; d_{r,\mu}(x,y) =-\log\frac{G_r(x,y)}{G_r(1,1)}
$$
which is a proper and left invariant  metric, and is quasi-isometric to word metric on $G$ for $1\le r<R^{-1}$ (cf.\cite[Eq. (7)]{BB07}). We write $d_\mu=d_{1,\mu}$.

We recall the following fact in a non-elementary relatively hyperbolic group, which is direct consequence of \cite[Theorem 1.1]{GGPY}.

\begin{lemma}
Suppose that $G$ is a non-elementary relatively hyperbolic group.  Fix any finite independent set $F$ of three hyperbolic  elements in $G$ for the action on the Cayley graph. Fix $r\in [1,R^{-1})$. Then there exists $L>0$ such that $d_{r,\mu}$ is coarsely additive along $(\mathcal F_1,L)$-transversal points.
\end{lemma}
\begin{proof}
By \cite[Corollary 1.4]{GGPY}, $d_{r,\mu}$ is  coarsely additive along the transition points defined in the sense of maximal parabolic subgroups. %
The result then follows directly from \autoref{LEM: Transition}.
\end{proof}

As a corollary, up to choosing a larger constant $L$ for two Green metrics, we obtain.
\begin{corollary}\label{MLSGreenMetrics}
Let $\mu_1, \mu_2$ be two probability measures on a non-elementary relatively hyperbolic group $G$ with finite symmetric support generating $G$. Then $d_{1,\mu_1}$ and $d_{1,\mu_2}$ have marked length spectrum rigidity. 
\end{corollary}

Recall that for a subset $A\subseteq G$,   the growth rate of $A$ relative to $d$ is defined  as follows
$$
\delta_d(A):=\limsup_{n\to\infty} \frac{
\log \sharp \{g\in A: d(1,g)\le n\}}{n}$$
Alternatively, $\delta_d(A)$ is just the convergence radius of the so-called Poincar\'e series
$$
s\ge 0,\; \Theta_d(A,s)=\sum_{g\in A} \mathrm{e}^{-sd(1,g)}
$$
At last, we present an example of a relatively hyperbolic group endowed with  a pair of proper left invariant metrics, which  has no  marked length spectrum rigidity  from   a subgroup with maximal growth. %
However, as we will see in \autoref{Thm: MLSR From Conjugacy-large Subset}, some right coset of this subgroup still implies the MLS rigidity. 

\begin{proposition}\label{MLSFailsMaximalSubgroup}
There exists a relatively hyperbolic group $G$ with two proper left invariant metrics $d_1, d_2$, which are quasi-isometric to word metric, with the following properties:
\begin{enumerate}
    \item 
    there is a subgroup $\Gamma\subseteq G$ with $\delta_{d_1}(G)=\delta_{d_2}(G)=\delta_{d_1}(\Gamma)$.
    \item 
    the metrics $d_1, d_2$ are coarsely equal on $\Gamma$, but  not on the whole group $G$.
    \item 
    the metrics $d_1, d_2$ are coarsely additive along the transition points.
\end{enumerate}
\end{proposition}
\begin{proof}
Consider a free product $G=\Gamma\ast\mathbb Z^3$ where   $\Gamma=\mathbb F_n\times \mathbb F_m$  for $n\ne m\ge 2$. By  \cite[Theorem 5.1]{DWY23}, there exists a finitely supported irreducible probability measure on $\Gamma$ so that the family of Green metrics $d_r$ for $1\le r<R^{-1}$ has the following phase transition at  some $r_\star \in (1, R)$.  
\begin{enumerate}
    \item 
    for $r\le r_\star$, the Poincar\'e series $\Theta_{d_r}(\Gamma,s)$  is divergent at $s=\delta_{d_r}(\Gamma)$.
    \item 
    for $r> r_\star$, the Poincar\'e series $\Theta_{d_r}(\Gamma,s)$ is convergent at $s=\delta_{d_r}(\Gamma)$.
\end{enumerate} 
Write $\delta=\delta_{d_r}(\Gamma)$ for simplicity.

Pick up $r>r_\star$ so that $\sum_{h\in \Gamma} \mathrm{e}^{-\delta d_r(1,h)}=s<\infty$.
Choose two proper geodesic metrics $|\cdot|_1$ and $|\cdot|_2$ on $K=\mathbb Z^3$, which are quasi-isometric to the word metric, so that 
\begin{equation}\label{ChoiceofS}
\begin{aligned}
\sum_{k\in K} \mathrm{e}^{-\delta |k|_i} <1/s
\end{aligned}
\end{equation}
and $|\cdot|_1$ and $|\cdot|_2$ are not coarsely equal. For instance, we could, if necessary, re-scale a word metric $|\cdot|_1$ on $K$ and set $|\cdot|_2=|\cdot|_1/2$.

Define two proper metric $d_1, d_2$ on $G=\Gamma\ast K$ as follows. By normal form theorem, any element $g\in G$ admits a unique product of the form $g=h_1k_1\cdots h_nk_n$ with $h_i\in \Gamma$ and $k_i\in K$. It is only possible that $h_1=1$ or $k_n=1$.
Define $d_j(1,g)=\sum_{i=1}^n (d_r(1,h_i)+|k_i|_j)$ for $j=1,2$. It is clear that $d_1$ and $d_2$ are equal on $\Gamma$, but are not coarsely equal on $K$ (and thus $G$). By definition of $d_j$, as $G$ is hyperbolic relative to $\{\Gamma, K\}$, two metrics $d_1$ and $d_2$  are coarsely additive along transition points. 

Set $C_0=\max\{s,s^{-1}\}$. At last, a simple computation shows that 
$$
\sum_{g\in G}\mathrm{e}^{-\delta d_j(1,g)} \le C_0^2 \sum_{n\ge 0} \left[\left(\sum_{h\in \Gamma} \mathrm{e}^{-\delta d_j(1,h)} \right)^n\left(\sum_{h\in K} \mathrm{e}^{-\delta |k|_j} \right)^n\right]<\infty
$$
according to \autoref{ChoiceofS}. This implies $\delta_{d_1}(G)=\delta_{d_2}(G)=\delta$. 
\end{proof}

\section{MLS rigidity for cusp-uniform actions}\label{SEC: Application}
%

\subsection{Preliminaries for Gromov hyperbolic spaces}

\text{ }

We recall some basic definitions and properties in the theory of Gromov hyperbolic spaces, and we refer the readers to \cite{Bow06,BH13,BS07} for more details.
\begin{definition}
A metric space $(X,d)$ is called a \textit{$\delta$-hyperbolic space} for some $\delta\geq 0$, if for any four points $x,y,z,w\in X$:
\begin{align*}
\gp{x}{y}{w}\geq \min\left\{\gp{x}{z}{w},\gp{z}{y}{w}\right\}-\frac{\delta}{2}.
\end{align*}
wherew the Gromov product is defined as:
\begin{align}\label{EQU: Gromov Product}
\gp{x}{y}{o}:=\frac{1}{2}(d(o,x)+d(o,y)-d(x,y)).
\end{align}
\end{definition}

An important notion, associated with a hyperbolic space, is the Gromov boundary.
\begin{definition}[{\cite[Definition 3.12]{BH13}} ]
Let $X$ be a proper $\delta$-hyperbolic space with a basepoint $p$. A sequence $(x_n)_{n=1}^{\infty}$ in $X$ converges at infinity if $\gp{x_i}{x_j}{p}\to \infty$ as $i,j\to \infty$.
The Gromov boundary $\partial X$ is defined by
$$\partial X:=\{(x_n)\mid(x_n)\text{ converges at infinity}\}/\sim$$
where $(x_n)\sim (y_n)$ if $\gp{x_i}{y_j}{p}\to\infty$ as $i,j\to \infty$. 
\end{definition}

For a sequence $(x_n)$ converging at infinity, let $\lim x_n\in \partial X$ denote the corresponding boundary point.

Note that the definition of Gromov boundary is independent of the choice of the basepoint $p$, and $\overline{X} = X\cup \partial X$ gives a compactification of $X$.

Furthermore, if the proper $\delta$-hyperbolic space $X$ is a geodesic space, then $\partial X$ is visible in the following sense: for any $\ksi,\eta \in \overline{X}$, there is a geodesic $c:I\to X$ connecting $\ksi$ and $\eta$, in the sense that if $\eta \in \partial X$, then $\eta=\lim_{t\to +\infty} c(t)$ and if $\ksi \in \partial X$, then $\ksi=\lim_{t\to -\infty} c(t)$.

In addition, in a proper hyperbolic space we can extend the Gromov product to the Gromov boundary.
\begin{definition}[{\cite[Definition 3.15]{BH13}} ]\label{DEF: Gromov Product}
Let $X$ be a proper $\delta$-hyperbolic space with a basepoint $p$. We extend the Gromov product to $\overline{X}=X\cup\partial X$ by:
\begin{align*}
\gp{x}{y}{p}=\sup\liminf_{i,j\to\infty}\gp{x_i}{y_j}{p}
\end{align*}
where the supremum is taken over all sequences $(x_i)$ and $(y_j)$ in $X$ such that $x=\lim x_i$ and $y=\lim y_j$.
\end{definition}
\begin{remark}[{{\cite[Remark 3.17]{BH13}},{\cite[Lemma 2.2.2]{BS07}}} ]\label{PROP: Gromov Product Inequalities}
Let $X$ be a proper $\delta$-hyperbolic space with a basepoint $p$.
\begin{enumerate}[(1)]
\item \autoref{DEF: Gromov Product} coincides with the classical definition (\ref{EQU: Gromov Product}) on $X$.
\item\label{6.7.2} $\gp{x}{y}{p}=\infty$ if and only if $x=y\in\partial X$.
\item\label{6.7.3} For all $x,y\in \overline{X}$, and all sequences $(x_i)$ and $(y_i)$ in $X$ with $x= \lim x_i$ and $y=\lim y_j$, we have $\gp{x}{y}{p}-2\delta\leq \liminf_{i\to \infty}\gp{x_i}{y_i}{p}\leq\limsup_{i\to \infty}\gp{x_i}{y_i}{p}\leq  \gp{x}{y}{p}$.
\end{enumerate}
\end{remark}

The Gromov boundary of a proper hyperbolic space is invariant under quasi-isometries.

\begin{proposition}\cite[Proposition 11.107, Theorem 11.108]{DK18}\label{Gromov Product Preserved by Quasi-isometry}
If $X_1$ and $X_2$ are two proper geodesic $\delta$-hyperbolic spaces and $f:X_1\to X_2$ is a $(\lambda,c)$-quasi-isometric-embedding. Then there exists $C=C(\delta,\lambda,c)\ge 0$ such that for any $x,y,z\in X_1$, $\frac{1}{\lambda}\gp{x}{y}{z,X_1}-C\leq \gp{f(x)}{f(y)}{f(z),X_2}\leq {\lambda}\gp{x}{y}{z,X_1}+C$. In addition, if $f$ is a quasi-isometry, then $f$ extends to a homeomorphism $\partial X_1\to \partial X_2$.
\end{proposition}

\begin{definition}
Suppose that a group $G$ acts by isometry on a proper $\delta$-hyperbolic space $X$ with a basepoint $p$. The limit set of $G$ is defined by $\Lambda G=\overline{Gp}\cap \partial X$, and is independent of the choice of the basepoint.
\end{definition}

Notice by definition that $\Lambda G$ admits a natural action by $G$.

From now on, we assume all metric spaces in this section to be geodesic. For any three points $x,y,z$ in a metric space, let $\Delta(x,y,z)$ be a geodesic triangle with three vertices $x,y,z$.

A classical geometric characterization of $\delta$-hyperbolic spaces is the thin triangle property.
\begin{proposition}[{\cite[Proposition 1.17]{BH13}} ]\label{PROP: Thin Triangle}
Let $(X,d)$ be a $\delta$-hyperbolic geodesic space, then any (ideal) geodesic triangle satisfies the $6\delta$-thin triangle property: for any geodesic triangle $\Delta(a,b,c)$ and any points $x\in [a,b], y\in [a,c]$ with $d(a,x)=d(a,y)\le (b|c)_a$, $d(x,y)\le 6\delta$. 
\end{proposition}

\subsection{Relatively hyperbolic group and cusp-uniform action}
\text{ }

In this subsection, we review the definition and basic properties for relatively hyperbolic groups established by Bowditch \cite{Bow12}.
\begin{definition} 
Suppose that a group $G$ acts isometrically and properly  on a proper $\delta$-hyperbolic geodesic space $X$.
    \begin{enumerate}
        \item A point $y\in \partial{X}$ is called a \textit{conical limit point} if there exists a sequence $(g_n)_{n\in\mathbb{N}}$ in $G$, and distinct points $a,b\in \partial {X}$, such that $\lim g_n y= a$ and $\lim g_nx= b$ for all $x\in \partial {X}\setminus\left\{y\right\}$.
        \item A subgroup $H<G$ is called a \textit{parabolic subgroup} if $H$ is infinite, has a unique fixed point $p$ in $\partial X$, and contains no loxodromic elements. It is clear that $p$ is the limit set $\Lambda H$.   
        \item A parabolic subgroup $H<G$ is called \textit{maximal} if for any parabolic subgroup $H'>H$, we have $H'=H$.
        \item A parabolic group $H<G$ is called \textit{bounded parabolic} if the action of $H$ on $\partial X\setminus \left\{p\right\}$ is cocompact, where $p$ is the fixed point of $H$.
        \item A point $y\in \partial{X}$ is called a \textit{bounded parabolic point} if its stablizer $\mathrm{Stab}_G(p)$ is bounded parabolic.
    \end{enumerate}
\end{definition}

\begin{definition}[Cusp-uniform Action]\label{DEF: Cusp Uniform}
Suppose that a group $G$ acts by isometry on a proper $\delta$-hyperbolic geodesic space $X$. The action is called \textit{cusp-uniform} relative to a {finite} collection of subgroups $\left\{H_i\right\}_{i\in I}$, denoted by $(G,\left\{H_i\right\})\curvearrowright X$, if the following hold:
\begin{enumerate}
\item The action $G\curvearrowright X$ is proper.
\item Every point of $\partial X$ is either a conical limit point or a bounded parabolic point. 
\item Each $H_i$ is a maximal parabolic subgroup, and any maximal parabolic subgroup of $G$ conjugates to exactly one subgroup in $\left\{H_i\right\}_{i\in I}$.
\end{enumerate}
\end{definition}

A cusp-uniform action is called \textit{non-elementary} if $\sharp \Lambda G>2$. Throughout this section, we are only concerned about non-elementary cusp-uniform actions.

In \cite{Tuk98}, it is proven that $\left\{H_i\right\}_{i\in I}$ must be a finite collection of subgroups.

\begin{definition}[Relatively Hyperbolic Group]\label{RHGDefn}
A group $G$ is \textit{hyperbolic relative} to a finite collection of subgroups $\left\{H_i\right\}$ if $(G,\left\{H_i\right\})$ admits a cusp-uniform action on a proper $\delta$-hyperbolic geodesic space.  
\end{definition}

In practice, we will use an equivalent definition for cusp-uniform actions which is geometrically more intuitive. Let us first introduce some terminologies.

Fix $X$ to be a proper $\delta$-hyperbolic geodesic space. 
\begin{definition}[Horofunction]
Suppose $\eta\in \partial X$. A function $\mathpzc{h}: X\to \mathbb{R}$ is a \textit{horofunction} about $\eta$ if there exist a point $p\in X$, a geodesic $\gamma:[o,+\infty)\to X$ such that $\gamma(0)=p$ and $\lim_{t\to +\infty}\gamma (t)=\eta$, and a constant $D\geq 0$ such that:
\begin{align*}
\abs{\mathpzc{h}(x)-\lim_{t\to +\infty}(d(x,\gamma(t))-d(p,\gamma(t)))}\leq D.
\end{align*}
\end{definition}
\begin{definition}[Horoball]\label{DEF: Horoball}
A horoball centered at $\eta \in \partial X$ is the pre-image of $(0,+\infty)$ of a horofunction $\mathpzc{h}$ about $\eta$. For a horoball $B=\mathpzc{h}^{-1}\left((0,+\infty)\right)$, its boundary (or thickened horosphere) in $X$ is denoted by $\partial^0 B=\mathpzc{h}^{-1}([0,C])$ for some sufficiently large $C$. 
\end{definition}

Note that $\overline{B}\cap \partial X=\{\eta\}$ when $B\subseteq X$ is a horoball centered at $\eta$.

\begin{proposition}[{\cite[Lemma 6.2]{Bow12}} ]
If $G$ acts parabolically with the fixed point $\eta\in \partial X$, then there exists a $G$-invariant horofunction (hence a $G$-invariant horoball) about $\eta$.
\end{proposition}
\begin{definition}[Bounded Parabolic Action on Horoball]\label{DEF: Bounded Parabolic}
Suppose $G$ acts  parabolically with the fixed point $\eta \in \partial X$ and $B$ is a $G$-invariant horoball centered at $\eta$. Choose an arbitrary point $y\in B$ and a geodesic $\gamma:[0,+\infty)\to X$ with $\gamma(0)=y$, $\lim_{t\to +\infty}\gamma(t)=\eta$. The action of $G$ on $B$ is called \textit{bounded parabolic} if $B$ lies in a finite metric neighbourhood of $\bigcup_{g\in G}g\gamma$.
\end{definition}
\begin{proposition}[{\cite[Proposition 6.13]{Bow12}} ]\label{PROP: Cusp Uniform}
Suppose $(G,\left\{H_i\right\})$ admits a cusp-uniform action on a proper $\delta$-hyperbolic geodesic space $X$. Then there exists a $G$-invariant collection of disjoint horoballs $\mathcal{B}$, such that:
\begin{enumerate}
\item\label{6.16.1} Each bounded parabolic point $\eta\in \partial X$ is associated to exactly one horoball $B\in \mathcal{B}$.
\item\label{6.16.2} Each $H_i$ is a stablizer of a horoball $B\in \mathcal{B}$.
\item\label{6.16.3} The stablizer $\mathrm{Stab}_G (B)$ of each horoball $B\in \mathcal{B}$ conjugates to exactly one subgroup $H\in \left\{H_i\right\}$ and the action of $\mathrm{Stab}_G (B)$ on $B$ is bounded parabolic.
\item\label{6.16.4} The action of $G$ on $X-\bigcup_{B\in\mathcal{B}}B$ is cocompact.
\end{enumerate}
\end{proposition}

\subsection{MLS rigidity for cusp-uniform actions}

\begin{theorem}\label{THM: Cusp Uniform}
Suppose that
$(G,\left\{H_i\right\})$ admits two non-elementary cusp-uniform actions on two proper $\delta$-hyperbolic geodesic spaces $(X_1,d_1)$, $(X_2,d_2)$. If the two actions have the same marked length spectrum, then there exists a $G$-coarsely equivariant rough isometry between $X_1$ and $X_2$.
\end{theorem}

In order to prove \autoref{THM: Cusp Uniform}, we need the following technical lemma which extends a rough isometry from the boundaries of horoballs to their interior.

\begin{lemma}\label{LEM: Cusp Extension}
Let $(X_1,d_1)$ and $(X_2,d_2)$ be two $\delta$-hyperbolic spaces and $B_1\subseteq X_1, B_2\subseteq X_2$ be two horoballs which admit bounded parabolic actions by a parabolic group $H$.

Suppose $f:\partial^0 B_1\to\partial^0 B_2$ is an $H$-equivariant rough isometry, then $f$ extends to an $H$-coarsely equivariant rough isometry $\tilde{f}:B_1\to B_2$.
\end{lemma}

\begin{proof}
For $i=1,2$, let $\eta_i\in \partial X_i$ be the corresponding boundary point of $B_i$. Let $B_i$ be the pre-image of $(0,+\infty)$ of a $H$-invariant horofunction $\mathpzc{h}_i$. $\partial^0 B_i=\mathpzc{h}^{-1}([0,C_i])$ for $C_i$ sufficiently large.

Then, for each $i=1,2$, there exists a point $p_i\in X$, an infinite geodesic $\gamma_i:[0,+\infty)\to X_i$ connecting $p_i$ and $\eta _i$ and a constant $D_i$ such that 
\begin{align}\label{EQU: Horofunction}
\abs{\mathpzc{h}_i(x)-\lim_{t\to+\infty}(d_i(x,\gamma_i(t))-d_i(p_i,\gamma_i(t)))}\leq D_i.
\end{align}

Let $D=\max\left\{D_1,D_2\right\}$.

Choose a point $x_1\in \partial^0 B_1$, and let $x_2=f(x_1)\in\partial^0B_2$.

For each $i=1,2$, choose an infinite geodesic with length parameterization $\alpha^i:[0,+\infty)\to X_i$ such that $\alpha^i(0)=x_i$ and $\lim_{t\to+\infty}\alpha^i(t)=\eta_i$.

For any $h\in H$, $hx_i\in\partial^0 B_i$, denote by $\alpha^i_{h}=h\alpha:[0,+\infty)\to X_i$ an infinite geodesic connecting $hx_i$ and $\eta_i$. Then according to \autoref{DEF: Bounded Parabolic}, $B_i$ lies in a finite metric neighbourhood of $\bigcup_{h\in H}\alpha^i_{h}$.

It suffices to establish a $H$-equivariant rough isometry $\bigcup_{h\in H}\alpha^1_{h}\to \bigcup_{h\in H}\alpha^2_{h}$ given explicitly by:
\begin{equation}
\begin{aligned}\label{EQU: Cusp Extension}
\tilde{f}:\;\;\bigcup_{h\in H}\alpha^1_{h}&\to \bigcup_{h\in H}\alpha^2_{h}\\
\alpha^1_{h}(t)&\mapsto \alpha^2_{h}(t)\;\;\;\;\;\forall h\in H,t\in[0,+\infty).
\end{aligned}
\end{equation}

The definition of $\tilde{f}$ might be multi-valued; however, once we have proven that $\tilde{f}$ is a rough isometry, $\tilde{f}$ becomes well-defined up to a bounded distance. 
The $H$-equivariance and surjectivity of $\tilde{f}$ is clear.

Note that $f$ is a rough isometry, hence the map $hx_1\in Hx_1\mapsto hx_2\in Hx_2$ is a $(1,c_0)$-quasi-isometry for some $c_0\geq 0$. In addition, $\partial^0 B_i$ lies in a finite neighbourhood of $Hx_i$, and $\tilde{f}$ coincides with $f$ on $Hx$ up to a bounded error.

For each $g,h\in H$, since $\mathpzc{h}_1$ is $H$-invariant, $\mathpzc{h}_1(gx_1)=\mathpzc{h}_1(hx_1)$. Then according to (\ref{EQU: Horofunction}):
$$\abs{\lim_{t\to+\infty}(d_1(gx_1, \gamma_1(t))-d_1(hx_1,\gamma_1(t)))}\leq 2D.$$ 

Hence, 
$$\gp{hx_1}{\eta_1}{gx_1,d_1}\sim_{D,\delta} \frac{1}{2}d_1(gx_1,hx_1).$$

Similarly, 
$$\gp{hx_2}{\eta_2}{gx_2,d_2}\sim_{D,\delta}\frac{1}{2}d_2(gx_2,hx_2)$$

By applying the $6\delta$-thin triangle property (cf. \autoref{PROP: Thin Triangle}) to the ideal geodesic triangles $\Delta(gx_1,hx_1,\eta_1)$ and $\Delta(gx_2,hx_2,\eta_2)$, there exists a constant $C=C(D,\delta,c_0)$%
, such that for $s,t\in[0,+\infty)$:
\begin{enumerate}
\item\label{6.19.1} If $s\geq \frac{1}{2}d_1(gx_1,hx_1)$:

By thin-triangle property, we have: $$d_i(\alpha^i_{g}(s),\alpha^i_{h}(s))\leq C\;\;(i=1,2).$$

Then, 
$$d_i(\alpha^i_{g}(s),\alpha^i_{h}(t))\sim_{C}\abs{s-t}\;\;(i=1,2).$$

Hence, 
$$d_1 (\alpha^1_{g} (s),\alpha^1_{h}(t))\sim_{C} d_2(\alpha^2_{g} (s),\alpha^2_{h} (t)).$$
\item If $t\geq \frac{1}{2}d_1(gx_1,hx_1)$:

A similar proof as (\ref{6.19.1}) implies $d_1(\alpha^1_{g}(s),\alpha^1_{h}(t))\sim_{C} d_2(\alpha^2_{g}(s),\alpha^2_{h}(t))$.
\item If $s\leq \frac{1}{2}d_1(gx_1,hx_1)$ and $t\leq \frac{1}{2}d_1(gx_1,hx_1)$:

By thin-triangle property, 
$$d_i(\alpha^i_{g}(s),\alpha^i_{h}(t))\sim_C d_i(gx_i,hx_i)-s-t\;\;(i=1,2).$$

Hence, 

$$d_1(\alpha^1_{g}(s),\alpha^1_{h}(t))\sim_{C,c_0}d_2(\alpha^2_{g}(s),\alpha^2_{h}(t)).$$
\end{enumerate}

Thus, we have proven that the map given by (\ref{EQU: Cusp Extension}) is an $H$-equivariant rough isometry.
\end{proof}

\begin{proof}[Proof of \autoref{THM: Cusp Uniform}]
For each $i=1,2$, we can find a $G$-invariant collection of horoballs $\mathcal{B}_i$ in $X_i$, according to \autoref{PROP: Cusp Uniform}. By \autoref{PROP: Cusp Uniform} (\ref{6.16.2}) and (\ref{6.16.3}), $\mathcal{B}_i/G$ consists of $N$ orbits of horoballs. For each orbit, choose a representative horoball so that we get a collection of $N$ horoballs in $\left\{B_i^{(1)},\cdots,B_i^{(n)}\right\}\subseteq \mathcal{B}_i$. Without loss of generality, suppose that $B_i^{(k)}$ is associated to the unique boundary point $\eta_i^{(k)}$ in $\Lambda_i H_k$, where $\Lambda_i H_k$ denotes the limit set of the parabolic subgroup $H_k$ in $X_i$. Hence $H_k$ is the stabilizer of $B_i^{(k)}$ which acts as a bounded parabolic action.

For each horoball $B_i^{(k)}$ ($i=1,2$, $1\leq k\leq n$), define and fix a boundary $\partial^{0}B_i^{(k)}$ according to \autoref{DEF: Horoball}. And for each $B=gB_i^{k}\in\mathcal{B}_i$, define $\partial^0 B=g\partial^{0}B_i^{(k)}$. Then $A_i=(X_i-\bigcup_{B\in \mathcal{B}_i}B)\cup(\bigcup_{B\in \mathcal{B}_i}\partial^0 B)$ is $G$-invariant and lies in a finite metric neighbourhood of $X_i-\bigcup_{B\in \mathcal{B}_i}B$. Hence $G$ acts {coboundedly} on $A_i$.

For brevity, we may enlarge $A_i$ and each $B\in \mathcal{B}_i$ to its open $1$-neighbourhood, so that we can assume that $A_i$ and each $B\in \mathcal {B}_i$ are open subsets ($i=1,2$).

Choose basepoints $o_i\in A_i$ ($i=1,2$). The actions of $G$ on $X_1$ and $X_2$ are both non-elementary isometries on Gromov hyperbolic spaces. Hence, both actions have contracting property (cf. \autoref{EX: Contracting Property}). According to \autoref{THM: Main Rigidity}, the orbit map $\rho: Go_1\to Go_2$ is a $(1,C_0)$-quasi-isometry for some $C_0\geq 0$. Since $G$ acts cobounded on $A_i$, this induces a $G$-coarsely equivariant $(1,C_1)$-quasi-isometry $f: A_1\to A_2$, and also a natural homeomorphism $\Lambda_1 G=\partial X_1\to \Lambda_2 G=\partial X_2$.

Hence, for each $1\leq k\leq n$, $f(\partial^0 B_1^{(k)})$ lies in a finite metric neighbourhood of $\partial^0 B_2^{(k)}$. With an alteration up to a bounded distance, we may assume $f(\partial^0 B_1^{(k)})\subseteq \partial^0 B_2^{(k)}$.

Then according to \autoref{LEM: Cusp Extension}, we can extend $f$ to a $H_k$-coarsely equivariant rough isometry: $\tilde{f}^{(k)}: B_1^{(k)}\to B_2^{(k)}$. Assume that $\tilde{f}^{(k)}$ is a $(1,c_k)$-quasi-isometry, and $c=\max\left\{c_1,\cdots,c_n\right\}$.

For each $B=gB_1^{(k)}\in \mathcal{B}_1$, let $\tilde{f}=g\circ \tilde{f}^{(k)}\circ g^{-1}: B=gB_1^{(k)}\to gB_2^{(k)}$ be a $(1,c)$-quasi-isometry. Note that $gB_i^{(k)}$ is associated to the boundary point $g\eta_i^{(k)}$, which is the unique point in $\Lambda_i gHg^{-1}$.

Hence $\tilde{f}$ is defined on each $B\in \mathcal{B}_1$, and $\tilde{f}:X_1=A_1\cup(\bigcup_{B\in\mathcal{B}_1}B)\to X_2=A_2\cup(\bigcup_{B\in\mathcal{B}_2}B)$ is a  $G$-coarsely equivariant extension of $f$.

A similar construction from $X_2$ to $X_1$ will also provide a $G$-coarsely equivariant map $\tilde{f}_{\ast}:X_2\to X_1$, which is a quasi-inverse of $\tilde {f}$: $\sup_{x\in X_1} d_1(x,\tilde{f}_\ast(\tilde{f}(x)))<\infty$ and $\sup_{x\in X_2}d_2(x,\tilde{f}(\tilde{f}_\ast(x)))<\infty$. 

We will prove that $\tilde{f}$ is a rough isometry.

For any $x,y\in X_1$, let $\alpha=[x,y]$ denote a geodesic segment connecting $x$ and $y$. There are five cases to be considered.

\textbf{Case 1:} $x,y\in A_1$. 
    Then 
    $d_2(\tilde{f}(x),\tilde{f}(y))\sim_{C_1}d_1(x,y)$.

\textbf{Case 2:} $x,y\in B$ for a horoball $B\in\mathcal{B}_1$. 
    Then 
    $d_2(\tilde{f}(x),\tilde{f}(y))\sim_{c}d_1(x,y)$.
    
\textbf{Case 3:} $x\in A_1$ and $y\in B$ for a horoball $B\in\mathcal{B}_1$.
    
    Then we can seperate $\alpha$ into two parts: $\alpha=[x,z]\cup [z,y]$, such that $z\in A_1\cap B$. Hence,
    \begin{align*}
        d_2(\tilde{f}(x),\tilde{f}(y))&\leq d_2(\tilde{f}(x),\tilde{f}(z))+d_2(\tilde{f}(z),\tilde{f}(y))\\
        &\leq d_1(x,z)+C_1+d_1(z,y)+c\\
        &=d_1(x,y)+C_1+c.
    \end{align*}

\textbf{Case 4:} $y\in A_1$ and $x\in B$ for a horoball $B\in\mathcal{B}_1$, 
which is similar to \textbf{Case 3}.

\textbf{Case 5:} $x\in B$ and $y\in B'$ for two horoballs $B\neq B'\in \mathcal{B}_1$.

    Then we can separate $\alpha$ into three parts: $\alpha=[x,z]\cup [z,w]\cup [w,y]$, such that $z\in A_1\cap B$ and $w\in A_1\cap B'$. Hence,
    \begin{align*}
        d_2(\tilde{f}(x),\tilde{f}(y))&\leq d_2(\tilde{f}(x),\tilde{f}(z))+d_2(\tilde{f}(z),\tilde{f}(w))+d_2(\tilde{f}(w),\tilde{f}(y))\\
        &\leq d_1(x,z)+c+d_1(z,w)+C_1+d_1(w,y)+c\\
        &=d_1(x,y)+C_1+2c.
    \end{align*}

Hence $d_2(\tilde{f}(x),\tilde{f}(y))\leq d_1(x,y)+C$ for any $x,y\in X_1$, where $C=C_1+2c$.

A similar proof to the quasi-inverse $\tilde{f}_{\ast}$ shows that there exists $C_\ast\ge 0$ such that $d_1(\tilde{f}_{\ast}(x),\tilde{f}_{\ast}(y))\leq d_2(x,y)+C_{\ast}$ for any $x,y\in X_2$.

Hence $\tilde{f}:X_1\to X_2$ is a $G$-coarsely equivariant rough isometry. 
\end{proof}

\section{MLS rigidity from  confined subgroups}\label{Sec: MLSR from confined subgroups}\text{ }

In this section, we study the MLS rigidity from confined subgroups defined as follows.
\begin{definition}\label{DEF: Confined Subgroup}
    Let $G$ be a group. We say that a subgroup $H\leq G$ is  a \textit{confined subgroup} of $G$ if there exists a finite set $P\subseteq G\setminus\left\{1\right\}$ such that $gHg^{-1}\cap P\neq \varnothing$ for all $g\in G$. The set $P$ is referred to as the \textit{confining} subset. 
\end{definition}
If $G$ acts properly and co-compactly on a proper metric space $(X, d)$ then a torsion-free confined subgroup is amount to saying that $X/H$ has bounded injective radius from above (\cite[Lemma 5.3]{CGYZ}). We refer to \cite{CGYZ} for more references and relevant discussions about confined subgroups.

Normal subgroups are important examples of confined subgroups. We  first prove the MLS rigidity from normal subgroups under isometric actions with contracting property, and then the MLS rigidity from confined subgroups, if the actions are assumed additionally to be proper.

\subsection{Normal subgroups}\text{ }
For a group $G$ acting isometrically on a metric space $X$, a subgroup $H\le G$ is called \textit{unbounded in $X$} (or \textit{unbounded} for short) if any orbit of $H$ is an unbounded subset of $X$. An unbounded subgroup is equivalent to an infinite subgroup when the action $G\curvearrowright X$ is assumed to be proper.

\begin{theorem}\label{THM: Normal Subgroup Rigidity}
Suppose that a group $G$ acts isometrically on two geodesic metric spaces $(X_1,d_1)$ and $(X_2,d_2)$ with contracting property. Let $H\lhd G$ be a normal subgroup which is unbounded in either $X_1$ or $X_2$. %
Then $G$ has marked length spectrum rigidity from $H$.
\end{theorem}
%
In order to prove \autoref{THM: Normal Subgroup Rigidity}, we need a slightly different construction compared with \autoref{COR: Perturbation Length}.
\begin{lemma}\label{LEM: Unbounded Normal Subgroup Contracting}
    Let $G$ be a group acting isometrically on a geodesic metric space $(X,d)$ with contracting property, and $H\lhd G$ be an unbounded normal subgroup. Then $H$ also has contracting property.
\end{lemma}
\begin{proof}  
    Fix a basepoint $o\in X$. Choose five pairwise independent contracting elements $g_1,\cdots,g_5\in G$. Then $g_1^{-1},\cdots,g_5^{-1}$ are also pairwise independent contracting elements. Let $\tau>0$ be the maximal one among the corresponding $\tau$ given by \autoref{LEM: Modified Extension Lemma} for any triples within $g_1,\cdots,g_5$ or $g_1^{-1},\cdots,g_5^{-1}$. Let $C\ge 0$ be the maximal contraction constant of $\left\langle g_i\right\rangle o $ ($i=1,\cdots,5$), and $\mathcal{R}=\mathcal{R}_{\tau,C}$ be given by \autoref{COR: Bounded Intersection}. Let $D_1=D(\tau,C)$ and $\epsilon=\epsilon(\tau,C)$ be given by \autoref{COR: Bounded Intersection}, $D_2=D(\tau,C,\mathcal{R})$ be given by \autoref{PROP: Admissible Path Contracting}, and $D=\max\left\{D_1+3\epsilon,D_2\right\}$. For each $1\le i\le 5$, choose $f_i\in \left\langle g_i\right\rangle $ such that $d(o,f_io)>D$.

    Since $H$ is unbounded, we can choose an element $h\in H$ such that $d(o,ho)>D$. Then according to \autoref{LEM: Modified Extension Lemma} and by Pigeonhole Principle, there exists a set of three elements $F_1\subset\left\{f_1,\cdots, f_5\right\}$ such that $(h,f,h)$ labels a special $(D,\tau)$-admissible path for each $f\in F_1$. And similarly, there exists a set of three elements $F_2\subset\left\{f_1^{-1},\cdots, f_5^{-1}\right\}$ such that $(h,f^{-1},h)$ labels a special $(D,\tau)$-admissible path for each $f^{-1}\in F_2$. Therefore, by Pigeonhole Principle, there exists an element $f=f_i$ such that $(\cdots,h,f,h,f^{-1},h,f,h,f^{-1},\cdots)$ labels a special $(D,\tau)$-admissible path  with respect to  $\mathbb{X}\subseteq \left\{ g\cdot\left\langle f\right\rangle o\mid g\in G \right\}$.

    This admissible path has $\mathrm{Len}_p$-constant $d(o,fo)$ and $\mathrm{Len}_q$-constant $d(o,ho)$. Therefore, by \autoref{PROP: Admissible Path Contracting}, it is contracting; and by \autoref{PROP: Fellow Travel} (\ref{2.15.2}), the map $n\in\mathbb{Z}\mapsto (hfhf^{-1})^no\in  X$ is a quasi-isometric embedding.

    Hence the element $\tilde{h}=hfhf^{-1}\in H$ is a contracting element. By \autoref{COR: New Contracting}, we can find a contracting element $g\in G $ which is weakly independent with $\tilde{h}$. Then according to \autoref{LEM: Non Elementary} (\ref{_NE2}), there exists $n\gg 0$ such that $\tilde{h}$ and $(\tilde{h}^ng^n)\tilde{h}(\tilde{h}^ng^n)^{-1}$ are weakly independent contracting elements in $H$. Hence $H$ has contracting property.
\end{proof}

\begin{lemma}\label{LEM: Normal Subgroup Perturbation}
Suppose that a group $G$ acts by isometry on a geodesic metric space $(X,d)$ with a basepoint $o$, and there are three pairwise weakly independent contracting elements $h_1,h_2,h_3\in G$. Then there exist $\tilde{\epsilon}, \tilde{D}>0$ with the following property.

Fix any $F=\left\{f_1,f_2,f_3\right\}$ with $f_i\in\left\langle h_i\right\rangle,d(o,f_io)>\tilde{D}$. Then, for any $g\in \Gamma$ satisfying $d(o,go)>\tilde{D}$, we can choose $f_i,f_j\in F$ so that:
\begin{enumerate}[(1)]
\item $d(o,gf_ig^{-1}f_jo)\geq 2d(o,go)-3\tilde{\epsilon}$.
\item $\ell_d(gf_ig^{-1}f_j)\geq d(o,gf_ig^{-1}f_jo)-2\tilde{\epsilon}$.
\end{enumerate}
\end{lemma}
\begin{proof}
For $h_1,h_2,h_3\in \Gamma$, let $\tau$ be decided by \autoref{LEM: Modified Extension Lemma}, and $C$ be a contraction constant for the contracting system $\mathbb{Y}=\left\{\left\langle h_1\right\rangle o,\left\langle h_2\right\rangle o,\left\langle h_3\right\rangle o \right\}$.%

Let $\tilde{D}=D(\tau,C),\tilde{\epsilon}=\epsilon(\tau,C)$ be decided according to \autoref{PROP: Fellow Travel}.

Fix any $F=\left\{f_1,f_2,f_3\right\}$ with $f_i\in\left\langle h_i\right\rangle, d(o,f_io)>\tilde{D}$.

For the pair of elements $(g,g^{-1})$, we can find $f_i\in F$ according to \autoref{LEM: Modified Extension Lemma} (\ref{ext1}); and for the pair of elements $(g^{-1},g)$, we can find $f_j\in F$ according to \autoref{LEM: Modified Extension Lemma} (\ref{ext1}).

Since $d(o,go)=d(o,g^{-1}o)>\tilde{D}$, \autoref{LEM: Modified Extension Lemma} (\ref{ext2}) implies that the path $\gamma$ labelled by $(\cdots, g,f_i,g^{-1},f_j,g,f_i,g^{-1},f_j,\cdots)$ is a special $(\tilde{D},\tau)$-admissible path  with respect to  $\mathbb{X}$.

According to \autoref{PROP: Fellow Travel}, any geodesic segment connecting two vertices of $\gamma$ $\tilde{\epsilon}$-fellow travels the subpath of $\gamma$ between these two vertices.

\textbf{(1).} Let $\alpha=[o,gf_ig^{-1}f_jo]$, then $\alpha $ $\tilde{\epsilon}$-fellow travels the path $[o,go]\cup[go,gf_io]\cup[gf_io,gf_ig^{-1}o]\cup[gf_ig^{-1}o,gf_ig^{-1}f_jo]$. 
We can find three linearly ordered points $x,y,z$ on $\alpha$, such that $d(go,x)\leq \tilde{\epsilon}$, $d(gf_io,y)\leq\tilde{\epsilon}$, $d(gf_ig^{-1}o,z)\leq \tilde{\epsilon}$. Hence, by triangle inequality, 
\begin{align*}
d(o,gf_ig^{-1}f_jo) \geq d(o,x)+d(y,z) \geq d(o,go)+d(o,g^{-1}o)-3\tilde{\epsilon}=2d(o,go)-3\tilde{\epsilon}.
\end{align*}

\textbf{(2).} For $n\in\mathbb{N}^{+}$, let $\alpha_n=[o,(gf_ig^{-1}f_j)^no]$. Then $\alpha_n$ $\tilde{\epsilon}$-fellow travels $\gamma_n=[o,go]\cup[go,gf_io]\cup\cdots\cup[(gf_ig^{-1}f_j)^{n-1}gf_ig^{-1}o,(gf_ig^{-1}f_j)^no]$. 
We can find a point $w\in \alpha_n$ such that $d((gf_ig^{-1}f_j)^{n-1}o,w)\leq \tilde{\epsilon}$. Hence, by triangle inequality,
\begin{align*}
d(o,(gf_ig^{-1}f_j)^no)&=d(o,w)+d(w,(gf_ig^{-1}f_j)^no)
\\ &\geq d(o,(gf_ig^{-1}f_j)^{n-1}o)+d(o,gf_ig^{-1}f_jo)-2\tilde{\epsilon}.
\end{align*}

By induction, $d(o,(gf_ig^{-1}f_j)^no)\geq nd(o,gf_ig^{-1}f_jo)-2(n-1)\tilde{\epsilon}$.

Thus, $\ell_d(gf_ig^{-1}f_j)=\lim_{n\to +\infty}\frac{1}{n}d(o,(gf_ig^{-1}f_j)^no)\geq d(o,gf_ig^{-1}f_jo)-2\tilde{\epsilon}$.
\end{proof}
\begin{proof}[Proof of \autoref{THM: Normal Subgroup Rigidity}]
By assumption, the two actions under consideration have the same marked length spectrum on $H$:
\begin{align}\label{NormalEQU0}
    \ell_{d_1}(h)=\ell_{d_2}(h),\quad \forall h\in H.
\end{align}

We first show that both actions $H\curvearrowright X_1$ and $H\curvearrowright X_2$ are unbounded and hence have contracting property. By assumption and without loss of generality, we may suppose $H\curvearrowright X_1$ is unbounded. Then by \autoref{LEM: Unbounded Normal Subgroup Contracting}, $H\curvearrowright X_1$ has contracting property. Pick an element $h\in H$ which is contracting in $X_1$, then $\ell_{d_2}(h)=\ell_{d_1}(h)>0$ and therefore $h$ must have an unbounded orbit in $X_2$. According to \autoref{LEM: Unbounded Normal Subgroup Contracting}, $H\curvearrowright X_2$ also has contracting property.

According to \autoref{LEM: Simultaneous Contracting}, we can choose a weakly independent set of simultaneously contracting elements $\left\{h_1,h_2,h_3\right\}\subseteq H\cap \mathcal{SC}(G)$. For each $i=1,2$, consider the action $G\curvearrowright X_i$ and let $\tilde{\epsilon}_i,\tilde{D}_i>0$ be decided by \autoref{LEM: Normal Subgroup Perturbation} for the three elements $h_1,h_2,h_3$. Let $\epsilon=\max\left\{\epsilon_1,\epsilon_2\right\}$ and $D=\max\left\{D_1,D_2\right\}$. Fix $F=\left\{f_1,f_2,f_3\right\}$ with $f_k\in \left\langle h_k\right\rangle$, and $d_i(o_i,f_ko_i)>{D}$ for each $ k=1,2,3$ and $i=1,2$. Let $C_0=\max\left\{d_i(o_i,f_ko_i)\mid k=1,2,3,\,i=1,2\right\}>D$.

Then according to \autoref{LEM: Normal Subgroup Perturbation}, for any $g\in G$ such that $d_1(o_1,go_1)>D$, there exist $ f_i,f_j\in F$ satisfying: 
\begin{align}
\label{NormalEQU1}&d_1(o_1,gf_ig^{-1}f_jo_1)\geq 2d_1(o_1,go_1)-3{\epsilon},
\\&\label{NormalEQU2}\ell_{d_1}(gf_ig^{-1}f_j)\geq d_1(o_1,gf_ig^{-1}f_jo_1)-2{\epsilon}.
\end{align}

Notice that for any $f_i,f_j\in F$ and $g\in G$, $gf_ig^{-1}f_j\in H$.\par
Hence, for any $g\in G$ such that $d_1(o_1,go_1)>D$,
\begin{align*}
2d_1(o_1,go_1)&\leq d_1(o_1,gf_ig^{-1}f_jo_1)+3\epsilon & (\ref{NormalEQU1})\\
&\leq \ell_{d_1}(gf_ig^{-1}f_j)+5\epsilon & (\ref{NormalEQU2})\\
&=\ell_{d_2}(gf_ig^{-1}f_j)+5\epsilon & (\ref{NormalEQU0}) \\
&\leq d_2(o_2, gf_ig^{-1}f_jo_2)+5\epsilon & (\text{\autoref{RMK: MLS Basics}\,(\ref{MLSBasics2})})\\
&\leq d_2(o_2,go_2)+d_2(o_2,f_io_2)+d_2(o_2,g^{-1}o_2)+d_2(o_2,f_jo_2)+5\epsilon & (\text{$\Delta$-Ineq.})\\
&\leq 2d_2(o_2,go_2)+2C_0+5\epsilon.&
\end{align*}

On the other hand, for any $g\in G$ such that $d_1(o_1,go_1)\leq D$,
\begin{align*}
d_1(o_1,go_1)\leq d_2(o_2,go_2)+D.
\end{align*}

Let $C=C_0+\frac{5}{2}\epsilon$. %
Then, $d_1(o_1,go_1)\leq d_2(o_2,go_2)+C$, for all $ g\in G$.\par
Similarly, $d_2(o_2,go_2)\leq d_1(o_1,go_1)+C$, for all $ g\in G$.\par
Hence, the map $\rho: Go_1\to Go_2$, $go_1\mapsto go_2$ is a $(1,C)$-quasi-isometry.
\end{proof}
\begin{corollary}\label{COR: Sequence of Normal Subgroup}
Suppose there is a subnormal sequence of unbounded subgroups $H_n\lhd H_{n-1}\lhd\cdots\lhd H_1\lhd H_0=G$. 
Then, $G$ has marked length spectrum rigidity from $H_n$.
\end{corollary}
\begin{proof}
    For each $1\leq i\leq n$, we have established the following fact:

    If $\ell_{d_1}(h)=\ell_{d_2}(h)$ for all $h\in H_i$, then according to \autoref{THM: Normal Subgroup Rigidity}, the orbit map on $H_{i-1}$, $\rho|_{H_{i-1}o_1}: H_{i-1}o_1\to H_{i-1}o_2$ is a rough isometry. As a consequence, $\ell_{d_1}(h)=\ell_{d_2}(h)$ for all $h\in H_{i-1}$.
    
    An induction of $i$ from $n$ to $1$ proves that the orbit map on $H_0=G$, $\rho: Go_1\to Go_2$, is a rough isometry.
\end{proof}

\subsection{Elementary subgroups and elliptic radical}

In this subsection,  suppose that a group $G$ acts isometrically and properly on a geodesic metric space $(X,d)$. That is, $\sharp \left\{g\in G \mid d(o,go)\leq n\right\}<+\infty$ for any $o\in X$ and $n\in \mathbb{N}$. If $G$ contains a contracting element and $G$ is not virtually cyclic, then the action $G$ has  contracting property.  In the following, we fix a basepoint $o\in X$.

\begin{definition}[Elementary Subgroup]\label{DEF: Elementary subgroup}
    For any contracting element $h\in G$, the elementary subgroup associated to $h$ is defined as: $$E(h)=\left\{g\in G\mid d_{H}(\left\langle h\right\rangle o,g\cdot \left\langle h\right\rangle o)<+\infty\right\}.$$
\end{definition}

\begin{proposition}[{\cite[Lemma 2.11]{Yang19}} ]\label{LEM: Elementary Subgroup}
    Let $h\in \G$ be a contracting element, then:
    \begin{enumerate}
        \item\label{2.30.1} $\left\langle h\right\rangle\le E(h)$ and $\left [E(h):\left\langle h\right\rangle\right ]<+\infty$.
        \item\label{ES3} $E(h)=\left\{g\in G\mid \exists m>0\text{ s.t. }gh^mg^{-1}=h^m\text{ or }gh^mg^{-1}=h^{-m}\right\}$.
    \end{enumerate}
\end{proposition}
\begin{corollary}\label{COR: Commensurable}
    Suppose $f,h\in G$ are two contracting elements.
    \begin{enumerate}
        \item\label{ESCor1} For any non-torsion element $g\in E(h)$, there exist $n,m\in \mathbb{Z}\setminus\left\{0\right\}$ such that $g^n=h^m$.
        \item\label{ESCor2} If $f$ and $h$ are weakly independent, then $f\notin E(h)$.
    \end{enumerate}
\end{corollary}
\begin{proof}
    (1) According to \autoref{LEM: Elementary Subgroup} (\ref{2.30.1}), there exists a finite collection of elements $a_1,\cdots, a_n\in E(h)$ such that $E(h)=\bigcup_{i=1}^{n}a_i\left\langle h\right\rangle$.
        By Pigeonhole Principle, there exists $0\leq j<k\leq n$ such that $g^j$ and $g^k$ lie in the same coset $a_i\left\langle h\right \rangle$. Therefore, $g^{k-j}\in \left\langle h\right\rangle$. Hence, $g^{k-j}=h^m$. In addition, $g$ is non-torsion, which implies that $m\neq 0$.
        
        (2) If $f\in E(h)$, since $f$ is non-torsion, by (\ref{ESCor1}) there exist $n,m\neq 0$ such that $f^n=h^m$. Therefore, $d_H(\left\langle f\right\rangle o, \left\langle h\right\rangle o)\leq d_H(\left\langle f\right\rangle o,\left\langle f^n\right\rangle o)+d_H(\left\langle h\right\rangle o,\left\langle h^m\right\rangle o)<+\infty$, so $\left\langle f\right\rangle o$ and $\left\langle h\right\rangle o$ do not have bounded intersection, which is a contradiction.
\end{proof}

For a contracting element $h\in G$, we define $$E^{+}(h)=\left\{g\in G\mid gh^mg^{-1}=h^m \text{ for some } m>0\right\}.$$ According to \autoref{LEM: Elementary Subgroup} (\ref{ES3}), $E^{+}(h)$ is a subgroup of $E(h)$ whose index is at most $2$.

As a matter of fact, for any $g\in G$,
\begin{equation}\label{EQU: Conjugation of E}
    E(ghg^{-1})=gE(h)g^{-1};\quad E^+(ghg^{-1})=gE^+(h)g^{-1}.
\end{equation}
\begin{lemma}\label{LEM: E+}
Suppose $f,h\in G$ are two contracting elements.
    \begin{enumerate}
        \item\label{E+1} There are only finitely many torsion elements in $E^+(h)$.
        \item\label{E+2} If $f$ and $h$ are not weakly independent, then $E^+(f)=E^+(h)$.
    \end{enumerate}
\end{lemma}
\begin{proof}
(1) Obviously, $\left\langle h\right\rangle \le E^+(h)$. According to \autoref{LEM: Elementary Subgroup} (\ref{2.30.1}), we have $[E^+(h):\left\langle h\right\rangle]<+\infty$. Since any finite index subgroup contains a finite index normal subgroup, there exists $n>0$ such that $\left\langle h^n\right\rangle \lhd E^+(h)$ and $[E^+(h):\left\langle h^n\right\rangle]<+\infty$.

We claim that each coset of $\left\langle h^n\right\rangle$ in $E^+(h)$ contains at most one torsion element.

Suppose $a\in E^+(h)$ and consider the coset $a\left\langle h^n\right\rangle$. Without loss of generality, suppose that $a$ is a torsion element and $a^k=1$ ($k>0$). By definition of $E^+(h)$, there exists $m>0$ such that $ah^ma^{-1}=h^m$. On the other hand, since $\left\langle h^n\right\rangle $ is a normal subgroup, we have $ah^na^{-1}=h^{tn}$ for some $t\in \mathbb{Z}$. Therefore, $h^{mn}=ah^{mn}a^{-1}=h^{mtn}$, and since $h$ is non-torsion we must have $t=1$. Thus, $ah^n=h^na$.

For any element $ah^{sn}\in a\left\langle h^n\right\rangle$. If $s\neq 0$, then $(ah^{sn})^k=a^kh^{snk}=h^{snk}$ which is a non-torsion element. Hence, $ah^{sn}$ is non-torsion, and our claim is proven.

(2) Since $f$ and $h$ are not weakly independent, there exists $r>0$ such that $\diam(N_r(\left\langle f\right\rangle o)\cap N_r(\left\langle h\right\rangle o))=+\infty$. Therefore, there exists an infinite sequence of distinct integers $\left\{n_k\right\}_{k=0}^{+\infty}$, such that for each $k\ge 0$ there exists $m_k\in \mathbb{Z}$ with $d(f^{n_k}o,h^{m_k}o)\le 2r$.

Therefore, $d(h^{-m_k}f^{n_k}o,o)\le 2r$. Since the action is proper, i.e. $\sharp \left\{g\in G\mid d(go,o)\le 2r\right\}<+\infty$, there exists $k>j\ge 0$ such that $h^{-m_k}f^{n_k}=h^{-m_j}f^{n_j}$. Therefore, $h^{m_j-m_k}=f^{n_j-n_k}$. Since $h,f$ are non-torsion and $m_j-m_k\neq 0$, we have $n_j-n_k\neq 0$.

Thus it is easy to verify that $E^+(f)=E^+(f^{n_j-n_k})=E^+(h^{m_j-m_k})=E^+(h)$. 
\end{proof}

\begin{proposition}\label{PROP: Elliptical Radical}
    If $G$ admits an isometric and proper action on a geodesic metric space with contracting property, then there exists a unique maximal finite normal subgroup $E(G)\lhd G$. In addition, $$E(G)=\bigcap_{f\in G:\text{ contracting}}E(f)=\bigcap_{f\in G:\text{ contracting}}E^+(f).$$
\end{proposition}
\begin{proof}
    If denote $K=\bigcap_{f\in G:\text{ contracting}}E(f)$ and $K^+=\bigcap_{f\in G:\text{ contracting}}E^+(f)$, it is obvious that $K^+\le K\le G$. According to (\ref{EQU: Conjugation of E}), $K\lhd G$ and $K^+\lhd G$ are normal subgroups.

    In addition, $K$ does not contain any contracting element. If not, suppose $h\in K$ is a contracting element.  According to \autoref{COR: New Contracting}, we can then find a contracting element $f\in \Gamma$ weakly independent with $h$. However,  according to \autoref{COR: Commensurable} (\ref{ESCor2}), $h\notin E(f)$, which leads to a contradiction. Hence, \autoref{LEM: Unbounded Normal Subgroup Contracting} shows that an orbit of $K$ must be bounded. Since the action is proper, we conclude that $K$ is a finite subgroup and so is $K^+$.

    In the following, we will show that $H\le K^+$ for any finite normal subgroup $H\lhd G$. Otherwise, suppose there exists $h\in H$ such that $h\notin K^+$. Then there exists a contracting element $f\in G$ such that $h\notin E^+(f)$. By definition, $\left\{f^{-n}hf^{n}\right\}_{n\in \mathbb{N}}$ are pairwise different elements. On the other hand, $f^{-n}hf^{n}\in H$ by normality of $H$, hence contradicting the finiteness of $H$. 

    In particular, $K\le K^+$. Hence $K=K^+$ is the unique maximal finite normal subgroup.
\end{proof}
In the following, we call $E(G)$ the \textit{elliptic radical} of $G$, which is a characteristic subgroup of $G$ and is, in particular, independent of the choice of a proper action of $G$ with contracting property.
\begin{remark}
    When $(X,d)$ is a $\delta$-hyperbolic space, the elliptic radical coincides with the definition of elliptic radical by Osin: $E(G):=\{g\in G: g\xi=\xi,\forall \xi\in \Lambda G\}$; see \cite[Prop. 3.4]{Osi22}. 
\end{remark}
\begin{lemma}\label{LEM: ER1}
    Suppose $f,h\in G$ are two weakly independent contracting elements, and $g\in G\setminus E^+(f)$. Then there exists $N>0$ such that $g\notin E^+(f^nhf^{-n})$ for any $n\ge N$.
\end{lemma}
\begin{proof}
    According to (\ref{EQU: Conjugation of E}), $E^+(f^nhf^{-n})=f^nE^+(h)f^{-n}$, so it suffices to show that there are only finitely many $n>0$ such that $f^{-n}gf^{n}\in E^+(h)$.

    There are two cases to be considered.

    \textbf{Case 1:} $g$ is a torsion element.

    Then $f^{-n}gf^{n}$ are all torsion elements. 
    Since $g\notin E^{+}(f)$, by definition, $\left\{f^{-n}gf^n\right\}_{n\in\mathbb{N}}$ are pairwise different elements. 
    The argument then follows from \autoref{LEM: E+} (\ref{E+1}) that $E^+(h)$ contains only finitely many torsion elements.

    \textbf{Case 2:} $g$ is a non-torsion element.

    We claim that there is at most one $n\in \mathbb{Z}$ such that $f^{-n}gf^{n}\in E^+(h)$. 
    
    Suppose on contrary that for $n<m$, $f^{-m}gf^{m},f^{-n}gf^n\in E^+(h)$. Then according to \autoref{COR: Commensurable} (\ref{ESCor1}), there exist $r_1,r_2,s_1,s_2\in \mathbb{Z}\setminus \left\{0\right\}$ such that $f^{-m}g^{r_1}f^{m}=h^{s_1}$ and $f^{-n}g^{r_2}f^{n}=h^{s_2}$. Hence, $\abs{r_1}\ell_d(g)=\ell_d(f^{-m}g^{r_1}f^{m})=\ell_d(h^{s_1})=\abs{s_1}\ell_d(h)$, and similarly, $\abs{r_2}\ell_d(g)=\abs{s_1}\ell_d(h)$. Since $\ell_d(h)\neq 0$, we have $\frac{r_1}{r_2}=\pm\frac{s_1}{s_2}$. Therefore, by taking a power, there exists $r,s\in \mathbb{Z}\setminus \left\{0\right\}$ such that $f^{-m}g^rf^m=h^s$ and $f^{-n}g^rf^{n}=h^{\pm s}$. Thus we have $f^{m-n}h^sf^{n-m}=h^{\pm s}$, which is to say $f\in E(h)$. This is contradictory with \autoref{COR: Commensurable} (\ref{ESCor2}).
\end{proof}
\begin{lemma}\label{LEM: ER2}
    If $P\subseteq G\setminus E(G)$ is a finite subset, then there exists a contracting element $f\in G$ such that $P\cap E^+(f)=\varnothing$.
\end{lemma}
\begin{proof}
    We prove by induction on $\sharp P$. If $\sharp P=1$, the conclusion follows directly from \autoref{PROP: Elliptical Radical}: $E(G)=\bigcap_{f:\text{ contracting}}E^+(f)$.

    Now we suppose $n\ge 1$ and the conclusion holds for $\sharp P=n$. We will prove the conclusion for $\sharp P=n+1$.

    Suppose $P=\left\{p_1,\cdots, p_n,q\right\}$. By induction hypothesis, there exists a contracting element $f$ such that $p_1,\cdots, p_n\notin E^+(f)$. In addition, there exists a contracting element $h$ such that $q\notin E^+(h)$.  

    First of all, if  $q\notin E^+(f)$, then $P\cap E^+(f)=\varnothing$ and there is nothing to do in this case. We thus assume $q\in E^+(f)$ in what follows.

    As $q\notin E^+(h)$, we must have $E^+(f)\neq E^+(h)$, so according to \autoref{LEM: E+} (\ref{E+2}), $f$ and $h$ must be weakly independent. By \autoref{LEM: ER1}, there exists $K>0$ such that $p_1,\cdots, p_n\notin E^+(f^khf^{-k})$ for any $k\geq K$.

    Moreover, since $q\in E^+(f)$, there exists $m>0$ such that $qf^mq^{-1}=f^m$. We will show that $q\notin E^+(f^{tm}hf^{-tm})$ for any $t\in \mathbb{Z}$. Otherwise, there exists $l>0$ such that $qf^{tm}h^lf^{-tm}q^{-1}=f^{tm}h^lf^{-tm}$. Since $qf^mq^{-1}=f^m$, this implies that $qh^lq^{-1}=h^l$, i.e. $q\in E^+(h)$, which is a contradiction.

    Thus, by choosing $t\gg 0$ such that $tm>K$, we have $P\cap E^+(f^{tm}hf^{-tm})=\varnothing$, where $f^{tm}hf^{-tm}$ is a contracting element.
\end{proof}
\begin{corollary}\label{COR: ER3}
    Suppose $P\subseteq G\setminus E(G)$ and $A\subseteq G$ are both finite subsets. Then there exists $g_0\in G$ such that $g_0Pg_0^{-1}\cap A=\varnothing$.
\end{corollary}
\begin{proof}
    According to \autoref{LEM: ER2}, there exists a contracting element $f\in \Gamma$ such that $P\cap E^+(f)=\varnothing$. Let $N=\sharp P\cdot \sharp A$. We will show that there exists $g_0=f^n$ ($0\le n\le N$) such that $g_0Pg_0^{-1}\cap A=\varnothing$.

    Suppose, by contrary, for each $0\le n\le N$ there exists $p_n\in P$ and $a_n\in A$ such that $f^np_nf^{-n}=a_n$. By Pigeonhole Principle, there exists $0\le n<m\le N$ such that $p_n=p_m$ and $a_n=a_m$. Therefore, $f^np_nf^{-n}=f^mp_nf^{-m}$, implying that $p_n\in E^+(f)$, which is a contradiction. 
\end{proof}
\subsection{Confined subgroups}

We now state the main result of this subsection.
\begin{theorem}\label{Thm: MLS from confined subgps}
   Suppose the two actions $G\curvearrowright X_1$, $G\curvearrowright X_2$ are proper with contracting property, and $H\le G$ is a confined subgroup with a confining subset disjoint with $E(G)$. Then $G$ has marked length spectrum rigidity from $H$.
\end{theorem}
\begin{proof}
    By \autoref{DEF: Confined Subgroup}, there exists a finite set $P\subseteq G\setminus E(G)$ such that $gHg^{-1}\cap P\neq \varnothing $ for all $g\in G$. Now we suppose that
    \begin{align}\label{THMConfinedEQU1}
        \ell_{d_1}(h)=\ell_{d_2}(h),\quad \forall h\in H.
    \end{align}

    Let $N=2\sharp P+3$. Consider the action $G\curvearrowright X_1$ and pick a collection of $N$ contracting elements $\left\{h_1,\cdots,h_N\right\}$ pairwise weakly independent in $X_1$. For each triple $1\leq i<j<k\leq N$, let $\tau_{i,j,k}$ be decided by Extension Lemma (\autoref{LEM: Modified Extension Lemma}) for the three elements $h_i,h_j,h_k$ and let $\tau=\max_{1\leq i<j<k\leq N} \tau_{i,j,k}$. Let $C$ be a contraction constant for the contracting system $\mathbb{Y}=\left\{\left\langle h_i\right\rangle o_1\mid 1\leq i\leq N\right\}$, and let $D=D(\tau,C)$, $\epsilon=\epsilon(\tau,C)$ be decided by \autoref{PROP: Fellow Travel}. Choose $f_i\in\left\langle h_i\right\rangle$ so that $d_1(o_1,f_io_1)>D$ for each $1\le i\le N$.

    Let $A=\left\{g\in G\mid d_1(o_1,go_1)\leq D\right\}$, which is a finite subset of $G$. According to \autoref{COR: ER3}, there exists an element $g_0\in G$ such that $g_0Pg_0^{-1}\cap A=\varnothing$. We denote $Q=g_0Pg_0^{-1}$. Then $d_1(o_1,qo_1)>D$ for all $q\in Q$, and $gHg^{-1}\cap Q=g_0(g_0^{-1}gHg^{-1}g_0\cap P)g_0^{-1}\neq \varnothing$ for all $g\in G$.

    We denote $M=\max_{1\le i\le N} d_2(o_2,f_io_2)$ and $K=\max_{q\in Q} d_2(o_2,qo_2)$.

    Now, we will show that there exists $c\geq 0$ such that $d_1(o_1,go_1)\leq d_2(o_2,go_2)+c$ for all $g\in G$. We consider two cases.

    \textbf{Case 1:} $d_1(o_1,go_1)>D$.

    For each element $a\in Q\cup\left\{g\right\}$, according to \autoref{LEM: Modified Extension Lemma} (\ref{ext1}) and by Pigeonhole Principle, there exists at least $N-2$ indices $i\in \left\{1,\cdots, N\right\}$ such that $\max\left\{\d{\left\langle h_i\right\rangle o_1}(o_1,ao_1),\d{\left\langle h_i\right\rangle o_1}(o_1,a^{-1}o_1)\right\}\le \tau$. However, $N=2\sharp Q+3\geq 2\sharp (Q\cup\left\{g\right\})+1$. By Pigeonhole Principle, there exists an index $1\le j\le N$ such that $\max\left\{\d{\left\langle h_j\right\rangle o_1}(o_1,ao_1),\d{\left\langle h_j\right\rangle o_1}(o_1,a^{-1}o_1)\right\}\le \tau$ holds simultaneously for all $a\in Q\cup\left\{g\right\}$.

    Let $f=f_j$. Since $(gf)^{-1}H(gf)\cap Q\neq \varnothing$, we can find $q\in Q$ such that $gfqf^{-1}g^{-1}\in H$. Similarly, there exists $q'\in Q$ such that $fq'f^{-1}\in H$. Consider the periodic path $\gamma$ labelled by $(\cdots, g,f,q,f^{-1},g^{-1},f,q',f^{-1}, g,f,q,f^{-1},g^{-1},f,q',f^{-1},\cdots)$. By construction, $\gamma$ is a special $(D,\tau)$-admissible path  with respect to  $\mathbb{X}=\left\{u\left\langle h_j\right\rangle o_1\mid u\in G\right\}$. According to \autoref{PROP: Fellow Travel}, $\gamma$ has $\epsilon$-fellow travel property. We denote $h=gfqf^{-1}g^{-1}fq'f^{-1}$, then $h\in H$. Let $\alpha=[o_1,ho_1]$ be a geodesic segment. Then there exists three linearly ordered points $x,y,z\in \alpha$ such that $d_1(go_1,x)\leq\epsilon$, $d_1(gfqf^{-1}o_1,y)\leq \epsilon$ and $d_1(gfqf^{-1}g^{-1}o_1,z)\leq \epsilon$. Therefore, 
    \begin{align}\label{THMConfinedEQU2}
        2d_1(o_1,go_1)\leq d_1(o_1,x)+\epsilon + d_1(y,z)+2\epsilon\leq d_1(o_1,ho_1)+2\epsilon.
    \end{align}
    For $n>0$, let $\beta=[o_1,h^no_1]$ be a geodesic segment. Then there exists $w\in \beta$ such that $d_1(h^{n-1}o_1,w)\leq \epsilon$. Thus, $d_1(o_1,h^no_1)=d_1(o_1,w)+d_1(w,h^no_1)\geq d_1(o_1,h^{n-1}o_1)+d_1(h^{n-1}o_1,h^{n}o_1)-2\epsilon$. Hence,
    \begin{align}\label{THMConfinedEQU3}
        \ell_{d_1}(h)\geq d_1(o_1,ho_1)-2\epsilon.
    \end{align}

    Therefore, 
    \begin{align*}
        2d_1(o_1,go_1)\leq & d_1(o_1,ho_1)+2\epsilon&(\ref{THMConfinedEQU2})\\
        \leq &\ell_{d_1}(h)+4\epsilon &(\ref{THMConfinedEQU3})\\
        =&\ell_{d_2}(h)+4\epsilon &(\ref{THMConfinedEQU1})\\
        \leq& d_2(o_2,ho_2)+4\epsilon & (\text{\autoref{RMK: MLS Basics}\,(\ref{MLSBasics2})})\\
        \leq& 2d_2(o_2,go_2)+4M+2K+4\epsilon & (\text{$\Delta$-Ineq.})
    \end{align*}

    \textbf{Case 2:} $d_1(o_1,go_1)\le D$.

    Then obviously, $d_1(o_1,go_1)\leq d_2(o_2,go_2)+D$.

    Combining these two cases, we let $c=\max\left\{ 2M+K+{2}\epsilon, D\right\}$, and $d_1(o_1,go_1)\leq d_2(o_2,go_2)+c$ holds for all $g\in G$.

    Similarly, by switching $(X_1,d_1)$ and $(X_2,d_2)$, there exists $c'\ge 0$ such that $d_2(o_2,go_2)\leq d_1(o_1,go_1)+c'$ for all $g\in G$.

    Therefore, $\sup_{g,h\in G}\abs{d_1(go_1,ho_1)-d_2(go_2,ho_2)}\le \max\left\{c,c'\right\}<+\infty$.
\end{proof}

\begin{remark}
    Noticing that $E(G)$ is invariant under conjugation, we derive that if $H\le G$ is a confined subgroup and $H\cap E(G)=\left\{1\right\}$, then the minimal confining subset of $H$ must be disjoint with $E(G)$.
\end{remark}

\section{MLS rigidity from geometrically dense subgroups}\label{Sec: MLSR from GD subgroups}

We continue to investigate MLS rigidity from a larger class of subgroups, called geometrically dense subgroups, in a sense of Osin. Roughly speaking, such a subgroup have a maximal limit set as the ambient group, on the horofunction boundary of the metric space. Unbounded normal subgroups are geometrically dense, and moreover, it is proved in \cite[Lemma 5.10]{CGYZ} that confined subgroups (with confining subset disjoint with $E(G)$) are geometrically dense in a proper action.  

\subsection{Horofunction boundary}

\begin{definition}
    Let $(X,d)$ be a proper geodesic metric space with a basepoint $o\in X$. For each $y\in X$, we define a Lipschitz map $b^o_y: X\to \mathbb R$ by $$b^o_y(x)=d(x,y)-d(o,y).$$

    Endowed with the compact-open topology,   the closure of $\{b^o_y: y\in X\}$ gives a compactification of $X$ by Arzela-Ascoli Lemma. The complement $\overline{X}\setminus X$, denoted by $\partial_h X$, is called the \textit{horofunction boundary}. 
    
    Two horofunctions are said to be \textit{equivalent} if they differ by a bounded amount. This defines a \textit{finite difference relation} on $\partial_h X$, where we use $[\ksi]$ to denote the equivalence class under finite difference relation of $\ksi\in \horo{X}$, and the quotient $[\partial_h X]=\horo{X}/\sim $ is usually called \textit{reduced} horofunction boundary. Note that the quotient topology $[\partial_h X]$ might not be Hausdorff in general. 
\end{definition}

The topological type of horofunction boundary is independent of the choice of basepoints, since if
$d(x,y_n)-d(o,y_n)$ converges as $n\to \infty$, then so does $d(x,y_n)-d(o',y_n)$. Moreover, the corresponding
horofunctions differ by an additive amount:
$$b_{\xi}^o(\cdot)-b_{\xi}^{o'}(\cdot)=b_{\xi}^o(o'),$$
so we will omit the upper index $o$ in $b_{\xi}^o$ when it dose not matter.
In addition, for any $g\in \mathrm{Isom}(X)$, we can naturally extend continuously to an action on $\horo{X}$: $\forall x\in X, g\cdot b_{\xi}^o(y):=b_{g\xi}^o(y)$, and descend to an action on $\hor{X}$.

In \cite{Yang22}, the horofunction boundary with finite difference relation  is proved  to have the \textit{convergence property}.  Its general definition is not relevant here; in what follows, we only collect the relevant properties about horofunction boundary to be used later on.

We say that a sequence of points $(x_n)$ in $X$ \textit{accumulates} at $[\xi]\in\hor{X}$, if any accumulation point of $(x_n)$ in $\overline{X}$ lies in $[\xi]$. We denote this by $[\xi]=[\lim_{n\to+\infty}x_n]$.

\begin{proposition}[{\cite[Corollary 5.2]{Yang22}} ]\label{PROP: Endpoint of Contracting Quasi-geodesic}
    Let $(X,d)$ be a proper geodesic metric space, and suppose that $\gamma:\mathbb{R}\to X$ is a $C$-contracting quasi-geodesic. Then $\gamma(t)$ accumulates at $\hor{X}$ as $t\to +\infty$ and $t\to -\infty$ respectively. In particular, $[\gamma^+]:=[\lim_{t\to+\infty}\gamma(t)]$ and $[\gamma^-]:=[\lim_{t\to -\infty}\gamma(t)]$, where $[\gamma^+], [\gamma^-]\in \hor{X}$.
\end{proposition}

\begin{lemma}\label{LEM: Projection to Infinity}
    Under the assumption of \autoref{PROP: Endpoint of Contracting Quasi-geodesic}, if there is a sequence of points $(x_n)$ in $X$ such that $[\gamma^+]=[\lim_{n\to+\infty}x_n]$, then $\pi_{\gamma}(x_n)$ accumulates at $[\gamma^+]$. In particular, $\gamma^{-1}(\pi_{\gamma}(x_n))$ tends to $+\infty$.
\end{lemma}
\begin{proof}
    Fix a basepoint $o\in \gamma$, and suppose $\gamma$ is a $(\lambda,c)$-quasi-geodesic.
    
    Suppose on the contrary that $\pi_{\gamma}(x_n)$ does not accumulate at $[\gamma^+]$. Then there exists a sub-sequence of $(x_n)$, denoted by $(x_n')$, such that $\gamma^{-1}(\pi_{\gamma}(x_n'))$ has an upper-bound $T$. Thus, for each $r>0$ we can choose a sufficiently large $K(r)>0$, such that for any $t\ge K(r)$ and any $x_n'$, $d(\gamma{t},\pi_{\gamma}(x_n'))>r$.
    
    Furthermore, we can find a sub-sub-sequence of $(x_n'')$ such that $b^o_{x_n''}$ converge locally uniformly to a function $b^o_{\ksi}:X\to \mathbb{R}$, where $\ksi\in [\gamma^+]$. Take points $w_n=\gamma(t_n)\in \gamma$, such that $t_n\to +\infty$ and $b^o_{w_n}$ converge locally uniformly to $b^o_{\eta}$ where $\eta\in [\gamma^+]$. By definition, there exist $C_0:=\parallel b^o_{\eta}-b^o_{\ksi}\parallel_{\infty} <+\infty$. Then for any point $z=\gamma(s)$ with $s\gg 0$, $b^o_{w_n}(z)<0$ for any $n\gg 0$; therefore, $b^o_{\eta}(z)\le 0$. Fix such a $z$ with $s>K(4C+C_0+\lambda T+c+1)$.

    For each $x_n''$, choose $y_n''\in \pi_{\gamma}(x_n'')$. Then by \autoref{LEM: Geodesic Along Projection}, $b^o_{x_n''}(z)=d(x_n'',z)-d(o,x_n'')\ge d(x_n'',y_n'')+d(y_n'',z)-4C-(d(o,y_n'')+d(y_n'',x_n''))=d(y_n'',z)-d(o,y_n'')-4C\ge C_0+1$. Therefore, $b^o_{\ksi}(z)>C_0+1$.

    However, $b^o_{\eta}(z)\le 0$, which is contradictory to $\parallel b^o_{\eta}-b^o_{\ksi}\parallel_{\infty}=C_0$.
\end{proof}

As a special example, let $g$ be a contracting element acting on $X$, and let $o\in X$ be a basepoint. Then $\left\langle g\right\rangle o$ is a contracting quasi-geodesic. We denote by $[g^+_X]=[\lim_{n\to +\infty}g^no]\in \hor{X}$, and $[g^-_X]=[\lim_{n\to -\infty}g^no]\in \hor{X}$, which is independent with the choice of the basepoint $o$. It is easy to see that $[g^+_X]\ne [g^-_X]$ are distinct (e.g. by \cite[Lemma 5.5]{Yang22}).

\begin{proposition}[{\cite[Lemma 5.1]{Yang22}} ]\label{PROP: Assump A}
    Let $\gamma:\mathbb{R}\to X$ be a contracting quasi-geodesic. If for a sequence of points $(x_n)$ in X, $\gamma^{-1}(\pi_{\gamma}(x_n))$ tends to $+\infty$ as $n\to +\infty$, then $[\lim_{n\to+\infty}x_n]=[\gamma^+]$.
\end{proposition}

For any subset $A\subseteq X$, denote $\Lambda A=\horo{X}\cap \overline A$. \cite[Lemma 3.22]{Yang22} implies the following proposition. Recall that for any $x,y\in X$, $\d{A}(x,y)=\diam(\pi_A(\{x,y\}))$.
\begin{proposition}[{\cite[Lemma 3.22]{Yang22}} ]\label{PROP: Projection of Infinite Point}
    Let $(X,d)$ be a proper geodesic metric space and $A\subseteq X$ be a $C$-contracting quasi-geodesic. Then there exists $\Kappa=\Kappa(C)\geq 0$, such that for any $[\ksi]\in \hor{X}\setminus [\Lambda A] $, and any sequence $(x_n)$ in $X$ accumulating at $[\ksi]$, $$\limsup_{n,m\to +\infty}\d{A}(x_n,x_m)\leq  \Kappa.$$
\end{proposition}

As a corollary,  we can extend the shortest projection $\d{A}$ to a projection map $\pi_A([\ksi])$ defined for each $[\ksi]\in \hor{X}\setminus [\Lambda A] $.
\begin{definition}
    Let $(X,d)$ be a proper geodesic metric space and $A\subseteq X$ be a $C$-contracting quasi-geodesic. For any $[\ksi]\in \hor{X}\setminus [\Lambda A] $, we define $\pi_A([\ksi]):=\bigcup \pi_A(x_N)$, where the union is taken over all choices of $(x_n)_{n=1}^{+\infty}\subseteq X$, such that $[\lim_{n\to+\infty}x_n]=[\ksi]$ and $\d{A}(x_n,x_m)\leq \Kappa(C)+\frac{1}{2}$ for any $n,m\geq N$.
\end{definition}

By \autoref{PROP: Projection of Infinite Point}, $\diam(\pi_A([\ksi]))\leq 2\Kappa+1$; and as a direct observation, for any $g\in\mathrm{Isom}(X)$, $g\cdot \pi_A([\ksi])=\pi_{g\cdot A}(g\cdot[\ksi])$.

\subsection{MLS rigidity from geometrically dense subgroup}

\text{ }

In this subsection, we suppose that a group $G$ acts by isometries on two proper geodesic metric spaces $(X_1,d_1)$ and $(X_2,d_2)$ with fixed basepoints $o_1\in X_1$ and $o_2\in X_2$, and with {contracting property} respectively.

As before, $\mathcal{SC}(G)=\{ g\in G\mid g\text{ is contracting on both }X_1\text{ and }X_2\}$.

The limit set of a subgroup $K<G$ is defined as $[\Lambda_{X_i} K]:=[\overline{Ko_i}\cap \horo{X_i}]$ ($i=1,2$).

Following Osin \cite{Osi22}, a subgroup $H\le G$ is called \textit{geometrically dense} if $[\Lambda_{X_i} H]=[\Lambda_{X_i} G]$ for $i=1,2$. For instance, any normal subgroup with a contracting element is a geometrically dense subgroup \cite[Lemma 3.8]{Yang22}.

Throughout this subsection, we fix a geometrically dense subgroup $H\le G$ with contracting property acting on both $X_1$ and $X_2$.

\begin{definition}\label{Condition B}
    We say that an \hyperref[Condition B]{\textbf{Extension Condition}} is satisfied for $H$, if for any $g\in \mathcal{SC}(G)$, there is a sequence $\left\{h_n\right\}_{n=1}^{+\infty}$ in $H$ such that $[\lim h_no_1]=[g^+_{X_1}]$ and $[\lim h_no_2]=[g^+_{X_2}]$.
\end{definition}

In the following, we will show that \hyperref[Condition B]{{Extension Condition}} is satisfied in various situations.

\begin{definition}
    Let a group $\Gamma$ act properly on a proper geodesic metric space $X$ by isometry. For any $x,y\in X$ and $r>0$, the \textit{cone} and \textit{shadow} are defined as:
\begin{align*}
    &\Omega_x(y,r)_X:=\left\{ z\in X\mid \exists \,\text{a geodesic segment}\,[x,z]\text{ such that}\,[x,z]\cap N_r(y)\neq \varnothing\right\};\\
    &\Pi_x(y,r)_X:=\overline{\Omega_x(y,r)_X}\cap \horo{X}.
\end{align*}

The \textit{conical limit set} of $G$ is defined as $$[\Lambda^c_{X}G]:=\left [\bigcup_{\gamma\in\Gamma ,r> 0}\limsup_{\beta\in \Gamma }\Pi_{\gamma\cdot o}(\beta\cdot o,r)_X\right ]\subseteq \hor{X}$$ for arbitrarily chosen basepoint $o\in X$. Here the $\limsup$ of a collection of sets is defined as $\limsup_{i\in I} A_i:=\left\{x\in \bigcup_{i\in I}A_i\mid \exists \,\text{infinitely many}\,j\in I\,\text{such that}\,x\in A_j\right\}.$
\end{definition}

\begin{lemma}\label{Lem: Conditions for extension property}
    Let $H\le G$ a geometrically dense subgroup  with contracting property acting on both $X_1$ and $X_2$. Then \hyperref[Condition B]{Extension Condition} is satisfied for $H$ if one of the following holds:
    \begin{enumerate}[(1)]
        \item $X_1$ and $X_2$ are Gromov hyperbolic geodesic spaces, along with a $G$-coarsely equivariant quasi-isometric-embedding $f:X_1\to X_2$.
        \item $H$ is a normal subgroup.
        \
        \item For any sequence $\left\{h_n\right\}$ in $H$ and any $g\in \mathcal{SC}(G)$, $\limsup_{n\to\infty}\gp{h_no_1}{g^no_1}{o_1,d_1}=+\infty$ if and only if $\limsup_{n\to\infty}\gp{h_no_2}{g^no_2}{o_2,d_2}=+\infty$.
        \item The action $G\curvearrowright X_1$ is proper, $[\Lambda^c_{X_1}H]=[\Lambda^c_{X_1}G]$, and the orbit map $\rho:Go_1\to Go_2$ is coarsely Lipschitz: for each $R\geq 0$ there exists $S=S(R)\geq 0$ such that for any $g,g'\in G$ with $d_1(go_1,g'o_1)\leq R$, we have $d_2(go_2,g'o_2)\leq S$. 
    \end{enumerate}    
\end{lemma}

\begin{proof}
    
        \textbf{(1).} In this case, the reduced horofunction boundary $\hor{X_i}$ is homeomorphic to the Gromov boundary $\partial X_i$. 
        For each $g\in \mathcal{SC}(G)$, since $H$ is geometrically dense, we can find a sequence $\left\{h_n\right\}$ in $H$ so that $[\lim h_no_1]=[g^+_{X_1}]$. Therefore, $\lim_{n,m\to+\infty} \gp{h_no_1}{g^mo_1}{o_1,d_1}=+\infty$.
        
        Suppose $f:X_1\to X_2$ is a $G$-coarsely equivariant $(\lambda,c)$-quasi-isometric-embedding, and without loss of generality that $o_2=f(o_1)$. By \ref{Gromov Product Preserved by Quasi-isometry}, there exists $C\geq 0$ such that $\frac{1}{\lambda}\gp{f(x)}{f(y)}{f(z),d_2}-C\le \gp{x}{y}{z,d_1}\le{\lambda}\gp{f(x)}{f(y)}{f(z),d_2}+C $. Therefore, $\lim_{n,m\to+\infty} \gp{h_no_2}{g^mo_2}{o_2,d_2}=+\infty$, which implies that $[\lim h_no_2]=[g^+_{X_2}]$.

        \textbf{(2).} 
        According to \autoref{LEM: Simultaneous Contracting}, we can choose $k\in \mathcal{SC}(G)\cap H$. According to the north-south dynamics of contracting elements (cf. \cite[Lemma 3.19]{Yang22}), for each $g\in \mathcal{SC}(G)$, $g^n\cdot[k^+_{X_i}]$ tends to $[g^+_{X_i}]$ as $n\to +\infty$ ($i=1,2$). 
        In addition, $g^n\cdot[k^+_{X_i}]=[\lim_{m\to+\infty}g^nk^mo_i]=[\lim_{m\to+\infty}g^nk^m(g^{-n}o_i)]=[\lim_{m\to+\infty}(g^nkg^{-n})^mo_i]$ ($i=1,2$). 
        As the compactification $\overline{X_i}=X_i\cup\horo{X_i}$ is metrizable, for countable open basis $\left \{U_j^i\right \}_{j=1}^{+\infty}$ of $[g^+_{X_i}]$ in $\overline{X_i}$,
        we can choose $n_j,m_j\gg 0$ such that $(g^{n_j}kg^{-n_j})^{m_j}o_i\in U_j^i$.
        Hence, $(g^{n_j}kg^{-n_j})^{m_j}\in H$ and $[\lim_{j\to+\infty } (g^{n_j}kg^{-n_j})^{m_j}o_i]=[g^+_{X_i}]$ for $i=1,2$.

        \textbf{(3).} 
        For each $g\in \mathcal{SC}(G)$, since $H$ is geometrically dense, we can first choose a sequence $\left\{h_n\right\}$ in $H$ such that $[\lim_{n\to \infty} h_no_1]=[g^+_{X_1}]$. Choose a point $g^{l_n}o_1\in \pi_{\left\langle g\right\rangle o_1}(h_no_1)$ and also a point $g^{t_n}o_2\in \pi_{\left\langle g\right\rangle o_2}(h_no_2)$. \autoref{LEM: Projection to Infinity} proves that $l_n\to +\infty$. Without loss of generality, we suppose that $l_n>0$ for each $n\in\mathbb{N}^+$. By \autoref{LEM: Geodesic Along Projection}, there exists $C\ge 0$ depending only on $g$ and the choice of the basepoints such that $d_1(o_1,h_no_1)\sim_Cd_1(o_1,g^{l_n}o_1)+d_1(g^{l_n}o_1,h_no_1)$ and $d_1(g^no_1,h_no_1)\sim_Cd_1(g^no_1,g^{l_n}o_1)+d_1(g^{l_n}o_1,h_no_1)$. In addition, by \autoref{LEM: Fellow Travel Property}, there exists $\epsilon\ge 0$ relying only on $g$ and the choice of the basepoints such that $d_1(o_1,g^no_1)\sim_\epsilon d_1(o_1,g^{\min \left\{l_n,n\right\}}o_1)+d_1(g^{l_n}o_1,g^no_1)$.
        
        Therefore, $\gp{h_no_1}{g^no_1}{o_1,d_1}\sim_{C,\epsilon}d_1(o_1,g^{\min \left\{l_n,n\right\}}o_1)$, which tends to $ +\infty$ as $n\to +\infty$. By assumption, $\limsup_{n\to\infty}\gp{h_no_2}{g^no_2}{o_2,d_2}=+\infty$. By choosing a sub-sequence indexed by $n_k$, we may assume $\lim_{k\to \infty}\gp{h_{n_k}o_2}{g^{n_k}o_2}{o_2,d_2}=+\infty$. Furthermore, a similar calculation shows that $\gp{h_{n_k}o_2}{g^{n_k}o_2}{o_2,d_2}\sim_{C,\epsilon}d_2(o_2,g^{\min \left\{t_{n_k},n_k\right\}}o_2)$. Therefore, we must have $t_{n_k}\to+\infty$ as $k\to+\infty$. According to \autoref{PROP: Assump A}, $[\lim_{k\to+\infty} h_{n_k}o_2]=[g^+_{X_2}]$. Thus $\left\{h_{n_k}\right\}_{k=1}^{+\infty}$ is the desired sequence in $H$.
        
        \textbf{(4).} For each $g\in \mathcal{SC}(G)$, we recall the well-known fact that $[g^+_{X_1}]\in [\Lambda^c_{X_1}G]$. In fact, by Ascoli-Arzela lemma, we can find a sequence of increasing positive integers $\left\{n_k\right\}$ such that $[o_1,g^{n_k}o_1]$ converges locally uniformly to a geodesic ray $\gamma:[0,+\infty)\to X_1$. In addition $[\lim_{t\to+\infty}\gamma(t)]=[g^+_{X_1}]$. According to \autoref{LEM: Fellow Travel Property}, there exists $\epsilon> 0$ relying only on $g$ and the choice of the basepoints such that $d_1(g^no_1,\gamma)< \epsilon$ for all $n\in\mathbb{N}^+$. Therefore, $[g^+_{X_1}]\in \left[\limsup_{\beta\in G}\Pi_{o_1}(\beta o_1,\epsilon)\right ]\subseteq [\Lambda^c_{X_1}G]$. Furthermore, combined with \autoref{LEM: Quasi Convex}, $d_H(\left\{g^n o_1|n\ge 0\right\},\gamma)<+\infty$.
        
        By assumption, $[g^+_{X_1}]\in [\Lambda^c_{X_1}H]$. Therefore, there exists $k\in H$ and $r> 0$ such that for infinitely many $h\in H$, indexed as $\left\{h_n\right\}$, $[g^+_{X_1}]\in \left [\Pi_{ko_1}(h_no_1,r)\right ]$. For any such $h_n$, by Ascoli-Arzela lemma, we can construct an infinite geodesic $\gamma_{h_n}:[0,+\infty)\to X_1$ with $[\lim_{t\to+\infty}\gamma_{h_n}(t)]=[g^+_{X_1}]$ and $\gamma_{h_n}\cap N(h_no_1,r)\neq \varnothing$. 
        By quasi-convexity of contracting subsets (cf. \autoref{LEM: Quasi Convex}) or by a direct application of \autoref{LEM: Thin Quadrilateral}, there exists $K>0$  relying only on $g$ and the choice of the basepoints such that $\gamma_{h_n}\subseteq N_{K}(\left\{g^n o_1|n\ge 0\right\})$ for all $n$. Therefore, there exists $M>0$ such that $\gamma_{h_n}\subseteq N_{M}(\gamma)$, and hence $\gamma\cap N(h_no_1,r+M)\neq \varnothing$ for all $n$.
        
        Since the action $G\curvearrowright X_1$ is proper, we must have $d_1(o_1,h_no_1)\to +\infty$. Therefore, there exist a sequence of points $g^{l_n}o_1$ with $l_n\ge 0$ and $d_1(h_no_1,g^{l_n}o_1)\le r+M$. Hence $l_n\to+\infty$ and $\pi_{\left\langle g\right\rangle o_1}({h_no_1})$ tends to $[g^+_{X_1}]$. By \autoref{PROP: Assump A}, $[\lim h_no_1]=[g^+_{X_1}]$.
        
        Let $S=S(r+M)$ be given by the coarsely-Lipschitz condition, then $d_2(h_no_2,g^{l_n}o_2)\le S$ for all $n$. Therefore, $\pi_{\left\langle g\right\rangle o_1}({h_no_1})$ lies in a uniform neighbourhood of $g^{l_n}o_2$ and also tends to $[g^+_{X_1}]$. By \autoref{PROP: Assump A} again, $[\lim h_no_2]=[g^+_{X_2}]$, and is the desired sequence in $H$.
\end{proof}

However, without these additional assumptions, \hyperref[Condition B]{Extension Condition} may not be satisfied as shown by the following examples.

\begin{example}\label{Example without extension property}
    Let $X$ denote a proper $\delta$-hyperbolic geodesic space and $\partial X$ denote its Gromov boundary. Suppose $\sharp \partial X>2$. In this case $\hor{X}$ is identical with $\partial X$.
    
    Fix an isometric group action $H\curvearrowright X$. Let $G=H\ast \mathbb{Z}$ be a free product, and $f$ be a generator of $\mathbb{Z}$. By determining different actions by $f$ on $X$, we can construct different isometric group actions $G\curvearrowright X$. 
    \begin{enumerate}
        \item Assume the action $H\curvearrowright X$ is non-elementary, so that we can choose two weakly independent loxodromic elements $h,k\in H$.

        The first action $G\curvearrowright X$, denoted by a dot symbol ``$\cdot$'', is defined as $\phi \cdot x= \phi x$ and $f\cdot x= hx$, for all $\phi \in H$ and $x\in X$.

        The second action $G\curvearrowright X$, denoted by a star symbol ``$\star$'', is defined as $\phi \star x=\phi x$ and $f\star x= kx$, for all $\phi \in H$ and $x\in X$.

        Then it is easy to verify that $G\cdot o=Ho=G\star o$, hence $\Lambda_1 G=\Lambda_2 G=\Lambda H$. The orbit map $\rho|_{Ho}=id$ and extends to $\tilde{\rho}_H=id:\Lambda H\to \Lambda H$.
        However, $f\in \mathcal{SC}(G)$ and $\lim f^n\star o=k^+\neq id(h^+)=h^+=\lim f^n\cdot o$. 

        \item Assume the action $H\curvearrowright X$ is cocompact; then $\Lambda H=\partial X$. In such case, choose two weakly independent loxodromic elements $\alpha, \beta\in \mathrm{Isom}(X)$.

         The first action $G\curvearrowright X$, denoted by a dot symbol ``$\cdot$'', is defined as $\phi \cdot x= \phi x$ and $f\cdot x= \alpha x$, for all $\phi \in H$ and $x\in X$.

        The second action $G\curvearrowright X$, denoted by a star symbol ``$\star$'', is defined as $\phi \star x=\phi x$ and $f\star x= \beta x$, for all $\phi \in H$ and $x\in X$.

        Obviously, $\Lambda_1 G=\partial X=\Lambda_2 G=\Lambda H$. Still, the orbit map $\rho|_{Ho}=id$ and extends to $\tilde{\rho}_H=id:\Lambda H\to \Lambda H$.
        However, $f\in \mathcal{SC}(G)$ and $\lim f^n\star o=\beta^+\neq id(\alpha^+)=\alpha^+=\lim f^n\cdot o$.
    \end{enumerate}
\end{example}

The following theorem ensures a marked length spectrum rigidity from a wider range of subgroups compared with \autoref{THM: Normal Subgroup Rigidity}.

\begin{theorem}\label{THM: MLSR from Geometrically Dense Subgroup}
   Let $H\le G$ a geometrically dense subgroup  with contracting property acting on both $X_1$ and $X_2$.  If \hyperref[Condition B]{Extension Condition} is satisfied for $H$, then $G$ has marked length spectrum rigidity from $H$.
\end{theorem}

Before giving the proof, we need the following lemma.

\begin{lemma}\label{LEM: Gromov Product of Horofunction Boundary}
    Suppose $g$ is a contracting element acting isometrically on a proper geodesic metric space $(X,d)$ with a basepoint $o\in X$. Then there exists $R=R(g,o)\geq 0$ such that for any sequence $(y_n)$ in $X$ that accumulates at $[g^+_X]$ and $(x_n)$ in $X$ that accumulates at $[\ksi]\in \hor{X}\setminus\left\{[g^+_X],[g^-_X]\right\}$, we have:
    \begin{enumerate}
        \item If there exists $l\leq 0$ such that $g^lo\in \pi_{\left\langle g\right\rangle o}([\ksi])$, then $$\limsup_{n\to+\infty}\gp{x_n}{y_n}{o}\sim_R 0.$$
        \item If there exists $l\geq 0$ such that $g^lo\in \pi_{\left\langle g\right\rangle o}([\ksi])$, then $$\limsup_{n\to+\infty}\gp{x_n}{y_n}{o}\sim_R d(o,\pi_{\left\langle g\right\rangle o}([\ksi])).$$
    \end{enumerate}
\end{lemma}

\begin{proof}
    Let $C$ be a contraction constant of $\left\langle g\right\rangle o$, and let $\epsilon>0$ be given as in \autoref{LEM: Fellow Travel Property} for the contracting element $g$. Fix $g^lo\in \pi_{\left\langle g\right\rangle o}([\ksi])$, then $d_H(g^lo,\pi_{\left\langle g\right\rangle o}([\ksi]))\sim_C0$.

    For sufficiently large $n$, we may assume that $d_H(g^lo,\pi_{\left\langle g\right\rangle o}(x_n))\sim_C0$.
    
    According to \autoref{LEM: Projection to Infinity}, for sufficiently large $n$, we may assume that $d(\pi_{\left\langle g\right\rangle o}(x_n),\pi_{\left\langle g\right\rangle o}(y_n))>C $, and for any $g^{r_n}o\in \pi_{\left\langle g\right\rangle o}(y_n)$, $r_n>\max\left\{l,0\right\}$.

    For $n\gg 0$, by \autoref{LEM: Thin Quadrilateral},  $d(x_n,y_n)\sim_C d(x_n,g^lo)+d(g^lo,g^{r_n}o)+d(g^{r_n}o,y_n)$. Furthermore, by \autoref{LEM: Geodesic Along Projection}, $d(x_n,o)\sim_C d(x_n,g^lo)+d(o,g^lo)$ and $d(y_n,o)\sim_C d(y_n, g^{r_n}o)+d(g^{r_n}o,o)$.

    Therefore, $\gp{x_n}{y_n}{o}\sim_C \gp{g^lo}{g^{r_n}o}{o}$.

    If $l\leq 0$, according to \autoref{LEM: Fellow Travel Property}, we can find a point $z\in [g^lo,g^{r_n}o]$ such that $d(o,z)\leq \epsilon$. Thus $\gp{g^lo}{g^{r_n}o}{o}\sim_\epsilon 0$.

    If $l \geq 0$, then we can find $z'\in [o,g^{r_n}o]$ such that $d(g^lo,z')\leq \epsilon$. Thus $\gp{g^lo}{g^{r_n}o}{o}\sim_\epsilon d(o,g^lo)\sim_C d(o,\pi_{\left\langle g\right\rangle o}([\ksi]))$.

    A choice of $R=R(C,\epsilon)$ will finish the proof of this lemma.
\end{proof}

\begin{corollary}\label{COR: Stable Length from Gromov Product}
    Suppose $g$ is a contracting element on a proper geodesic metric space $(X,d)$ with a basepoint $o\in X$. Fix a sequence $(y_n)$ in $X$ that accumulates at $[g^+_X]$. Choose any element $[\ksi]\in\hor{X}\setminus\left\{[g^+_X],[g^-_X]\right\}$, and for each $m\geq 0$, choose a sequence $(x_{n,m})$ in $X$ that accumulates at $g^m\cdot[\ksi]$ as $n\to +\infty$. Then: $$\ell_d(g)=\lim_{m\to +\infty}\frac{\limsup_{n\to +\infty}\gp{x_{n,m}}{y_n}{o}}{m}.$$
\end{corollary}

\begin{proof}
    Let $R\geq 0$ be decided as in \autoref{LEM: Gromov Product of Horofunction Boundary}, and $C$ be a contraction constant of $\left\langle g\right\rangle o$. 

    Since $[g^+_X],[g^-_X]$ are the only fixed points of $g$ in $\hor{X}$, $g^m\cdot [\ksi]\notin \left\{[g^+_X],[g^-_X]\right\}$ for each $m$. In addition, $\left\langle g\right\rangle o$ is $g$-invariant. By definition, up to a bounded error of $K=2\kappa(C)+1$, we have $d_H(\pi_{\left\langle g\right\rangle o}(g^m[\ksi]),g^m\cdot \pi_{\left\langle g\right\rangle o}([\ksi]))\leq 2K$. In particular, for sufficiently large $m$, and any $g^lo\in \pi_{\left\langle g\right\rangle o}(g^m[\ksi])$, we have $l>0$.

    Hence, for $m\gg 0$, $\limsup_{n\to +\infty}\gp{x_{n,m}}{y_n}{o}\sim_R d(o,\pi_{\left\langle g\right\rangle o}(g^m[\ksi]))\sim_C d(o,g^m\cdot \pi_{\left\langle g\right\rangle o}([\ksi]))$, and $\pi_{\left\langle g\right\rangle o}([\ksi])$ is a bounded subset in $X$.

    Therefore, $\ell_d(g)=\lim_{m\to +\infty}\frac{d(o,g^m\cdot \pi_{\left\langle g\right\rangle o}([\ksi]))}{m}=\lim_{m\to +\infty}\frac{\limsup_{n}\gp{x_{n,m}}{y_n}{o}}{m}$.
\end{proof}

\begin{proof}[Proof of \autoref{THM: MLSR from Geometrically Dense Subgroup}]
    Suppose the two actions have the same marked length spectrum on $H$. According to \autoref{LEM: Simultaneous Contracting}, $\mathcal{SC}(G)\cap H$ contains infinitely many elements that are weakly independent in both actions. Pick up five of them, then according to \autoref{LEM: Weakly Independent One In Three}, for each $g\in \mathcal{SC}(G)$, we can find $h\in \mathcal{SC}(G)\cap H$ amongst the five chosen elements such that $h$ is weakly independent with $g$ in both actions. In particular, for each $i=1,2$, $[h^+_{X_i}]\notin \left\{[g^+_{X_i}],[g^-_{X_i}]\right\}$.

    For each $m\geq 0$ and $i=1,2$, $g^m\cdot [h^+_{X_i}]=[(g^mhg^{-m})^+_{X_i}]$ and $g^mhg^{-m}\in\mathcal{SC}(G)$. 
    Therefore, by \hyperref[Condition B]{Extension Condition}, we can choose a sequence of elements $(k_{n,m})$ in $H$, such that $(k_{n,m}o_i)$ accumulates at $g^m\cdot [h^+_{X_i}]=[(g^mhg^{-m})^+_{X_i}]$ as $n\to +\infty$ ($i=1,2$).

    Also, as $g\in \mathcal{SC}(G)$, we fix a sequence of elements $(h_n)$ in $H$, so that $(h_no_i)$ accumulates at $[g^+_{X_i}]$ ($i=1,2$) by \hyperref[Condition B]{Extension Condition}. 

    Therefore, according to \autoref{COR: Stable Length from Gromov Product}, we have:
    \begin{align*}
        \ell_{d_1}(g)=\lim_{m\to+\infty}\frac{\limsup_{n\to+\infty}\gp{k_{n,m}o_1}{h_no_1}{o_1,d_1}}{m};\\
        \ell_{d_2}(g)=\lim_{m\to+\infty}\frac{\limsup_{n\to+\infty}\gp{k_{n,m}o_2}{h_no_2}{o_2,d_2}}{m}.
    \end{align*}

    In addition, according to \autoref{THM: Main Rigidity}, the orbit map $\rho|_{Ho_1}:Ho_1\to Ho_2$ is a $(1,c)$-quasi-isometry, implying that for $k_{n,m},h_n\in H$, $$\abs{\gp{k_{n,m}o_1}{h_no_1}{o_1,d_1}-\gp{k_{n,m}o_2}{h_no_2}{o_2,d_2}}\leq \frac{3}{2}c.$$

    After taking limit,  the above right-hand sides of $\ell_{d_1}(g)$ and $\ell_{d_1}(g)$ are thus equal, so $\ell_{d_1}(g)=\ell_{d_2}(g)$ for all $g\in \mathcal{SC}(G)$.

    According to \autoref{COR: Rigidity from SC}, there exists $C\geq 0$ such that $\abs{d_1(go_1,g'o_1)-d_2(go_2,g'o_2)}\leq C$, for all $g,g'\in G$.
    Hence the orbit map $\rho: Go_1\to Go_2$ is well-defined up to a bounded error, and is a rough isometry. 
\end{proof}

In Gromov hyperbolic spaces, the reduced horofunction boundary is identical with the Gromov boundary. In addition, a geometrically dense subgroup of a non-elementary isometry group is also non-elementary. Thus, combining \autoref{THM: Cusp Uniform} and \autoref{THM: MLSR from Geometrically Dense Subgroup}, we reach the following conclusion:
\begin{corollary}\label{COR: Cusp Uniform Geometrically Dense}
Suppose that a group $G$ is hyperbolic relative to a finite collection of subgroups $\left\{H_i\right\}_{i=1}^{N}$, and $(G,\left\{H_i\right\})$ admits two cusp-uniform actions on two proper $\delta$-hyperbolic geodesic spaces $(X_1,d_1)$, $(X_2,d_2)$ such that there exist a $G$-coarsely equivariant quasi-isometry $X_1\to X_2$. If $K<G$ is a geometrically dense subgroup and the two actions have the same marked length spectrum on $K$, then there exists a $G$-coarsely equivariant rough isometry $X_1\to X_2$.
\end{corollary}

As is pointed out in \cite{HH20}, the requirement that there exist a $G$-coarsely equivariant quasi-isometry $X_1\to X_2$ can be guaranteed if the two actions have constant horospherical distortion.

\section{MLS rigidity from subsets}\label{Sec: MLSR from subsets}

Throughout this section, we consider two proper and cocompact actions $G\curvearrowright (X_1,d_1)$ and  $G\curvearrowright (X_2,d_2)$ with contracting property. Fix basepoints $o_1\in X_1$ and $o_2\in X_2$. 

As demonstrated in hyperbolic groups by Cantrell-Reyes \cite{CR23}, a continuous decreasing convex curve called {the} \textit{Manhattan curve} introduced by Burger \cite{Bur93} for a pair of proper actions is proven indispensable in obtaining MLS rigidity from subsets. Combined with our technology developed in previous sections, the section is attempting to  generalise some of works of Cantrell-Reyes \cite{CR23, CR23b} to groups with contracting elements.

\subsection{Manhattan curves for pairs of actions}
For $i=1,2$ and any $g\in G$, denote by $[g]$  the conjugacy class containing $g$ and by $\ell_i(g)$ the stable translation length of $g$ on $X_i$. We shall focus on the following two formal series associated with two proper group actions:
$$a, b\in \mathbb{R},\; \mathcal P(a,b)=\sum_{g\in G}e^{-ad_2(o_2,go_2)-bd_1(o_1,go_1)}, \quad \mathcal Q(a,b)=\sum_{[g]\in \mathbf{conj}(G)}e^{-a\ell_2(g)-b\ell_1(g)},$$
where $\mathbf{conj}(G)$ denotes the set of all conjugacy classes in $G$.

For any given $a\in \mathbb{R}$, denote $\theta(a)=\inf\left\{b\in \mathbb{R}\mid \mathcal{P}(a,b)<+\infty\right\}$ and $\Theta(a)=\inf\left\{b\in \mathbb{R}\mid \mathcal{Q}(a,b)<+\infty\right\}$.
In general, $\Theta(a)$ may be infinite if there are infinitely many conjugacy classes with bounded stable length. For instance, Conner \cite{Con00} constructed examples of groups with stable lengths accumulating at 0. We will impose the following  assumption on $(X,d)$ to get rid of this. 

\begin{definition}\label{Def: coarsely convex metric}
    A geodesic metric space $(X,d)$ is \textit{coarsely convex} if there exists a constant $C\ge 0$ such that for any two geodesic segments $[x_1,y_1]$ and $[x_2,y_2]$ with midpoints $m_1$ and $m_2$, respectively, in $X$, one has $d(m_1,m_2)\le \frac{1}{2}[d(x_1,x_2)+d(y_1,y_2)]+C$. 
\end{definition}

Recall that for any isometry $g$ on a geodesic metric space $(X,d)$ with a basepoint $o\in X$, $\mu_d(g)$ (resp. $\ell_d^o(g)$) is the minimal translation length (resp. algebraic length) of $g$ on $X$ (cf. \autoref{Algebraic and minimal length}).

\begin{lemma}\label{Lem: ConvexMetric}
    Suppose that a geodesic metric space $(X,d)$ is  $C$-coarsely convex. Then for any $g\in \mathrm{Isom}(X)$, $|\ell_d(g)-\mu_d(g)|\le 2C$. Moreover, if a group $G$ acts isometrically and cocompactly on $X$, then there exists a constant $C'>0$ such that $|\ell_d(g)-\ell_d^o(g)|\le C'$ for any $g\in G$.
\end{lemma}
\begin{proof}
    Let $g$ be any isometry on $X$. For arbitrary $x\in X$, let $m$ be the midpoint of $[x,gx]$. Then $gm$ is the midpoint of $[gx,g^2x]$. As the metric is $C$-coarsely convex, $d(m,gm)\le \frac{1}{2}d(x,g^2x)+C$. It follows from the arbitrary of $x$ that $\mu_d(g^2)\ge 2\mu_d(g)-2C$.

    By iterating this inequality, one has that $\mu_d(g^{2^n})\ge 2^n\mu_d(g)-(2^n-1)2C$. \autoref{EQU: MALS implies MLS} then implies that $\ell_d(g)\ge \mu_d(g)-2C$. As $\ell_d(g)\le \mu_d(g)$ is clear, it follows that $|\ell_d(g)-\mu(g)|\le 2C$.

    Moreover, the cocompact action $G\curvearrowright X$ shows that there exists a constant $D>0$ such that $X\subseteq N_D(Go)$, which implies that $|\mu_d(g)-\ell_d^o(g)|\le 2D$ for any $g\in G$. By setting $C'=2C+2D$, the conclusion follows.
\end{proof}

It is well-known that $\delta$-hyperbolic spaces \cite[Chap. 10, Cor. 5.2]{CDP06} and CAT(0) spaces \cite[Chap. II.2, Prop. 2.2]{BH13} are coarsely convex. From now on, we assume that $X_1,X_2$ are $C$-coarsely convex  for a fixed $C\ge 0$.

\begin{lemma}
    For any $a\in \mathbb{R}$, $\theta(a),\Theta(a)\in (-\infty,+\infty)$.
\end{lemma}
\begin{proof}
    Since the two actions $G\curvearrowright X_1$ and $G\curvearrowright X_2 $ are proper and cocompact, both $d_1(o_1,go_1)$ and $d_2(o_2,go_2)$ are quasi-isometric to a word metric, and consequently $\delta_1(G)<\infty$ and $\delta_2(G)<\infty$. Suppose \begin{equation}\label{EQU: Well Def Manhattan Curve}\frac{1}{\lambda} d_2(o_2,go_2)-c\le d_1(o_1,go_1)\le \lambda d_2(o_2,go_2)+c\end{equation} for some $\lambda\ge 1$ and $c\ge 0$. It follows from a direct calculation that $\mathcal P(a,b)<+\infty$ if $b>\lambda \abs{a}+\delta_1(G)$, and $\mathcal P(a,b)=+\infty$ if $b<-\lambda \abs{a}$. Thus, $-\infty<\theta(a)<+\infty$.

    According to \autoref{EQU: Well Def Manhattan Curve}, we have $\lambda^{-1}\ell_2(g)\le \ell_1(g)\le \lambda \ell_2(g)$. 
    As $X_1,X_2$ are $C$-coarsely convex, it follows from \autoref{Lem: ConvexMetric} that there exists a constant $C'>0$ such that $$\{[g]\in \mathbf{conj}(G): \ell_i(g)\le n\}\subseteq \{[g]\in \mathbf{conj}(G): \ell_i^{o_i}[g]\le n+C'\}.$$
    Recall that by \cite[Main Theorem]{GY22}, $\sharp \{[g]\in \mathbf{conj}(G): \ell_i^{o_i}[g]\le n\}\asymp e^{\delta_i(G)n}/n$. Thus there exists $D>0$ such that $\sharp \{[g]\in \mathbf{conj}(G): \ell_i(g)\le n\}\le De^{\delta_i(G)n}/n$. It follows from a similar calculation that $\mathcal Q(a,b)<+\infty$ if $b>\lambda \abs{a}+\delta_1(G)$, and $\mathcal Q(a,b)=+\infty$ if $b<-\lambda \abs{a}$. Thus, $-\infty<\Theta(a)<+\infty$.
\end{proof}

By H\"{o}lder inequality, $\theta(a)$ and $\Theta(a)$ are convex functions on $\mathbb R$, so they are continuous. Obviously, they are also decreasing functions. Recall that by \cite[Main Theorem]{GY22}, $\delta_i(G)=\delta^c_i(G)<\infty$. This implies that both curves $\{(a, \theta(a))\mid a\in \mathbb R\}$ and $\{(a, \Theta(a))\mid a\in \mathbb R\}$ pass through two points $(0,\delta_1(G))$ and $(\delta_2(G),0)$.

Following \cite{CR22}, we will call this curve $\{(a,\theta(a))\mid a\in \mathbb R\}$ the \textit{Manhattan curve} for the pair $(d_1,d_2)$. This notion was firstly introduced by Burger for the curve $\{(a,\Theta(a))\mid a\in \mathbb R\}$ associated to actions on rank 1 symmetric spaces \cite{Bur93} and has attracted a lot of interest in recent years, see \cite{CT21, CR22, CR23, CR23b} for details.

For a fixed large $L>0$, define the annulus-like set of orbital points $A_i(o_i,n,L):=\{g\in G: |d_i(o_i,go_i)-n|\le L\}$ for $i=1,2$. 
We then obtain
\begin{equation}\label{Expression of theta(a)}
    \theta(a)=\limsup_{n\to \infty}\frac{1}{n}\log \sum_{g\in A_1(o_1,n,L)}e^{-ad_2(o_2,go_2)}.
\end{equation}

Similarly, define the annulus-like set of conjugacy classes  $A_i^c(n,L):=\{[g]\in \mathbf{conj}(G): |\ell_i(g)-n|\le L\}$ for $i=1,2$. 
We have 
\begin{equation}\label{Expression of Theta(a)}
    \Theta(a)=\limsup_{n\to \infty}\frac{1}{n}\log \sum_{[g]\in A_1^c(n,L)}e^{-a\ell_2(g)}.
\end{equation}

As shown in \cite[Proposition 3.1]{CT21} for a hyperbolic group, we explore the relation between $\theta(a)$ and $\Theta(a)$.
\begin{proposition}\label{Prop: ConvergenceRadius}
For any $a\in \mathbb R$, we have   $\theta(a)\le \Theta(a)$.
\end{proposition}

We first prove several lemmas that will be useful in the proof of \autoref{Prop: ConvergenceRadius}.

\begin{lemma}[{\cite[Lemma 7.2]{CK02}} ]\label{Lem: ProperAction}
    Let $G\curvearrowright X$ be a proper action on a geodesic metric space. Then for any $o\in X$ and $n\gg M>0$, there exists $C=C(o,M)>0$ with the following property. For any geodesic segment $\gamma\subseteq X$ with $|\mathrm{Len}(\gamma)-n|\le M$, $\sharp\{g\in G\mid d(go,\gamma)\le M\}\le Cn$.
\end{lemma}

\begin{lemma}\label{Lem: FiniteColor}
    Let $G\curvearrowright X$ be a proper action on a proper geodesic metric space. For any $o\in X$, and any $R>0$, there exists $N=N(o,R)\in \mathbb N ^+$ such that $Go$ can be divided into $N$ pieces so that each piece is $R$-separated.
\end{lemma}
\begin{proof}
    As $G\curvearrowright X$ is proper, the set $\{g\in G\mid d(o,go)\le R\}$ is finite. Then $N:=\sharp(Go\cap B(o,R))$ is finite. The isometric group action guarantees that $N=\sharp(Go\cap B(go,R))$ for any $g\in G$.

    Denote $Go\cap B(o,R)=\{g_1o,\cdots,g_No\}$. Let $C=\{c_1,\ldots,c_N\}$ be a set consisting of $N$ distinct colors and $\phi: Go\cap B(o,R)\to C$ be a coloring map by assigning each $g_io$ the color $c_i$. Then we extend the coloring map to the whole orbit. Precisely speaking, let $A_0=Go\cap B(o,R)$. For any orbit point $go\notin A_0$, the set $Go\cap B(go,R)$ contains at most $N-1$ distinct colors as $\sharp((Go\cap B(go,R))\setminus go)\le N-1$. Then we assign $go$ one of the remaining colors in $C$ and set $A_1=A_0\cup\{go\}$. By repeatedly coloring point-by-point, we extend the coloring map to $A_1,A_2,\ldots, A_n$. As $X$ is proper, so $G$ is countable and one can require that the diameter of $A_n$ tends to infinity and $Go=\cup_{n=1}^{\infty}A_n$. 

    Finally, we divide $Go$ into $N$ pieces such that each piece maps to a same color. Clearly, any two orbit points with the same color are at least $R$-separated.
\end{proof}

Fix a subset $F=\{f_1,f_2,\cdots, f_{15}\}\subseteq \mathcal{SC}(G)$ given by \autoref{COR: Simultaneous Extension 15}. Let $\phi$ be the \textit{extension map} sending $g$ to $gf$ where $f\in F$ is determined by \autoref{COR: Simultaneous Extension 15} satisfying $|\ell_i(\phi(g))-d_i(o_i,\phi(g)o_i)|\le 2\epsilon$ for $i=1,2$. 
\begin{corollary}\label{Lem: Injectivity}
    The group $G$ can be divided into a finite collection of subsets $G=\bigsqcup_{1\le i\le N}A_i$ such that $\phi|_{A_i}$ is injective for each $1\le i \le N$.
\end{corollary}
\begin{proof}
    Let $R=2\max \left\{d_1(o_1,fo_1)\mid f\in F\right\}+1$. According to \autoref{Lem: FiniteColor}, we can divide $G$ into a finite collection of subsets $G=\bigsqcup_{1\le i\le N}A_i$, so that $A_io_1\subseteq X_1$ is $R$-separated for each $1\le i\le N$. For any two elements $g,g'\in A_i$, denote $\phi(g)=gf$ and $\phi(g')=g'f'$ where $f,f'\in F$. Then by triangle inequality, $d_1(gfo_1,g'f'o_1)\ge d_1(go_1,g'o_1)-d_1(fo_1,o_1)-d_1(f'o_1,o_1)>0$. Thus $gf\neq g'f'$, i.e. $\phi|_{A_i}$ is injective.
\end{proof}

Now, we can prove \autoref{Prop: ConvergenceRadius}.

\begin{proof}[Proof of \autoref{Prop: ConvergenceRadius}]
 \autoref{Lem: Injectivity} shows that we can divide the annulus-like set of orbital points into $N$ subsets $A_1(o_1,n,L)=\bigsqcup_{1\le i\le N}A_i$, such that $\phi|_{A_i}$ is injective for each $1\le i \le N$. Let $\epsilon $ be the constant given by \autoref{COR: Simultaneous Extension 15}.

     For any $h\in A_1(o_1,n,L)$, one has that $|\ell_j(\phi(h))-d_j(o_j,\phi(h)o_j)|\le 2\epsilon$ for $j=1,2$. For any $a\in \mathbb{R}$, let $K=\max\{d_j(o_j,fo_j)\mid f\in F,\,j=1,2\}$, and $C_0=e^{\abs{a}(K+2\epsilon)}$. Then we have $[\phi(h)]\in A_1^c(n,L+K+2\epsilon)$, and $e^{-ad_2(o_2,ho_2)}\le C_0e^{-a\ell_2(\phi(h)}$.

    Now, fix any $i\in \{1,\ldots,N\}$. As $\Theta(a)$ takes sum over all conjugacy classes, we next show that there are not ``too many'' elements in $A_i$ mapped into a same conjugacy class by $\phi $. For each $h\in A_i$, the path $\gamma_h$ labelled by $hf$ in $X_1$ is actually a special $(D,\tau)$-admissible path (cf. \autoref{LEM: Modified Extension Lemma}) and also a $c$-quasi-geodesic, where $c$ depends only on $\tau$. Suppose that $h'\in A_i$ such that $[hf]=[h'f']$. Then there is $g\in G$ such that $h'f'=g(hf)g^{-1}$. This shows that $\gamma_{h'}$ and $g\gamma_h$ are both $c$-quasi-geodesic for $h'f'$ in $X_1$. As $hf$ is a contracting element on $X_1$, let $\eta_h$ be a geodesic connecting $(hf)^-$ and $(hf)^+$. Hence, by the $\epsilon$-fellow travel property of a $(D,\tau)$-admissible path, up to translation by $h'f'$, $o\in \gamma_{h'}$ lies in a $\epsilon$-neighbourhood of a geodesic segment $\gamma\subseteq g\eta_h$ with length $|l(\gamma)-n|\le L+K+2\epsilon$. Note that this is equivalent to say $g^{-1}o\in N_{\epsilon}(g^{-1}\gamma)$ and $g^{-1}\gamma$ is a geodesic segment in $\eta_h$, which is independent of $g$. Hence, as a result of \autoref{Lem: ProperAction}, there exists $C_1=C_1(\epsilon)>0$ such that the number of such $g^{-1}$ has a cardinality bounded by $C_1n$.

    Totally speaking, we have $\sum_{h\in A_i}e^{-ad_2(o_2,ho_2)}\le C_0C_1n\sum_{[g]:|\ell_1(g)-n|\le L+K+2\epsilon}e^{-a\ell_2(g)}$.

    Hence,  
    $$\sum_{|d_1(o_1,go_1)-n|\le L}e^{-ad_2(o_2,ho_2)}\le C_0C_1Nn\sum_{[g]:|\ell_1(g)-n|\le L+K+2\epsilon}e^{-a\ell_2(g)}.$$
    It follows from (\ref{Expression of theta(a)}) and (\ref{Expression of Theta(a)}) that $\theta(a)\le \Theta(a)$.
\end{proof}

\begin{proposition}\label{Prop: ConvergenceRadius2}
For any {$a\in (-\infty,0]\cup [\delta_2(G),+\infty)$}, we have    $\theta(a)= \Theta(a)$.
\end{proposition}
\begin{proof}
    As \autoref{Prop: ConvergenceRadius} shows that $\theta(a)\le \Theta(a)$ for every $a\in \mathbb R$, it suffices to show that $\theta(a)\ge \Theta(a)$ for $a\in (-\infty,0]\cup [\delta_2(G),+\infty)$.

     \autoref{Lem: ConvexMetric} gives that $|\ell_1(g)-\ell_1^{o_1}(g)|\le  C'$ for a constant $C'$ independent of $g\in G$. For any $[g]\in \mathbf{conj}(G)$, there exists $h\in [g]$ such that $d_1(o_1,ho_1)=\ell_1^{o_1}(g)\le \ell_1(g)+C'$. In addition, it follows from $d_2(o_2,ho_2)\ge \ell_2(h)=\ell_2(g)$ that $e^{-ad_2(o_2,ho_2)}\ge e^{-a\ell_2(g)}$ for $a\le 0$. Therefore, 
     $$\sum_{|d_1(o_1,ho_1)-n|\le L+C'}e^{-ad_2(o_2,ho_2)}\ge \sum_{[g]:|\ell_1(g)-n|\le L}e^{-a\ell_2(g)},$$
     which by (\ref{Expression of theta(a)}) and (\ref{Expression of Theta(a)}) implies that $\theta(a)\ge \Theta(a)$ when $ a\le 0$.

     Then we switch $(X_1,d_1)$ and $(X_2,d_2)$ to get another two curves $\tilde{\theta}(a)$ and $\tilde{\Theta}(a)$ (which is actually symmetric to the original one). The same arguments show that $\tilde{\theta}(a)= \tilde{\Theta}(a)$ when $a\le 0$, which implies that $\theta^{-1}(b)= \Theta^{-1}(b)$ when $b<0$. This is also equivalent to say $\theta(a)= \Theta(a)$ for $a\ge \delta_2(G)$. Then the conclusion follows.
\end{proof}

For any continuous, decreasing and convex function $f: \mathbb R\to \mathbb R$, both limits $\lim_{t\to +\infty}f(t)/t$ and $\lim_{t\to - \infty}f(t)/t$ exist (may be infinite). As $\theta(a)$ and $\Theta(a)$ are such functions, we can define $$\alpha_{min}:=-\lim_{a\to+\infty}\frac{\theta(a)}{a},\quad \alpha_{max}:=-\lim_{a\to-\infty}\frac{\theta(a)}{a}$$
 and $$\beta_{min}:=-\lim_{a\to+\infty}\frac{\Theta(a)}{a},\quad \beta_{max}:=-\lim_{a\to-\infty}\frac{\Theta(a)}{a}.$$ Then $\alpha_{min}, \alpha_{max}, \beta_{min}, \beta_{max}$ are positive and finite as $d_1$ and $d_2$ are quasi-isometric.

 Denote $$\mathrm{Dil}_-:=\inf_{[g]:\ell_1(g)>0}\frac{\ell_2(g)}{\ell_1(g)},\quad \mathrm{Dil}_+:=\sup_{[g]:\ell_1(g)>0}\frac{\ell_2(g)}{\ell_1(g)}.$$

 \begin{corollary}\label{Cor: DilationInequality}
     $\mathrm{Dil}_-=\beta_{min}= \alpha_{min}\le \alpha_{max}= \beta_{max}=\mathrm{Dil}_+$.
 \end{corollary}
 \begin{proof}
     As the Manhattan curve $\theta(a)$ is decreasing and convex, we obtain $\alpha_{min}\le \alpha_{max}$. 
     By \autoref{Prop: ConvergenceRadius2}, the functions $\theta$ and $\Theta$ agree on $(-\infty,0]\cup [\delta_2(G),+\infty)$, so we have $\alpha_{min}=\beta_{min}$ and $\beta_{max}=\alpha_{max}$. 
     Hence, up to symmetry, it suffices to show $\mathrm{Dil}_-=\beta_{min}$.

     By the definition of $\mathrm{Dil}_-$, we have $\mathrm{Dil}_-\cdot \ell_1(g)\le \ell_2(g)$ for any $g\in G$. If   $|\ell_1(g)-n|\le L$, then $e^{-a\ell_2(g)}\le e^{-a\mathrm{Dil}_-\ell_1(g)}\le e^{-a\mathrm{Dil}_-(n-L)}$ for any $a>0$. 
     
     By \autoref{Lem: ConvexMetric}, we have $|\ell_1(g)-\ell_1^{o_1}(g)|\le C'$ for a constant $C'$ independent of $g\in G$.   Recall that by \cite[Main Theorem]{GY22}, the set $\{[g]: |\ell_1^o(g)-n|\le L\}$ has a cardinality at most $C e^{\delta_1(G)n}/n$ for a constant $C$ independent of $n$. Hence,
$$\sum_{[g]:|\ell_1(g)-n|\le L+C'}e^{-a\ell_2(g)}\prec e^{-a\mathrm{Dil}_-(n-L-C')}  e^{\delta_1(G)n}/n, $$     
     Recall that (\ref{Expression of Theta(a)}) gives 
$$
\Theta(a)=\limsup_{n\to \infty}\frac{1}{n}\log \sum_{[g]:|\ell_1(g)-n|\le L+C'}e^{-a\ell_2(g)}.
$$     
 We thus obtain    that $\Theta(a)\le -a\mathrm{Dil}_-+\delta_1(G)$,   implying $ \beta_{min}\ge \mathrm{Dil}_-$. 

    For the converse direction, by the definition of $\mathrm{Dil}_-$, for any small $\epsilon>0$, there exists an infinite-order element $g_0\in G$ such that 
     $$\ell_2(g_0)\le \ell_1(g_0)(\mathrm{Dil}_-+\epsilon).$$
    As the stable length is a homogeneous function on any cyclic subgroup, we have $\ell_1(g_0^n)=n\ell_1(g_0)$ and $\ell_2(g_0^n)\le \ell_1(g_0^n)(\mathrm{Dil}_-+\epsilon)$ for any $n\ge 1$. Hence, for any $n\ge 1$ and $a\ge 0$, 
    $$\sum_{[g]:|\ell_1(g)-n \ell_1(g_0)|\le L}e^{-a\ell_2(g)}\ge e^{-a\ell_2(g_0^n)}\ge e^{-a\ell_1(g_0^n)(\mathrm{Dil}_-+\epsilon)}\ge e^{-a\cdot n\ell_1(g_0)\cdot(\mathrm{Dil}_-+\epsilon)}.$$
    Taking the limit sup gives  $\Theta(a)\ge -a(\mathrm{Dil}_-+\epsilon)$ when $a\ge 0$. By letting $\epsilon\to 0$, one has $\beta_{min}\le \mathrm{Dil}_-$.

     In conclusion, we have $\mathrm{Dil}_-=\beta_{min}$, which completes the proof. 
\end{proof}

\begin{proposition}\label{Prop: RoughlySimilarEquivalence}
   The following are equivalent:
    \begin{enumerate}
        \item[(1)] the curve $\{(a,\Theta(a))\mid a\in \mathbb R\}$ for $d_1,d_2$ is a straight line;
        \item[(2)] $\ell_2(g)=k\ell_1(g)$ for all conjugacy classes $[g]$, where $k=\delta_1(G)/\delta_2(G)$;
        \item[(3)] the Manhattan curve $\{(a,\theta(a))\mid a\in \mathbb R\}$ for $d_1,d_2$ is a straight line.
    \end{enumerate}
\end{proposition}
\begin{proof}
    $(1)\Rightarrow (2)$: when $\Theta(a)$ is a straight line, there exists $k>0$ such that  $\beta_{min}=\beta_{max}=k$. As $\Theta(0)=\delta_1(G)$ and $\Theta(\delta_2(G))=0$, the absolute value of the slope of $\Theta(a)$ is $k=\delta_1(G)/\delta_2(G)$. Then the conclusion follows from \autoref{Cor: DilationInequality}.

    $(2)\Rightarrow (1)$: this follows from the definition of $\Theta(a)$.

    $(2)\Rightarrow (3)$: this follows from \autoref{COR: Rigidity from SC} and the definition of $\theta(a)$. 

    $(3)\Rightarrow (2)$: this follows from \autoref{Cor: DilationInequality}.
\end{proof}

\subsection{MLS rigidity for subsets with maximal growth}

Now we take $d_1,d_2$ as the pull-back pseudo-metrics on $G$: $d_i(g,h):=d_i(go_i,ho_i)$ for any $g,h\in G$. For any $a,b>0$, denote $d_*:=ad_2+bd_1$ and let $\ell_*$ and $\delta_*(G)$ (resp. $\delta_*^c(G)$) be the corresponding stable length and critical exponent (resp. growth rate of conjugacy classes) of $G$ on $(G, d_*)$. (see \autoref{Definition of growth rate of subset} and \autoref{Definition of conjugacy growth rate of subset} for their definitions). 

\begin{lemma}\label{Lem: LinearCombination}
    For any $a,b>0$, $d_*=ad_2+bd_1$ satisfies that $\ell_*=a\ell_2+b\ell_1$. Moreover,  $\delta_*^c(G)=1$ when $b=\Theta(a)$ and $\delta_*(G)=1$ when $b=\theta(a)$. 
\end{lemma}
\begin{proof}
    For any $g\in G$, $$\ell_*(g)=\lim_{n\to \infty}\frac{d_*(1,g^n)}{n}=\lim_{n\to \infty}\frac{ad_2(1,g^n)+bd_1(1,g^n)}{n}=a\ell_2(g)+b\ell_1(g).$$

    Note that $\delta_*^c(G)$ is also the convergence abscissa for $s$ of the following Poincare series $$\mathcal P(s)=\sum_{[g]\in \mathbf{conj}(G)}e^{-s\ell_*(g)}=\sum_{[g]\in \mathbf{conj}(G)}e^{-s(a\ell_2(g)+b\ell_1(g))}.$$ 
    When $b=\Theta(a)$, it follows from the definition of $\Theta(a)$ that $\delta_*^c(G)=1$. Similarly, $\delta_*(G)$ is the convergence abscissa of the following Poincare series $$\mathcal P(s)=\sum_{g\in G}e^{-sd_*(1,g)}=\sum_{g\in G}e^{-s(ad_2(1,g)+bd_1(1,g))}.$$  
    When $b=\theta(a)$, it follows from the definition of $\theta(a)$ that $\delta_*(G)=1$.
\end{proof}

A subset $E\subseteq G$ is called \textit{exponentially generic} in $d_i$ for given $i\in \{1,2\}$ if $\delta_i(G\setminus E)<\delta_i(G)$.

\begin{theorem}\label{Thm: MLSR From Large Subset}
    Suppose that $\{(a,\theta(a))\mid a\in \mathbb R\}$ is {either globally strictly convex or a straight line}. Assume    $\delta_1(G)=\delta_2(G)$ for simplicity. Let $F\subseteq \mathcal{SC}(G)$ be a finite subset determined by \autoref{COR: Simultaneous Extension 15}. Then 
    \begin{enumerate}[(1)]
        \item for any subset $E\subseteq G$ with $\delta_1(E)=\delta_1(G)$, there exists $f\in F$ such that $G$ has MLS rigidity from $Ef$. 
        \item for any exponentially generic subset $E\subseteq G$ in $d_1$ or $d_2$, $G$ has MLS rigidity from $E$. 
    \end{enumerate}
\end{theorem}
\begin{proof}
    (1) Denote $\delta_1(G)=\delta_2(G)=\delta$. For each $g\in E$, \autoref{COR: Simultaneous Extension 15} gives an $f\in F$ so that $|\ell_i(gf)-d_i(o_i,gfo_i)|\le 2\epsilon$ for $i=1,2$. As $F$ is finite, there exists $f_0\in F$ and $E_0\subseteq E$ with $\delta_1(E_0)=\delta_1(E)$ such that each element $g'\in E_0f_0$ satisfies $|\ell_i(g')-d_i(o_i,g'o_i)|\le 2\epsilon$. We will show that $G$ has MLS rigidity from $Ef_0$. For this purpose, we assume in the following that $\ell_1(g')=\ell_2(g')$ for each $g'\in Ef_0$.

    Suppose by contrary that $G$ does not have MLS rigidity from $Ef_0$, then the orbit map $\rho: Go_1\to Go_2$ is not a rough isometry. With \autoref{THM: Main Rigidity}, this is equivalent to $\mathrm{Dil}_-<\mathrm{Dil}_+$, and thus $\alpha_{min}<\alpha_{max}$ by \autoref{Cor: DilationInequality}. {Hence by assumption, $\{(a,\theta(a))\mid a\in \mathbb R\}$ is globally strictly convex, which means}  that $\{(a,\theta(a))\mid a\in \mathbb R\}$ is not a straight line restricting on $a\in (0,\delta)$.   Then there exists $\xi\in (0,\delta)$ such that $\xi+\theta(\xi)<\delta$. 
    
    Now we take $d_1,d_2$ as the pull-back pseudo-metric on $G$. Denote $\xi'=\xi+\theta(\xi)$ and $d_*:=\xi d_2+\theta(\xi)d_1=\xi' d_1+\xi(d_2-d_1)$. As both actions $G\curvearrowright (G,d_1)$ and $G\curvearrowright (G,d_2)$ have the same MLS on $Ef_0\supset E_0f_0$, for each $g\in E_0$,
    $$d_1(1,g)\sim_K d_1(1,gf_0)\sim_{\epsilon} \ell_1(gf_0)=\ell_2(gf_0)\sim_{\epsilon} d_2(1,gf_0)\sim_K d_2(1,g)$$
    where $K=\max\{d_i(1,f)\mid f\in F, i=1,2\}$. Hence, one has
    $$\{g\in E_0: d_1(1,g)\le n\}\subseteq \{g\in G: d_*(1,g)\le \xi' n+\xi M\}$$
    where $M$ is a constant depending only on $K,\epsilon$.
    As \autoref{Lem: LinearCombination} shows that $\delta_*(G)=1$, it follows from \autoref{Definition of growth rate of subset} that $\delta_1(E_0)\le \xi'\delta_*(G)=\xi'<\delta$, which is a contradiction as $\delta_1(E_0)=\delta_1(E)=\delta_1(G)=\delta$.

    (2) Without loss of generality, we assume that $E$ is exponentially generic in $d_1$. 
    
    Let $\tilde E:=\cap_{f\in F}Ef^{-1}$. As $E$ is exponentially generic and $F$ is finite, the set $\tilde E$ satisfies $\delta_1(\tilde E)=\delta_1(E)=\delta_1(G)$. Therefore, according to (1), there exists $f\in F$ such that $G$ has MLS rigidity from $\tilde{E} f$. As $\tilde{E}f\subset E$, we derive that $G$ has MLS rigidity from $E$.
%
%
%
\end{proof}

\begin{remark}\label{RMK: Straight Line Condition}
    As is shown in the proof of \autoref{Thm: MLSR From Large Subset}, the requirement that $\{(a,\theta(a))\mid a\in \mathbb{R}\}$ is  {either globally strictly convex or a straight line} can be relaxed into the following condition: $\theta(a)$ is a straight line restricting on $(0,\delta_2(G))$ if and only if the entire curve $\{(a,\theta(a))\mid a\in \mathbb{R}\}$ is a straight line.

    This condition serves as the focus of researches in Manhattan curves and has been proven true in several special cases, including possibly non-cocompact Fuchsian representations \cite{Kao20}. 
\end{remark}

\begin{theorem}\label{Thm: MLSR From Conjugacy-large Subset}
    Suppose that $\{(a,\Theta(a))\mid a\in \mathbb R\}$  is {either globally strictly convex or a straight line}. Assume  $\delta_1(G)=\delta_2(G)$ for simplicity. Let $F\subseteq \mathcal{SC}(G)$ be a finite subset determined by \autoref{COR: Simultaneous Extension 15}. Then 
    \begin{itemize}
        \item[(1)] for any subset $E\subseteq G$ with $\delta_1^c(E)=\delta_1^c(G)$, $G$ has MLS rigidity from $E$.
        \item[(2)] for any subset $E\subseteq G$ with $\delta_1(E)=\delta_1(G)$, there exists $f\in F$ such that $G$ has MLS rigidity from $Ef$. 
        \item[(3)] for any exponentially generic subset $E\subseteq G$ in $d_1$ or $d_2$, $G$ has MLS rigidity from $E$. 
    \end{itemize}
    
\end{theorem}
\begin{proof}
We first show that the condition stated in \autoref{RMK: Straight Line Condition} holds when $\{(a,\Theta(a))\mid a\in \mathbb R\}$  is {either globally strictly convex or a straight line}. If $\theta(a)$ is a straight line restricted on $(0,\delta_2(G))$, as $\theta(a)\le \Theta(a)$ by \autoref{Prop: ConvergenceRadius} and $\Theta(a)$ is convex, it follows that $\Theta(a)$ is also a straight line restricted on $(0,\delta_2(G))$. {By assumption, }  the entire curve $\{(a,\Theta(a))\mid a\in \mathbb R\}$ is a straight line. It follows from \autoref{Prop: RoughlySimilarEquivalence} that $\{(a,\theta(a))\mid a\in \mathbb R\}$ is also a straight line. 
Hence, the items (2) and (3) are immediate corollaries of \autoref{Thm: MLSR From Large Subset}. It thus remains  for us to show the item (1).
    
    Arguing by contradiction, assume that the orbit map $\rho: Go_1\to Go_2$ given by $\rho(go_1)=go_2$ is not a rough isometry. With \autoref{THM: Main Rigidity}, this is equivalent to $\mathrm{Dil_-}<\mathrm{Dil}_+$, and thus $\beta_{min}<\beta_{max}$ by \autoref{Cor: DilationInequality}. {Global strict convexity} of the curve  $\{(a,\Theta(a))\mid a\in \mathbb R\}$ shows that it is not a straight segment on $a\in (0,\delta_2(G))$. Denote $\delta_1(G)=\delta_2(G)=\delta$. Then there exists $\xi\in (0,\delta)$ such that $\xi+\Theta(\xi)<\delta$.

    Analogous to the proof of \autoref{Thm: MLSR From Large Subset}, we denote $\xi'=\xi+\Theta(\xi)$ and $d_*:=\xi d_2+\Theta(\xi)d_1$. Notice that $d_*=\xi' d_1+\xi(d_2-d_1)$ and then $\ell_*=\xi'\ell_1+\xi(\ell_2-\ell_1)$. As both actions $G\curvearrowright (G,d_1)$ and $G\curvearrowright (G,d_2)$ have the same MLS on $E$, $$\{[g]: g\in E, \ell_1(g)\le n\}\subseteq \{[g]: \ell_*(g)\le \xi'n\}$$ 
    for any $n>0$. As \autoref{Lem: LinearCombination} shows that $\delta_*^c(G)=1$, it follows from \autoref{Definition of conjugacy growth rate of subset} that $\delta_1^c(E)\le \xi' \delta_*^c(G)=\xi'<\delta$, which is a contradiction as $\delta_1^c(E)=\delta_1^c(G)=\delta$.
\end{proof}

\begin{remark}When $G$ is a relatively hyperbolic group, \autoref{MLSFailsMaximalSubgroup} provides an example in which the element $f$ could not be removed in the item (2). As in hyperbolic groups, the result of rigidity on subset \cite[Theorem 1.1]{CR23} by Cantrell-Reyes are exactly item (1), where their growth rate identifies with our definition of conjugacy growth rate; see \autoref{Definition of conjugacy growth rate of subset} for its definition.    
\end{remark}

\bibliographystyle{amsplain}   
\bibliography{MLSRC.bib}
\end{sloppypar}
\end{document}